\documentclass[12pt,fleqn]{amsart}
\usepackage{amsmath,amssymb,amsfonts,amsthm}
\usepackage{appendix}
\usepackage{mathtools}
\usepackage[hidelinks]{hyperref}
\hypersetup{colorlinks=true,linkcolor=blue, citecolor=blue, urlcolor=blue}
\usepackage{tikz}
\usepackage{tikz-cd}
\usepackage{enumerate}
\usepackage{graphicx, array}
\usepackage{float}
\usepackage{xfrac} 
\usepackage{arydshln}
\usepackage{dsfont}
\usepackage{comment}

\usepackage{graphicx}
\graphicspath{{figures/}}
\usepackage[
labelfont=small,
labelformat=parens,
subrefformat=parens
]{subcaption}

\theoremstyle{plain}
\newtheorem{theorem}{Theorem}[section]
\newtheorem{lemma}[theorem]{Lemma}
\newtheorem{corollary}[theorem]{Corollary}
\newtheorem{proposition}[theorem]{Proposition}
\theoremstyle{definition}
\newtheorem{definition}[theorem]{Definition}

\theoremstyle{remark}
\newtheorem{remark}[theorem]{Remark}

\usepackage[T5,T1]{fontenc}

\usepackage[
backend=biber,
style=alphabetic,
sorting=nyt
]{biblatex}
\DeclareFieldFormat*{title}{\textit{#1}}
\DeclareFieldFormat{journaltitle}{\textnormal{#1}}
\DeclareListFormat{publisher}{\textnormal{#1}}

\usepackage[left=2.5cm, right=2.5cm, top=3cm,bottom=3cm]{geometry}

\newcommand{\N}{\mathbb{N}}

\newcommand{\R}{\mathbb{R}}

\DeclareMathOperator{\im}{im}
\DeclareMathOperator{\ad}{ad}
\DeclareMathOperator{\spenn}{span}

\newcommand{\abs}[1]{\left\vert{#1}\right\vert}
\newcommand{\inner}[1]{\left\langle{#1}\right\rangle}

\title{Hamiltonian Hopf Bifurcations in Gaudin Models}

\author{Tobias V\r{a}ge Henriksen}
\address{Tobias V\r{a}ge Henriksen: 
  Bernoulli Institute for Mathematics,
  Computer Science and Artificial Intelligence\\
  University of Groningen\\ 
  P.O. Box 407, 9700 AK Groningen, The Netherlands,
  \textnormal{and}
  Department of Mathematics\\ 
  University of Antwerp\\ 
  Middelheimlaan 1, B-2020 Antwerp, Belgium
}
\email{t.v.henriksen@rug.nl}

\begin{document}

\date{\today}

\begin{abstract}
We show that $\textup{su}(2)$ rational and trigonometric Gaudin models, or in other words, generalised coupled angular momenta systems, have singularities that undergo Hamiltonian Hopf bifurcations. In particular, we find a normal form for the Hamiltonian Hopf bifurcation up to sixth order, letting us determine when the bifurcation is degenerate or not. Furthermore, in the non-degenerate case we may use the fourth order terms to determine whether the bifurcation is supercritical or subcritical; whether a flap appears in the image of the momentum map or not. Finally, figures illustrating some of the bifurcations taking place in $\textup{su}(2)$ Gaudin models are presented, showing that there are more bifurcations occurring than only Hamiltonian Hopf ones.
\end{abstract}
\maketitle
\section{Introduction}

Integrable systems are triples $(\mathcal{M},\omega,F=(f_{1},\dots,f_{n}))$, where $(\mathcal{M},\omega)$ is a $2n$-dimensional symplectic manifold, and $F : \mathcal{M} \to \R^{n}$ is a smooth mapping called the \emph{momentum map}. The components of $F$ Poisson commute with each other, i.e.\ $\{f_{i},f_{j}\} = 0$ for all $i,j \in \{1,\dots,n\}$, and are functionally independent almost everywhere, i.e.\ their gradients are linearly independent almost everywhere. Integrable systems play an important role both in mathematics and in physics, examples being the Kepler problem decribing planetary motion in celestial mechanics, and the Jaynes-Cummings model describing atoms interacting with an electromagnetic field in quantum mechanics. A third example, which is studied in this paper, are the Gaudin models, introduced by Gaudin in 1976 \cite{Gaudin1976}. The Gaudin models are examples of classical and quantum spin chains (see for instance Arutyunov \cite[Section 5.1.3]{Arutyunov2019}); we will consider the former. In particular, we study certain bifurcations taking place in Gaudin models.

In this paper, let $\mathcal{M} = \mathbb{S}^{2} \times \mathbb{S}^{2}$, and we endow each sphere with Cartesian coordinates $(x_{i},y_{i},z_{i})$ such that $x_{i}^{2} + y_{i}^{2} + z_{i}^{2} = 1$, for $i \in \{1,2\}$. Furthermore, we take the symplectic form to be $\omega = R_{1} \omega_{\mathbb{S}^{2}} \otimes R_{2} \omega_{\mathbb{S}^{2}}$, $R_{1},R_{2} \in \R_{>0}$ being the weight of the respective spheres, and $\omega_{\mathbb{S}^{2}}$ being the symplectic form on the $2$-sphere $\mathbb{S}^{2}$. 

Let $\mathbf{t} = (t_{0},t_{1},t_{2},t_{3},t_{4}) \in \R^{5}$ be a $5$-tuple of parameters, and $w \in \R$ another parameter, which we for simplicity usually take to be either $0$ or $1$ in specific examples in Section \ref{sec:momentum_map}. The components of the momentum map $F = (J,H_{w,\mathbf{t}})$ are
\begin{align} \label{eq:Gaudin-system-intro}
\begin{cases}
J(x_{1},y_{1},z_{1},x_{2},y_{2},z_{2}) = R_{1}z_{1} + R_{2}z_{2}, \\
H_{w,\mathbf{t}}(x_{1},y_{1},z_{1},x_{2},y_{2},z_{2}) = t_{0}(z_{1} + z_{2})^{2} + w(t_{1}z_{1} + t_{2}z_{2}) + t_{3}(x_{1}x_{2} + y_{1}y_{2}) + t_{4}z_{1}z_{2}.
\end{cases}
\end{align}
In particular, if $t_{0} = 0$, then $(\mathcal{M},\omega,(J,H_{w,\mathbf{t}}))$ is said to define a $\textup{su}(2)$ \emph{rational Gaudin model} (see Section \ref{sec:Gaudin-models}). (The prefix $\textup{su}(2)$ is the Lie algebra of the special unitary group of $2 \times 2$ matrices, which will be dropped after the introduction.) In Section \ref{sec:momentum_map}, we fix $w = 1$ in this case. If, on the other hand, $w = 0$, and $t_{0}$ is arbitrary, then $(\mathcal{M},\omega,(J,H_{w,\mathbf{t}}))$ is said to define a $\textup{su}(2)$ \emph{trigonometric Gaudin model}.

The most interesting points for the momentum map are the points for which the components of the momentum map are not functionally independent. These are the singularities of the system. Let $DF$ be the Jacobian of $F$. The rank of the singularity $z_{0} \in \mathcal{M}$ is the rank of $DF|_{z_{0}}$. There are four rank $0$ singularities in the system defined by \eqref{eq:Gaudin-system-intro}: 
\begin{align*}
\begin{split}
\begin{cases}
m_{0} := (0,0,1,0,0,-1), \\
m_{1} := (0,0,-1,0,0,-1), \\
m_{2} := (0,0,-1,0,0,1), \\
m_{3} := (0,0,1,0,0,1).
\end{cases}
\end{split}
\end{align*}
The rank $0$ singularities are said to go through a \emph{Hamiltonian Hopf bifurcation} if the eigenvalues of the linearised Hamiltonian vector field go from being purely imaginary, for which the singularity is called \emph{elliptic-elliptic}, to lying in the complex plane, with non-zero real part, for which the singularity is called \emph{focus-focus}. In fact, one also requires that the eigenvalues split off the imaginary axis transversally, see Section \ref{sec:HHB-theory}.

The system defined by \eqref{eq:Gaudin-system-intro} also possess rank $1$ singularities. In the image of the momentum map, these are represented as curves, which one can see in Figure \ref{fig:cam}. In this figure, we have fixed $w = 1$ and $t_{0} = t_{2} = 0$, and vary the remaining parameters $t_{1}$, $t_{3}$, and $t_{4}$. In Figures \ref{fig:cam_1}, \ref{fig:cam_2}, and \ref{fig:cam_3}, all curves are of \emph{elliptic-regular} type, which are singularities whose linearised Hamiltonian vector field has one pair of conjugate purely imaginary eigenvalues. In Figure \ref{fig:cam_4}, however, a new set of curves appears in the interior of the original curves (those connecting $m_{1}$, $m_{0}$, and $m_{3}$). The two curves adjacent to $m_{2}$ are of elliptic-regular type, whilst the curve opposite $m_{2}$ is of \emph{hyperbolic-regular} type. Hyperbolic-regular singularities are characterised by a pair of real eigenvalues $\pm a \in \R$. Note also that the endpoints of the hyperbolic-regular curve are cusps (see Section \ref{sec:HHB-theory}). In fact, what we are looking at is a projection of a domain embedded in three dimensions. The curves connected to $m_{2}$ define a section of a different sheet than the sheet $m_{1}$, $m_{0}$, and $m_{3}$ are located on. The two sheets are connected by the hyperbolic-regular line. We call the sheet on which $m_{2}$ sits a \emph{flap}.

The story told by Figure \ref{fig:cam} is a story of four Hamiltonian Hopf bifurcations. In Figure \ref{fig:cam_1}, all rank $0$ singularities $m_{0},m_{1},m_{2},m_{3}$ are of elliptic-elliptic type. Between Figures \ref{fig:cam_1} and \ref{fig:cam_2}, $m_{0}$ undergoes a \emph{supercritical} Hamiltonian Hopf bifurcation (see Section \ref{sec:HHB-theory}), and between Figures \ref{fig:cam_2} and \ref{fig:cam_3}, $m_{2}$ undergoes a similar bifurcation. Finally, between Figures \ref{fig:cam_3} and \ref{fig:cam_4}, $m_{0}$ undergoes another supercritical Hamiltonian Hopf bifurcation, whilst $m_{2}$ undergoes a \emph{subcritical} one. In this paper we will understand at what times these bifurcations take place, and under what conditions.

\begin{figure}[tb]
\def\scale{0.8}
\centering
\begin{tabular}{cc}
    \subfloat[$t_{1} = 1$, $t_{3} = 0$, $t_{4} = 0$. \label{fig:cam_1}]{\includegraphics[scale=\scale]{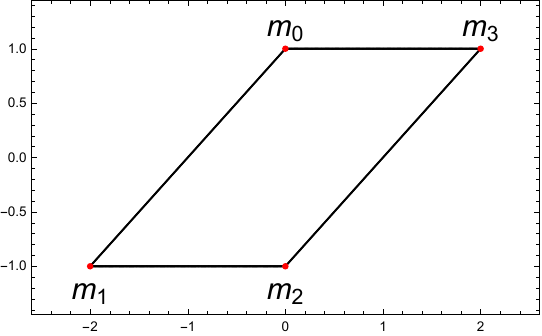}} & 
    \subfloat[$t_{1} = 0.5$, $t_{3} = 0.5$, $t_{4} = 0.5$. \label{fig:cam_2}]{\includegraphics[scale=\scale]{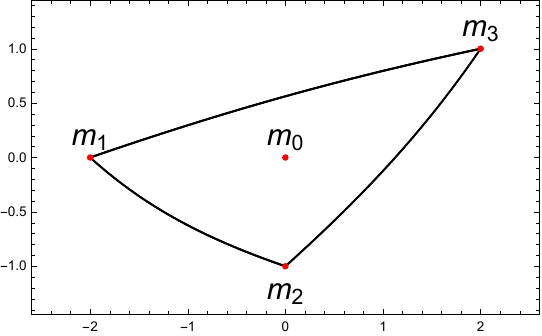}}\\
    \subfloat[$t_{1} = 0.5$, $t_{3} = 0.5$, $t_{4} = 0$. \label{fig:cam_3}]{\includegraphics[scale=\scale]{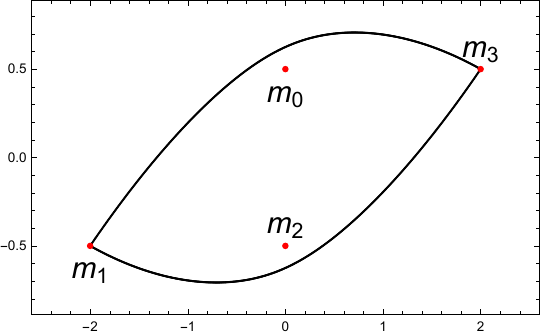}} & 
    \subfloat[$t_{1} = 0.5$, $t_{3} = 0.5$, $t_{4} = -1.5$. \label{fig:cam_4}]{\includegraphics[scale=\scale]{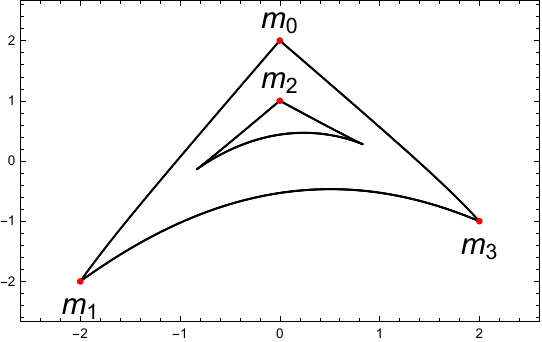}}
\end{tabular}
\caption{The figures shows the image of the momentum map for various choices of the parameters. In all figures, we have $R_{1} = R_{2} = 1$ and $t_{0} = t_{2} = 0$. In the first figure, all rank $0$ singularities are elliptic-elliptic. Then $m_{0}$, and next $m_{2}$, become focus-focus. Finally, all are elliptic-elliptic again, but $m_{2}$ sits on a flap.}
\label{fig:cam}
\end{figure}

There exist various methods to determine whether a singularity undergoes a Hamiltonian Hopf bifurcation. Han{\ss}mann and Van der Meer \cite{Hanssmann2002} used a geometric method to study Hamiltonian Hopf bifurcations in the $3D$ H{\'e}non-Heiles family. Later, Han{\ss}mann and Van der Meer \cite{Hanssmann2003} used singularity theory to study Hamiltonian Hopf bifurcations in integrable systems with $3$ degrees of freedom in more generality. Cushman and Van der Meer \cite{Cushman1990} (see also the book by Cushman and Bates \cite{Cushman1997}) studied Hamiltonian Hopf bifurcations in the Lagrange top by making (the Taylor expansion of) the symplectic structure standard up to the necessary order (see Lemma \ref{lem:j-jet-Hamiltonian}), for then to put the Hamiltonian, at the critical point, in normal form. When we refer to normal forms of Hamiltonians in this text, it is always assumed that the Hamiltonian is evaluated at a critical point of rank $0$.

In this paper we are going to mimic the approach of making the symplectic structure canonical. Then, to put the Hamiltonian into normal form, one studies its Taylor expansion. It is sufficient to study the second order term, $H_{w,\mathbf{t}}^{2}$, in the series to say whether or not this bifurcation takes place. As this order is also sufficient to find the linearised vector field, we say that the Hamiltonian Hopf bifurcation is \emph{linear} if $H_{w,\mathbf{t}}^{2}$ can be put into the following normal form (see also \eqref{eq:HHB_normal_form_linear} and Van der Meer \cite{vanderMeer1985}), where $(q_{1},q_{2},p_{1},p_{2})$ are symplectic coordinates such that the associated symplectic form is standard, $\omega = dq_{1} \wedge dp_{1} + dq_{2} \wedge dp_{2}$,
\begin{equation} \label{eq:linear-normal-form-intro}
\widehat{H}_{w,\mathbf{t}}^{2} = \rho (q_{1}p_{2} - q_{2}p_{1}) + \frac{\sigma}{2}(q_{1}^{2} + q_{2}^{2}),
\end{equation}
for some $\rho \in \R \setminus \{0\}$ and $\sigma = \pm 1$. (The hats in \eqref{eq:linear-normal-form-intro} and in future equations signify that they are in normal form.) The linear Hamiltonian Hopf bifurcation is also called a \emph{Krein collision}, see Marsden \cite{Marsden1992}.

Let us define 
\begin{align*}
&t_{4,m_{0}}^{\pm} := t_{4,m_{0}}^{\pm}(R_{1},R_{2},w,t_{1},t_{2},t_{3}) := \frac{w(t_{1}R_{2} - t_{2}R_{1}) \pm 2t_{3}\sqrt{R_{1}R_{2}}}{R_{1} + R_{2}}, \\
&t_{4,m_{2}}^{\pm} := t_{4,m_{2}}^{\pm}(R_{1},R_{2},w,t_{1},t_{2},t_{3}) := \frac{-w(t_{1}R_{2} - t_{2}R_{1}) \pm 2t_{3}\sqrt{R_{1}R_{2}}}{R_{1} + R_{2}}.
\end{align*}
In Section \ref{sec:linear-HHB} we prove the following two results:

\begin{lemma} \label{lem:rank0type-intro}
Let $t_{3} \neq 0$. The singularities $m_{1}$ and $m_{3}$ are of elliptic-elliptic type for any $(t_{0},t_{1},t_{2},t_{3},t_{4}) \in \R^{5}$. For $k \in \{0,2\}$, the singularity $m_{k}$ is of elliptic-elliptic type if $t_{4} < t_{4,m_{k}}^{-}$ or $t_{4} > t_{4,m_{k}}^{+}$. If either
\begin{itemize}
    \item $R_{1} = R_{2}$, $w \neq 0$, and $t_{1} + t_{2} \neq 0$, or
    \item $R_{1} \neq R_{2}$ and $(R_{1} - R_{2})t_{3} > \abs{w(t_{1} + t_{2})\sqrt{R_{1}R_{2}}}$,
\end{itemize}
then for $t_{4,m_{k}}^{-} < t_{4} < t_{4,m_{k}}^{+}$, $m_{k}$ is a focus-focus point. If neither of these conditions are met, $m_{k}$ is an elliptic-elliptic point for $t_{4,m_{k}}^{-} < t_{4} < t_{4,m_{k}}^{+}$. Finally, for $t_{4} \in \{t_{4,m_{k}}^{-},t_{4,m_{k}}^{+}\}$, $m_{k}$ is degenerate.
\end{lemma}

\begin{theorem}
Let $k \in \{0,2\}$. If the conditions necessary for focus-focus points in Lemma \ref{lem:rank0type-intro} are met, then, for both $t_{4} = t_{4,m_{k}}^{+}$ and $t_{4} = t_{4,m_{k}}^{-}$, the point $m_{k}$ undergoes a linear Hamiltonian Hopf bifurcation.
\end{theorem}

In Section \ref{sec:nonlinear-HHB} we go further, and find the normal form for the Hamiltonian Hopf bifurcation of rank $0$ singularities up to $6$-th order. This allows us to tell whether the bifurcation is degenerate or not. Such a computation was done for the first time by Glebsky and Lerman \cite{Glebsky1995}, where they studied the generalised $1$-D Swift-Hohenberg equation. For this computation, let us introduce the Hilbert generators (see Van der Meer \cite[p.\ 57]{vanderMeer1985}) $S = q_{1}p_{2} - q_{2}p_{1}$, $M = \frac{1}{2}(q_{1}^{2} + q_{2}^{2})$, $N = \frac{1}{2}(p_{1}^{2} + p_{2}^{2})$, and $T = q_{1}p_{1} + q_{2}p_{2}$ (the last one does not appear again in this section, but plays an important role later on). 

\begin{theorem}\label{thm:nonlinear-normal-form-intro}
Let the conditions necessary for focus-focus points in Lemma \ref{lem:rank0type-intro} be met. The normal form of $H_{w,\mathbf{t}}$ at $m_{0}$ up to third order in the Hilbert generators $S$, $M$, and $N$ is
\begin{equation}\label{eq:nonlinear-normal-form-intro}
\widehat{H}_{w,\mathbf{t}} = a_{1}S + N + a_{2}M + a_{3}M^{2} + a_{4}MS + a_{5}S^{2} + a_{6}M^{3} + a_{7}M^{2}S + a_{8}MS^{2} + a_{9}S^{3},
\end{equation}
where the $a_{i}$'s are polynomials in $t_{0},t_{1},t_{2},t_{3},t_{4}$ and $w$. In particular, 
\begin{align*}
a_{2}|_{t_{4} = t_{4,m_{0}}^{+}} = 0 \quad \text{and} \quad \frac{\partial a_{2}}{\partial t_{4}}|_{t_{4} = t_{4,m_{0}}^{+}} \neq 0,
\end{align*}
so $\widehat{H}_{w,\mathbf{t}}$ is a normal form describing a Hamiltonian Hopf bifurcation. Similarly, if $M$ and $N$ change roles, i.e.\ $M = \frac{1}{2}(p_{1}^{2} + p_{2}^{2})$ and $N = \frac{1}{2}(q_{1}^{2} + q_{2}^{2})$, then $a_{2}|_{t_{4} = t_{4,m_{0}}^{-}} = 0$ and $\frac{\partial a_{2}}{\partial t_{4}}|_{t_{4} = t_{4,m_{0}}^{-}} \neq 0$. The same is also true if we change $m_{0}$ with $m_{2}$.
\end{theorem}

This theorem is proven in Section \ref{sec:nonlinear-HHB}. Note that, even though the theorem tells us that we may put the Hamiltonian in the normal form for a Hamiltonian Hopf bifurcation at the parameter values $t_{4,m_{0}}^{\pm}$ and $t_{4,m_{2}}^{\pm}$, the coefficients are not the same in all four cases. The coefficients for $t_{4,m_{0}}^{+}$ are given in Appendix \ref{sec:appendix-thm-coeffs}, up to a scaling, defined in the proof of the theorem. 

The bifurcation is non-degenerate if $a_{3} \neq 0$. Furthermore, if the bifurcation is non-degenerate, we may also use the higher order terms to predict when the bifurcation is supercritical or subcritical. This is determined by the sign of the coefficient $a_{3}$ in the normal form.

Some examples of Hamiltonian Hopf bifurcations in $\textup{su}(2)$ Gaudin models are already well-known. The coupled angular momenta system is a $\textup{su}(2)$ rational Gaudin model with $w = 1$, $t_{0} = 0$, $t_{2} = 0$ and $t_{3} = t_{4} = (1-t_{1})$. It has been shown that, for this system, there exists one singularity which goes through a Hamiltonian Hopf bifurcation (see e.g.\ Sadovskií and Zhilinskií \cite{Sadovskii1999} and Le Floch and Pelayo \cite{LeFlochPelayoCAM}). In Figure \ref{fig:cam}, the bifurcation that happens in the coupled angular momenta system is illustrated in Figures \ref{fig:cam_1} and \ref{fig:cam_2}. Hohloch and Palmer \cite{HohlochPalmer2018} generalised the coupled angular momenta system. Their new system is given by \eqref{eq:Gaudin-system-intro}, where one fixes $w = 1$ and $t_{0} = 0$, for which they showed that two singularities may go through Hamiltonian Hopf bifurcations. Hence, they could find a system possessing two focus-focus points, as in Figure \ref{fig:cam_3}. In fact, one can also show that their system may undergo bifurcations producing flaps, as in Figure \ref{fig:cam_4}. In this article we find conditions for when the various types of Hamiltonian Hopf bifurcations occur for an even more generalised family of coupled angular momenta systems, namely the one defined by \eqref{eq:Gaudin-system-intro}.

A number of other systems undergoing Hamiltonian Hopf bifurcations can be found in van der Meer \cite[Section 5.4]{vanderMeer2017}. One of them, the Lagrange top, has been covered extensively by Cushman and Bates \cite[Chapter V]{Cushman1997}, and inspired several of the proofs given in the following sections of the present paper.

\subsection*{Overview}

The goal of this article is to compute the normal form up to $6$-th order for a generalised version of the $\textup{su}(2)$ rational and trigonometric Gaudin models, which we introduce in Section \ref{sec:Gaudin-models}, and to analyse some of the dynamics and geometry conveyed by the corresponding momentum map. In Section \ref{sec:HHB-theory}, we recall defining properties for the Hamiltonian Hopf bifurcation, as well as the normal form for Hamiltonians undergoing such a bifurcation. In Section \ref{sec:linear-HHB}, we show that the Hamiltonian system defined in \eqref{eq:Gaudin-system-intro} has two points both going through two Hamiltonian Hopf bifurcations, and, in Section \ref{sec:nonlinear-HHB}, we compute the normal form for this Hamiltonian up to $6$-th order. Finally, in Section \ref{sec:momentum_map}, we investigate the image of the momentum map at certain instances; in particular we look at how the coefficients in the normal form influences its shape. As the coefficients appearing in the normal form \eqref{eq:nonlinear-normal-form-intro} are very large, they are presented in Appendix \ref{sec:appendix-thm-coeffs}. In Appendix \ref{sec:appendix-EF-coeffs}, certain coefficients computed in the proof of Theorem \ref{thm:nonlinear-normal-form-intro} are presented.

\subsection*{Acknowledgments}

The author is very grateful to Heinz Han{\ss}mann, Sonja Hohloch, and Nikolay Martynchuk for many useful comments and suggestions which helped to improve this paper. The author was fully supported by the \textit{Double Doctorate Funding} of the Faculty of Science and Engineering of the University of Groningen, and paritally supported by the \textit{FNRS-FWO Excellence of Science (EoS) project `Symplectic Techniques in Differential Geometry' G0H4518N}.
\section{A brief study of Gaudin models} \label{sec:Gaudin-models}

The Gaudin model was introduced by Gaudin \cite{Gaudin1976} as a spin model related to the Lie algebra $\textup{sl}(2)$, and later generalised to be related to any semi-simple complex Lie algebra, see Petrera \cite{Petrera2007} and references therein. Gaudin models are integrable systems, to which one can associate a Lax matrix. In particular, the Lax matrices that we consider depend on a parameter $\lambda$, called the \emph{spectral parameter}. Petrera covered three different dependencies on $\lambda$ (as well as their so-called \emph{Leibniz extension}): rational, trigonometric, and elliptic. In this paper we are going to study Hamiltonian Hopf bifurcations (see Section \ref{sec:HHB-theory}) on a generalised version of the rational and trigonometric cases related to the Lie algebra $\textup{su}(2)$. We will always assume this Lie algebra, and henceforth not write it explicitly. This section shows why the system defined in \eqref{eq:Gaudin-system-intro} is indeed a rational or trigonometric Gaudin model depending on certain choices of the parameters $w$ and $\mathbf{t}$.

Let us briefly recall the notion of Lax matrices. We refer to Babelon, Bernard and Talon \cite{Babelon2003} for a more extensive discussion. A \emph{Lax pair} is a pair of time-dependent matrices $L = L(t)$, $K = K(t)$, where $L$ is called a \emph{Lax matrix}, and $K$ an \emph{auxiliary matrix}. Let $\dot{L}$ denote the time derivative of $L$. Then the Lax pair allows us to write the Hamiltonian equations as
\begin{equation*}
\dot{L} = KL - LK.
\end{equation*}
One can show that the spectrum of $L$ is invariant with respect to time; it is \emph{isospectral}. Thus, for $n \in \N$, $\textup{Tr}(L^{n})$, where $\textup{Tr}$ denotes the trace, are conserved quantities. Furthermore, the auxiliary matrix can be written as a function of the Lax matrix, i.e.\ we may write $K = R(f(L))$, where $R : \mathfrak{g} \to \mathfrak{g}$, $\mathfrak{g}$ being some Lie algebra, is a linear operator, and $f : \mathfrak{g} \to \mathfrak{g}$ is an $\textup{Ad}$-covariant function (see Petrera \cite[Section 1.1]{Petrera2007}).

From now on, we let the Lax pair depend on the spectral parameter $\lambda$. Let $\{\lambda_{q}\}_{q \in Q}$, where $Q$ is a set of indices, be the set of poles of $L(\lambda)$ and $K(\lambda)$, i.e.\ $\lambda_{q}$ is such that $L(\lambda_{q}) = \pm \infty$ and/or $K(\lambda_{q}) = \pm \infty$, for all $q \in Q$. Assuming there is no pole at infinity, we may write
\begin{align*}
L(\lambda) = L_{0} + \sum_{q \in Q}\sum_{r=-n_{q}}^{-1} L_{q,r} (\lambda - \lambda_{q})^{r}, \quad
K(\lambda) = K_{0} + \sum_{q \in Q}\sum_{r=-m_{q}}^{-1} K_{q,r} (\lambda - \lambda_{q})^{r},
\end{align*}
where $L_{0}$, $L_{q,r}$, $K_{0}$, and $K_{q,r}$ are matrices, and $n_{q}$ and $m_{q}$ are the order of the pole $\lambda_{q}$ for $L(\lambda)$ and $K(\lambda)$, respectively (see Babelon, Bernard and Talon \cite[Equations 3.6 and 3.7]{Babelon2003}). Computing the residues of $\textup{Tr}(L^{n})$ at $\lambda = \lambda_{q}$, i.e.\ 
\begin{equation*}
\frac{1}{(r-1)!} \lim_{\lambda \to \lambda_{q}} \frac{d^{r-1}}{d\lambda^{r-1}} \big( (\lambda - \lambda_{q})^{r} \textup{Tr}(L^{n}(\lambda)) \big),
\end{equation*}
yields the Hamiltonian function for the Gaudin model.

We consider Gaudin models related to the Lie algebra $\textup{su}(2)$, which is important in many areas of physics. A basis for $\textup{su}(2)$ is given by
\begin{align*}
U_{1} = \frac{1}{2} \begin{pmatrix} 0 & -i \\ -i & 0 \end{pmatrix}, \quad
U_{2} = \frac{1}{2} \begin{pmatrix} 0 & -1 \\ 1 & 0 \end{pmatrix}, \quad
U_{3} = \frac{1}{2} \begin{pmatrix} -i & 0 \\ 0 & i \end{pmatrix}.
\end{align*}
Note that the Lie bracket relations for $U_{1}$, $U_{2}$, and $U_{3}$ are given by $[U_{1},U_{2}] = U_{3}$, $[U_{2},U_{3}] = U_{1}$, and $[U_{3},U_{1}] = U_{2}$. Thus, the coadjoint orbits of the coadjoint action of $\textup{SU}(2)$ on $\textup{su}(2)$ are isomorphic to $\mathbb{S}^{2}$. The manifold in question is $\mathcal{M} = \mathbb{S}^{2} \times \mathbb{S}^{2}$, hence it is related to the Lie algebra $\textup{su}(2) \oplus \textup{su}(2)$.

Note that $\textup{su}(2)$ is isomorphic to the Lie algebra $\R^{3}$ for which the Lie bracket is the vector product. Let $\inner{\cdot,\cdot}$ denote the Euclidean inner product induced by $\R^{3}$. Furthermore, let boldface $\mathbf{v}_{i} := (x_{i},y_{i},z_{i})$ denote the vector with coordinates on the $i$-th sphere, and let $v_{i}^{j}$ denote the $j$-th component of $\mathbf{v}_{i}$, e.g.\ $v_{1}^{2} = y_{1}$. Furthermore, let boldface $\mathbf{w} = (w^{1},w^{2},w^{3})$ be some constant vector in $\R^{3}$. Petrera \cite[Equation 2.28]{Petrera2007} gives us the Lax matrices for the $\textup{su}(2)$ rational and trigonometric (and elliptic, which is omitted here) dependence on the spectral parameter $\lambda$:
\begin{align*}
&L^{\textsc{R}}_{\mathbf{w}}(\lambda) = \sum_{j = 1}^{3} \left( U_{j} w^{j} + \frac{U_{j} v_{1}^{j}}{\lambda - \lambda_{1}} + \frac{U_{j} v_{2}^{j}}{\lambda - \lambda_{2}} \right), \\
&L^{\textsc{T}}(\lambda) = \sum_{j = 1}^{3} \left( \frac{U_{j}v_{1}^{j} - (1 - \cos(\lambda - \lambda_{1}))U_{3}v_{1}^{3}}{\sin(\lambda - \lambda_{1})} + \frac{U_{j}v_{2}^{j} - (1 - \cos(\lambda - \lambda_{2}))U_{3}v_{2}^{3}}{\sin(\lambda - \lambda_{2})} \right),
\end{align*}
respectively (the superscript $\textsc{R}$ is for rational, and the superscript $\textsc{T}$ is for trigonometric). With this, we may find the corresponding Hamiltonians (see Petrera \cite[Propositions 2.4, 2.5 and 2.6]{Petrera2007} and references therein), which are given by computing the residues of $\textup{Tr}\left((L^{\textsc{R}}_{\mathbf{w}}(\lambda))^{2}\right)$ and $\textup{Tr}\left((L^{\textsc{T}}(\lambda))^{2}\right)$ at $\lambda = \lambda_{1}$ and $\lambda = \lambda_{2}$:
\begin{align*}
&H^{\textsc{R}}_{\mathbf{w},(t_{1},t_{2})} = \inner{\mathbf{w},t_{1}\mathbf{v_{1}} + t_{2}\mathbf{v_{2}}} + \frac{t_{1} - t_{2}}{\lambda_{1} - \lambda_{2}} \inner{\mathbf{v_{1}}, \mathbf{v_{2}}}, \\
&H^{\textsc{T}}_{(t_{0},t_{1},t_{2})} = t_{0} (v_{1}^{3} + v_{2}^{3})^{2} + \frac{t_{1} - t_{2}}{\sin(\lambda_{1} - \lambda_{2})} \left( \inner{\mathbf{v_{1}},\mathbf{v_{2}}} - (1 - \cos(\lambda_{1} - \lambda_{2})) v_{1}^{3}v_{2}^{3} \right).
\end{align*}

Recall that we want to consider integrable systems for which one of the integrals are given by $J = R_{1}z_{1} + R_{2}z_{2} = R_{1}v_{1}^{3} + R_{2}v_{2}^{3}$, as introduced in \eqref{eq:Gaudin-system-intro}. Let $\{\cdot,\cdot\}$ denote the Poisson bracket on $\mathbb{S}^{2}$. The coordinates $\mathbf{v}_{1}$ and $\mathbf{v}_{2}$ satisfy the bracket relations $\{v_{i}^{1},v_{i}^{2}\} = v_{i}^{3}$, $\{v_{i}^{2},v_{i}^{3}\} = v_{i}^{1}$, and $\{v_{i}^{3},v_{i}^{1}\} = v_{i}^{2}$. If $(J,H^{\textsc{R}}_{\mathbf{w},(t_{1},t_{2})})$ is to define a integrable system, then we must have that $\{J,H^{\textsc{R}}_{\mathbf{w},(t_{1},t_{2})}\} = 0$. However, this can only happen if we choose $w^{1} = w^{2} = 0$. Thus, we will only consider the case when $\mathbf{w} = (0,0,w^{3})$, and simply denote $\mathbf{w}$ by $w = w^{3}$.

\begin{definition}
Let boldface $\mathbf{t} = (t_{0},t_{1},t_{2},t_{3},t_{4}) \in \R^{5}$, and write $(x_{i},y_{i},z_{i})$ instead of $\mathbf{v}_{i}$, $i \in \{1,2\}$. We define a third Hamiltonian, which is a generalisation of $H^{\textsc{R}}_{\mathbf{w},(t_{1},t_{2})}$ and $H^{\textsc{T}}_{(t_{0},t_{1},t_{2})}$:
\begin{equation} \label{eq:generalised-Gaudin-Hamiltonian}
H_{w,\mathbf{t}}(x_{1},y_{1},z_{1},x_{2},y_{2},z_{2}) := t_{0}(z_{1} + z_{2})^{2} + w(t_{1}z_{1} + t_{2}z_{2}) + t_{3}(x_{1}x_{2} + y_{1}y_{2}) + t_{4}z_{1}z_{2}.
\end{equation}
\end{definition}

Note that if we set $t_{0} = 0$ and $t_{3} = t_{4} = \frac{t_{1}-t_{2}}{\lambda_{1}-\lambda_{2}}$ we obtain $H_{w,\mathbf{t}} = H^{\textsc{R}}_{w,(t_{1},t_{2})}$. Likewise, if we set $w = 0$, $t_{3} = \frac{t_{1}-t_{2}}{\sin(\lambda_{1}-\lambda_{2})}$ and $t_{4} = (t_{1}-t_{2})\cot(\lambda_{1}-\lambda_{2})$, then we obtain $H_{0,\mathbf{t}} = H^{\textsc{T}}_{(t_{0},t_{1},t_{2})}$. In fact, when we later refer to the rational Gaudin model, we are not going to enforce $t_{3} = t_{4}$, and so it would be more precise to call it a generalised rational Gaudin model.
\section{Normal form theory} \label{sec:HHB-theory}

In this section we recall some facts about the Hamiltonian Hopf bifurcation, in particular its normal form. Furthermore, we recall how one can use the normal form to make predictions about the dynamics of the system.

Let $(\mathcal{M},\omega,F=(f_{1},f_{2}))$ be an integrable system, and let $DF$ be the Jacobian of $F$. We say that a point $z_{0} \in \mathcal{M}$ is a singularity of $F$ if $DF|_{z_{0}}$ does not have maximal rank. In particular, the singularity is said to be of rank $1$ if the rank of $DF$ is $1$, and of rank $0$ or maximal corank if the rank of $DF$ is $0$. 

A rank $0$ singularity is said to be non-degenerate if the Hessians of $f_{1}$ and $f_{2}$ span a Cartan subalgebra in the real symplectic Lie algebra $\textup{sp}(4,\R)$ (cf.\ Bolsinov and Fomenko \cite[Section 1.8]{Bolsinov2004}). Non-degenerate rank $0$ singularities were classified by Williamson \cite{Williamson1936} by the eigenvalues of the linearised Hamiltonian vector field for the linear combination $c_{1}f_{1} + c_{2}f_{2}$ for generic $c_{1}, c_{2} \in \R$. He showed that there exists four different types of non-degenerate singularities of maximal corank for integrable systems with $2$ degrees of freedom. Let $\alpha,\beta \in \R \setminus \{0\}$. Then the singularities are classified as follows:
\begin{enumerate}[(i)]
    \item \emph{elliptic-elliptic}: four purely imaginary eigenvalues $i\alpha, -i\alpha, i\beta, -i\beta$,
    \item \emph{hyperbolic-hyperbolic}: four real eigenvalues $\alpha, -\alpha, \beta, -\beta$,
    \item \emph{elliptic-hyperbolic}: two real and two purely imaginary eigenvalues $\alpha, -\alpha, i\beta, -i\beta$,
    \item \emph{focus-focus}: four complex eigenvalues $\alpha+i\beta, \alpha-i\beta, -\alpha+i\beta, -\alpha-i\beta$.
\end{enumerate}
Note that in systems containing a circle action, i.e.\ systems for which at least one of the vector fields corresponding to $f_{1}$ and $f_{2}$ generate a periodic flow, hyperbolic-hyperbolic singularities cannot appear (see for instance Hohloch and Palmer \cite{Hohloch2021}). Note also that the eigenvalues for focus-focus singularities always come in quadruples, and hence such singularities cannot mix with the other types (unless we go to higher dimensions). 

Let us also recall what it means for a rank $1$ singularity $z_{0}$ to be non-degenerate (again, cf.\ Bolsinov and Fomenko \cite[Section 1.8]{Bolsinov2004}). Here $df_{1}$ and $df_{2}$ are linearly dependent, and so there exists $\lambda$ and $\mu$ such that $\lambda df_{1}(z_{0}) + \mu df_{2}(z_{0}) = 0$. Let $L$ be a tangent line to the orbit of the action of $\R^{2}$, and let $L'$ be its symplectic complement, i.e.\ $L' = \{ y \in T\mathcal{M} : \omega(x,y) = 0 \,\ \forall \, x \in T\mathcal{M} \}$. Then $z_{0}$ is said to be non-degenerate if the $2$-form $\lambda d^{2}f_{1}(z_{0}) + \mu d^{2}f_{2}(z_{0})$ is invertible. Note that the non-degenerate rank $1$ singularities can be either of \emph{elliptic-regular} type or of \emph{hyperbolic-regular} type. Furthermore, a simple type of degenerate singularity appears in the Gaudin models, namely \emph{cusps}, sometimes called \emph{parabolic} (for a precise definition, see e.g.\ Bolsinov, Guglielmi and Kudryavtseva \cite{Bolsinov2018}). As the name suggests, these singularities have a cuspidal shape in the image of the momentum map. The symplectic geometry of cusps has been given much attention recently, by Kudryavtseva \cite{Kudryavtseva2020}, Kudryavtseva and Martynchuk \cite{Kudryavtseva2021,Kudryavtseva2021b}, Kudryavtseva and Oshemkov \cite{Kudryavtseva2022}, as well as the classical work by Lerman and Umanski{\u i} \cite{Lerman1994}.

A rank $0$ singularity that goes through a Hamiltonian Hopf bifurcation changes from elliptic-elliptic type to focus-focus type, or the other way around, see Figure \ref{fig:HHBifurcation}. If we know that the eigenvalues change their type like this, and they do so transversally (to be made precise below), then we know that the singularity undergoes a Hamiltonian Hopf bifurcation. Furthermore, to determine whether or not we are in the presence of a Hamiltonian Hopf bifurcation, it is sufficient to show that the Hamiltonian function can be put in a certain normal form, which we now describe.

\begin{figure}[tb]
    \centering
    \includegraphics{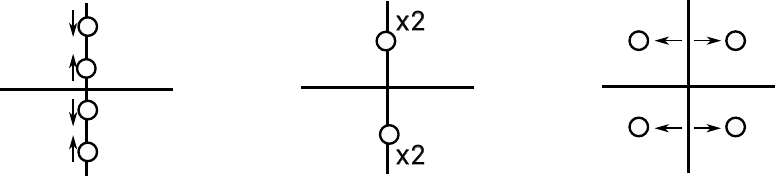}
    \caption{The eigenvalues for (from left to right) an elliptic-elliptic point, a degenerate point, and a focus-focus point, in the complex plane. The arrows indicates how the eigenvalues change with the bifurcation parameter $t$. The eigenvalues of the degenerate point have multiplicity $2$.}
    \label{fig:HHBifurcation}
\end{figure}

Let us introduce canonical coordinates $\{q_{1},q_{2},p_{1},p_{2}\}$ on $\mathcal{M}$, i.e.\ coordinates satisfying $\omega = \omega_{0} := dq_{1} \wedge dp_{1} + dq_{2} \wedge dp_{2}$. Let $t \in \R$ be a bifurcation parameter, and $H_{t} := f_{2} : \mathcal{M} \to \R$ be a Hamiltonian function for the integrable system defined in the beginning of this section. That is, we consider an integrable system with the momentum map $F = (f_{1},H_{t})$. Assume that $H_{t}$ has a singularity at $z_{0} \in \mathcal{M}$, and that $H_{t}(z_{0}) = 0$. We write the Taylor series expansion of $H_{t}$ at $z_{0}$ as $H_{t} = \sum_{k=2}^{\infty} H_{t}^{k}$, where $H_{t}^{k}$ denote the terms of $k$-th order. Williamson \cite{Williamson1936} and Cushman and Burgoyne \cite{Burgoyne1974a, Burgoyne1974b} found two normal forms for $H_{t}^{2}$ under real symplectic transformations. The normal forms depend on whether the linearisation of the Hamiltonian vector field of $H_{t}$ is semi-simple or not. Van der Meer \cite[Chapter 1.3]{vanderMeer1985} then defined the Hamiltonian Hopf bifurcation to be the one corresponding to the non-semi-simple case. Furthermore, he showed that the normal form of $H_{t}^{2}$ at the singularity and at the time of bifurcation, which we assume is $t = 0$, is given by
\begin{equation} \label{eq:HHB_normal_form_linear}
\widehat{H}_{0}^{2}(z_{0}) = \rho(q_{1}p_{2} - q_{2}p_{1}) + \frac{\sigma}{2}(q_{1}^{2} + q_{2}^{2}),
\end{equation}
for some $\rho \in \R \setminus \{0\}$ and $\sigma = \pm 1$. We denote normal forms by hats, e.g.\ $\widehat{H}_{t}$ and $\widehat{A}_{t}$, where $A_{t}$ denotes the linearisation of the Hamiltonian vector field of $H_{t}$. The normal form \eqref{eq:HHB_normal_form_linear} corresponds to the matrix
\begin{equation} \label{eq:HHB-normal-form-linear-matrix}
\widehat{A}_{0} =
\begin{pmatrix}
0 & -\rho & 0 & 0 \\
\rho & 0 & 0 & 0 \\
\sigma & 0 & 0 & -\rho \\
0 & \sigma & \rho & 0
\end{pmatrix}.
\end{equation}

\begin{definition} \label{def:HHB-linear}
For $i \in \{1,2\}$, let $\nu_{i} : I \to \R$ be two smooth functions on an open interval $I$ such that $\nu_{1}(0) = \nu_{2}(0) = 0$ and $(\partial \nu_{2} / \partial t)(0) \neq 0$. Assume that $H_{t}^{2}$ can be put in the normal form \eqref{eq:HHB_normal_form_linear} at $z_{0}$ (or equivalently, if $A_{t}$ can be put in the normal form \eqref{eq:HHB-normal-form-linear-matrix}). If $H_{t}^{2}$ has an unfolding of $z_{0}$, i.e. a germ $\widehat{H}_{t}^{2} : (\R^{4},0) \to (\mathcal{M},z_{0})$, given by
\begin{align}\label{eq:HHB_linear_unfolding}
\widehat{H}_{t}^{2}(z_{0}) = (\rho + \nu_{1}(t))(q_{1}p_{2} - q_{2}p_{1}) + \frac{\sigma}{2}(q_{1}^{2} + q_{2}^{2}) - \frac{\nu_{2}(t)}{2}(p_{1}^{2} + p_{2}^{2}),
\end{align}
then $H_{t}$ is said to go through a \emph{linear Hamiltonian Hopf bifucation} at $z_{0}$ and $t = 0$.
\end{definition}
\if
If $H_{t}^{2}$ can be put in this normal form, and has an unfolding of $z_{0}$, i.e. a germ $\widehat{H}_{t}^{2} : (\R^{4},0) \to (\mathcal{M},z_{0})$, given by
\begin{align}\label{eq:HHB_linear_unfolding}
\widehat{H}_{t}^{2}(z_{0}) = (\rho + \nu_{1}(t))(q_{1}p_{2} - q_{2}p_{1}) + \frac{\sigma}{2}(q_{1}^{2} + q_{2}^{2}) - \frac{\nu_{2}(t)}{2}(p_{1}^{2} + p_{2}^{2}),
\end{align}
where $\nu_{i} : I \to \R$ is a smooth function on an open interval $I$ for $i \in \{1,2\}$, then $H_{t}$ is said to go through a \emph{linear Hamiltonian Hopf bifucation} at $t = 0$ if $\nu_{1}(0) = \nu_{2}(0) = 0$ and $(\partial \nu_{2} / \partial t)(0) \neq 0$. 
\fi
The non-zero derivative condition in Definition \ref{def:HHB-linear} is the transversality condition, spoken of above. The unfolding of the matrix normal form \eqref{eq:HHB-normal-form-linear-matrix} is given by
\begin{equation} \label{eq:HHB-unfolding-linear-matrix}
\widehat{A}_{t} = 
\begin{pmatrix}
0 & -(\rho + \nu_{1}) & \nu_{2} & 0 \\
\rho + \nu_{1} & 0 & 0 & \nu_{2} \\
\sigma & 0 & 0 & -(\rho + \nu_{1}) \\
0 & \sigma & \rho + \nu_{1} & 0
\end{pmatrix}.
\end{equation}

Let us now consider the higher order normal form. Here we need to use some properties of the Poisson bracket, defined by $\{f_{1},f_{2}\} := \omega_{0}(X_{f_{1}},X_{f_{2}})$. Recall that, in canonical coordinates, we can write
\begin{equation} \label{eq:Poisson-bracket}
\{f_{1},f_{2}\} = \sum_{i=1}^{2} \left( \frac{\partial f_{1}}{\partial q_{i}}\frac{\partial f_{2}}{\partial p_{i}} - \frac{\partial f_{1}}{\partial p_{i}}\frac{\partial f_{2}}{\partial q_{i}} \right).
\end{equation}
As in Van der Meer \cite[Proof of Theorem 2.5]{vanderMeer1985}, let $G$ be a degree-$g$ homogeneous polynomial, and consider the adjoint action of $H_{t}$ by $G$, defined by $\ad_{G}H_{t} = \{G,H_{t}\}$. In particular, we consider the transformation $e^{\ad_{G}}H_{t}$. By using the power series expansion of $e^{x}$, and the Taylor expansion $H_{t} = \sum_{k=2}^{\infty} H_{t}^{k}$, we may write
\begin{align} \label{eq:transf-theory}
\begin{split}
e^{\ad_{G}}H_{t} &
= \left( \sum_{n=0}^{\infty} \frac{(\ad_{G})^{n}}{n!} \right) \sum_{k=2}^{\infty} H_{t}^{k} \\&
= (H_{t}^{2} + H_{t}^{3} + \cdots) + \ad_{G}(H_{t}^{2} + H_{t}^{3} + \cdots) + \frac{1}{2} \ad_{G}^{2}(H_{t}^{2} + H_{t}^{3} + \cdots) + \cdots.
\end{split}
\end{align}
Note that, using \eqref{eq:Poisson-bracket}, $\ad_{G}H_{t}^{k}$ has degree $g+k-2$. Let $G = G^{3}$ be a degree-$3$ homogeneous polynomial. The terms of degree $3$ are $H_{t}^{3} + \ad_{G^{3}}H_{t}^{2}$. We split the first of these terms into $H_{t}^{3} = \tilde{H}_{t}^{3} + \check{H}_{t}^{3}$, where $\tilde{H}_{t}^{3} \in \im(\ad_{H_{t}^{2}})$. This means that, by appropriately choosing $G^{3}$, we may simplify the degree-$3$ terms, namely by choosing a $G^{3}$ that solves $\tilde{H}_{t}^{3} + \ad_{G^{3}}H_{t}^{2} = \tilde{H}_{t}^{3} - \ad_{H_{t}^{2}}G^{3} = 0$. This can be extended to any degree, by appropriately choosing $G = G^{g}$, $g \geq 3$, and so we may remove all terms in $\im(\ad_{H_{t}^{2}})$ from \eqref{eq:transf-theory}. Thus, we say that $H_{t}^{k}$ is in normal form with respect to $H_{t}^{2}$ if $H_{t}^{k}$ lies in the complement of $\im(\ad_{H_{t}^{2}})$.

Next, let us make this complement a bit more concrete. Let $X_{f}$ denote the Hamiltonian vector field corresponding to a function $f$. In the normal form \eqref{eq:HHB_normal_form_linear}, we define $S := q_{1}p_{2} - q_{2}p_{1}$ and $M := \frac{1}{2}(q_{1}^{2} + q_{2}^{2})$, i.e.\ we write $\widehat{H}_{0}^{2} = \rho S + \sigma M$. This is the Jordan-Chevalley decomposition (decomposition into commuting semi-simple and niloptent matrices) of $\widehat{H}_{0}^{2}$. Indeed, note that the linearisation of $X_{S}$ is semi-simple, the linearisation of $X_{M}$ is nilpotent, and $\{S,M\} = 0$, where the Poisson bracket $\{\cdot,\cdot\}$ is defined in terms of the canonical symplectic structure. The functions $S$ and $M$ are two of the four Hilbert generators for the algebra of polynomials invariant under the action of the one parameter group corresponding to the flow of $X_{S}$ (see Van der Meer \cite[p.\ 57]{vanderMeer1985}). The remaining Hilbert generators are $N := \frac{1}{2}(p_{1}^{2} + p_{2}^{2})$ and $T := q_{1}p_{1} + q_{2}p_{2}$, and they satisfy $4MN = S^{2} + T^{2}$. The Hilbert generators $S,M,N,T$ satisfy the following Poisson bracket relations:
\begin{align} \label{eq:bracket-relations}
\{M,N\} = T, \quad
\{M,T\} = 2M, \quad
\{N,T\} = -2N, \quad
\{S,M\} = \{S,N\} = \{S,T\} = 0.
\end{align}
Thus, the span of $\im(\ad_{H_{t}^{2}})$ is $\spenn(M,T)$. The normal form with respect to $H_{t}^{2}$, lying in the complement of $\im(\ad_{H_{t}^{2}})$, lies then in the span of $S$ and $N$. Noticing that $\ker(\ad_{S}) = \spenn(S,M,N,T)$ and that $\ker(\ad_{N}) = \spenn(S,N)$ leads us to the following definition:

\begin{definition}[Van der Meer {\cite[Definition 2.1]{vanderMeer1993}}] \label{def:HHB-nonlinear}
\hfill
\begin{enumerate}[(i)]
    \item $H_{0}^{k}$ is in normal form with respect to $H_{0}^{2}$ if $H_{0}^{k} \in \ker(\ad_{S}) \cap \ker(\ad_{N})$.
    \item $H_{0}$ is in normal form up to order $k$ with respect to $H_{0}^{2}$ if $H_{0}^{l}$, $2 < l < k + 1$, is in normal form with respect to $H_{0}^{2}$.
\end{enumerate}
\end{definition}

By the bracket relations \eqref{eq:bracket-relations}, Definition \ref{def:HHB-nonlinear} implies that, at the time of bifurcation $t = 0$, the normal form of $H_{t}$ up to order $k$ is
\begin{equation} \label{eq:HHB-nonlin-normal-form-def}
\widehat{H}_{0} = \rho S + \sigma N + \sum_{l=2}^{k} \widehat{H}_{0}^{l}(M,S) + \text{higher order terms},
\end{equation}
where $\rho \in \R \setminus\{0\}$, $\sigma = \pm 1$, and the tuple $(M,S)$ in $\widehat{H}_{0}^{l}(M,S)$ means that $\widehat{H}_{0}^{l}$ only depends on the Hilbert generators $M$ and $S$, and not on $N$ and $T$.

\begin{remark}
In physics one usually says that the $q$-coordinates correspond to position, and $p$-coordinates correspond to momentum. In our context, however, there is no difference between these coordinates. Hence, the normal form in \eqref{eq:HHB-nonlin-normal-form-def} after changing the role of $M$ and $N$ still is a normal form for a Hamiltonian Hopf bifurcation.
\end{remark}

\begin{remark}
Van der Meer \cite[Section 3]{vanderMeer1982} showed that the normal form is void of all terms of odd degree.
\end{remark}

\begin{definition}
Let $H_{t}$ be a Hamiltonian function with the normal form at $t = 0$ given by \eqref{eq:HHB-nonlin-normal-form-def}. Furthermore, let the coefficient of $M^{2}$ in $\widehat{H}_{0}^{2}(M,S)$ be denoted by $a$. If $a \neq 0$, then the bifurcation is said to be \emph{non-degenerate}. If $\sigma a > 0$, the bifurcation is said to be \emph{supercritical}, and if $\sigma a < 0$, then the bifurcation is said to be \emph{subcritical}. If $a = 0$, but the coefficient of $M^{3}$ in $\widehat{H}_{0}^{3}(M,S)$ is non-zero, then the bifurcation is said to be \emph{degenerate}. In general, if the coefficient of $M^{k}$ for $2 \leq k \leq n + 1$ is zero but the coefficient of $M^{n+2}$ is non-zero, then we call the bifurcation \emph{$n$-degenerate}. 
\end{definition}

Note that if an integrable system undergoes a subcritical Hamiltonian Hopf bifurcation, then a so-called  \emph{flap} appears in the image of its momentum map (see Van der Meer \cite[Chapter 4]{vanderMeer1985}). A flap (see for instance Efstathiou and Giacobbe \cite{Efstathiou2012}) appears as an additional sheet in the image of the momentum map, connected to the original one along a line-segment of hyperbolic-regular values. To be more precise, consider a two-sheeted domain as in Figure \ref{fig:local_flap}. The sheet bounded by the elliptic-regular and hyperbolic-regular critical values is called a \emph{local flap}. The last line in the boundary of the local flap consists of regular values, and is called the \emph{free boundary} of the local flap. The other sheet is called the \emph{local base} of the local flap. One obtains a flap by gluing two local flaps by their respective free boundary, see Figure \ref{fig:flap}.

The gluing procedure discussed in the previous paragraph may be done in another order. If $\mathcal{F}_{1}$ is one local flap with local base $\mathcal{B}_{1}$, and $\mathcal{F}_{2}$ is another local flap with local base $\mathcal{B}_{2}$, then one obtains a \emph{pleat} (sometimes called a swallowtail) by gluing the free boundary of $\mathcal{F}_{1}$ to $\mathcal{B}_{2}$, and gluing the free boundary of $\mathcal{F}_{2}$ to $\mathcal{B}_{1}$, see Figure \ref{fig:pleat}. In Section \ref{sec:momentum_map} it is shown that also this structure occurs in Gaudin models. The bifurcation leading to pleats, usually called the swallowtail bifurcation, were discussed by, e.g., Efstathiou and Sugny \cite{EfstathiouSugny}. 

We have seen that if the eigenvalues of a rank $0$ singularity collide on the imaginary axis, and split off transversally into the complex plane, then the singularity undergoes a Hamiltonian Hopf bifurcation. A clear-cut method for telling where and when a pleat bifurcation happens is not known to the author. However, by studying the singularities of the Hamiltonian, we are still able to tell some basic information about (some of) the pleat bifurcations that appear in the Gaudin models.

\begin{figure}[tb]
    \centering
    \begin{subfigure}[t]{0.3\linewidth}
    \includegraphics[scale=0.3]{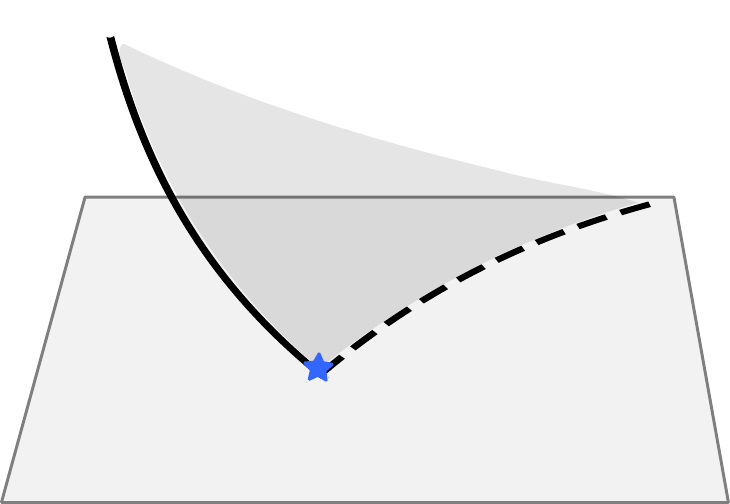}
    \caption{A local flap.}
    \label{fig:local_flap}
    \end{subfigure}
    \begin{subfigure}[t]{0.3\linewidth}
    \includegraphics[scale=0.3]{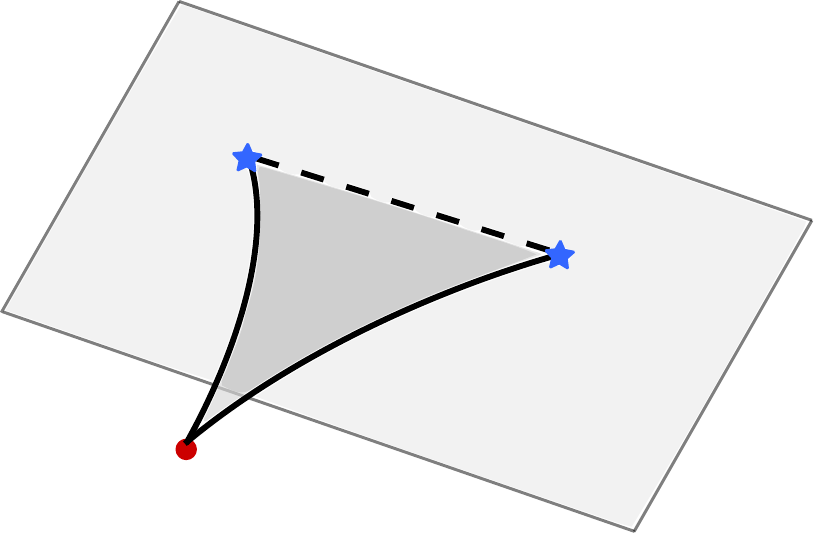}
    \caption{A flap.}
    \label{fig:flap}
    \end{subfigure}
    \begin{subfigure}[t]{0.3\linewidth}
    \includegraphics[scale=0.3]{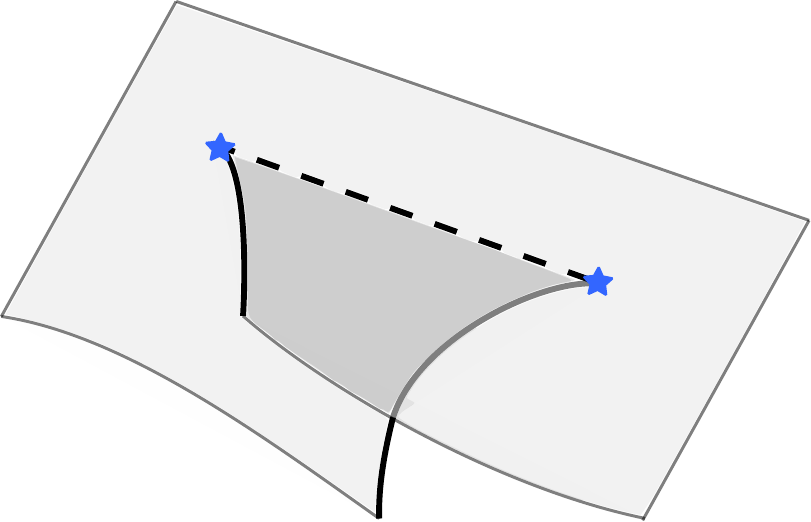}
    \caption{A pleat.}
    \label{fig:pleat}
    \end{subfigure}
    \caption{The figure shows \subref{fig:local_flap} a local flap, \subref{fig:flap} a flap, and \subref{fig:pleat} a pleat. The dashed line segments indicate the hyperbolic-regular values, whilst the thick black line segments indicate elliptic-regular values. The blue stars at the two ends of the hyperbolic-regular line indicate cusp values, and the red dot, where two elliptic-regular lines meet, indicates an elliptic-elliptic value.}
    \label{fig:flap_pleat}
\end{figure}
\section{Gaudin models undergo linear Hamiltonian Hopf bifurcations} \label{sec:linear-HHB}

Recall the integrable system $(\mathcal{M},\omega,(J,H_{w,\mathbf{t}}))$ from \eqref{eq:Gaudin-system-intro}, where the integrals of motion were given by
\begin{align} \label{eq:Gaudin-system}
\begin{cases}
J(x_{1},y_{1},z_{1},x_{2},y_{2},z_{2}) = R_{1}z_{1} + R_{2}z_{2}, \\
H_{w,\textbf{t}}(x_{1},y_{1},z_{1},x_{2},y_{2},z_{2}) = t_{0}(z_{1} + z_{2})^{2} + w(t_{1}z_{1} + t_{2}z_{2}) + t_{3}(x_{1}x_{2} + y_{1}y_{2}) + t_{4}z_{1}z_{2}.
\end{cases}
\end{align}
with $w \in \R$, $\textbf{t} = (t_{0},t_{1},t_{2},t_{3},t_{4}) \in \R^{5}$, and $(x_{i},y_{i},z_{i})$ being Cartesian coordinates on the $\mathbb{S}^{2}$-spheres, $i \in \{1,2\}$. In what follows, we simply write the integrals as $J$ and $H_{w,\mathbf{t}}$, skipping the coordinates.

The Hamiltonian Hopf bifurcation is a bifurcation of a singularity of rank 0, i.e.\ a singularity of both $J$ and $H_{w,\mathbf{t}}$. The integral $J$ has precisely four singularities, given by
\begin{align*}
\begin{split}
\begin{cases}
m_{0} := (0,0,1,0,0,-1), \\
m_{1} := (0,0,-1,0,0,-1), \\
m_{2} := (0,0,-1,0,0,1), \\
m_{3} := (0,0,1,0,0,1)
\end{cases}
\end{split}
\end{align*}
(see for example Le Floch and Pelayo \cite[Lemma 2.4]{LeFlochPelayoCAM}). It is easily verified that also $H_{w,\mathbf{t}}$ for all $\mathbf{t}$ has singularities at $m_{0}$, $m_{1}$, $m_{2}$, and $m_{3}$, and so these are the only candidates that may undergo a Hamiltonian Hopf bifurcation for specific values of $\mathbf{t}$.

Let us define 
\begin{align*}
&t_{4,m_{0}}^{\pm} := t_{4,m_{0}}^{\pm}(R_{1},R_{2},w,t_{1},t_{2},t_{3}) := \frac{w(t_{1}R_{2} - t_{2}R_{1}) \pm 2t_{3}\sqrt{R_{1}R_{2}}}{R_{1} + R_{2}}, \\
&t_{4,m_{2}}^{\pm} := t_{4,m_{2}}^{\pm}(R_{1},R_{2},w,t_{1},t_{2},t_{3}) := \frac{-w(t_{1}R_{2} - t_{2}R_{1}) \pm 2t_{3}\sqrt{R_{1}R_{2}}}{R_{1} + R_{2}},
\end{align*}
where we always assume the dependence on the parameters.

\begin{lemma} \label{lem:rank0type}
Let $t_{3} \neq 0$. The singularities $m_{1}$ and $m_{3}$ are of elliptic-elliptic type for any $(t_{0},t_{1},t_{2},t_{3},t_{4}) \in \R^{5}$. For $k \in \{0,2\}$, the singularity $m_{k}$ is of elliptic-elliptic type if $t_{4} < t_{4,m_{k}}^{-}$ or $t_{4} > t_{4,m_{k}}^{+}$. If either
\begin{itemize}
    \item $R_{1} = R_{2}$, $w \neq 0$, and $t_{1} + t_{2} \neq 0$, or
    \item $R_{1} \neq R_{2}$ and $(R_{1} - R_{2})t_{3} > \abs{w(t_{1} + t_{2})\sqrt{R_{1}R_{2}}}$,
\end{itemize}
then for $t_{4,m_{k}}^{-} < t_{4} < t_{4,m_{k}}^{+}$, $m_{k}$ is a focus-focus point. If neither of these conditions are met, $m_{k}$ is an elliptic-elliptic point in this region. Finally, for $t_{4} \in \{t_{4,m_{k}}^{-},t_{4,m_{k}}^{+}\}$, $m_{k}$ is degenerate.
\end{lemma}

\begin{remark}
If $w = 0$, then the condition $t_{1} + t_{2} \neq 0$ is unnecessary, as the effects of both $t_{1}$ and $t_{2}$ are not present in this case.
\end{remark}

\begin{remark}
If we let $t_{3} = 0$, then the linearisation of the Hamiltonian vector field of $H_{w,\mathbf{t}}$ at $m_{1}$ and $m_{3}$ has repeated eigenvalues for 
\begin{align*}
t_{4} = -4t_{0} - 2\frac{R_{2}t_{1} - R_{1}t_{2}}{R_{1}-R_{2}} \quad \text{and} \quad t_{4} = -4t_{0} + 2\frac{R_{2}t_{1} - R_{1}t_{2}}{R_{1}-R_{2}},
\end{align*}
respectively. Hence, $m_{1}$ and $m_{3}$ are degenerate singularities under these circumstances, depicted in Figure \ref{fig:t3=0_t4=1}. With $t_{3} \neq 0$, this can only happen for imaginary $t_{4}$. Furthermore, with $t_{3} = 0$, the image of the momentum map can look as in Figure \ref{fig:t3=0}, where the darker shaded area is covered twice, and the black curves corresponds to the boundary. We conjecture that the edges with no black curves are so-called fold singularities, as described in Giacobbe \cite{Giacobbe2007}.
\end{remark}

\begin{figure}[tb]
\def\scale{0.8}
\centering
\begin{tabular}{cc}
    \subfloat[$t_{4} = -1$. \label{fig:t3=0_t4=-1}]{\includegraphics[scale=\scale]{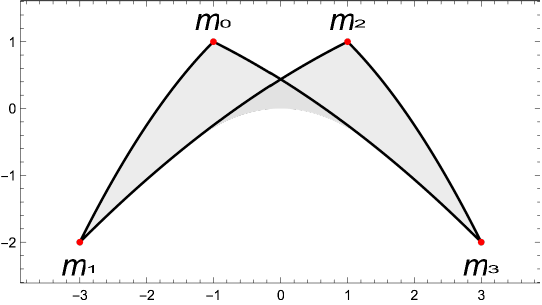}} & 
    \subfloat[$t_{4} = 1$. \label{fig:t3=0_t4=1}]{\includegraphics[scale=\scale]{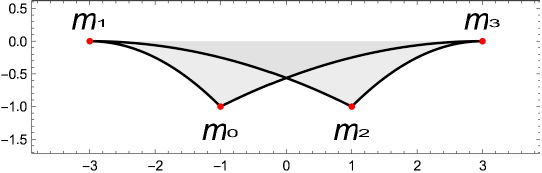}}
\end{tabular}
\caption{The image of the momentum map for two values of $t_{4}$ with $R_{1} = 1$, $R_{2} = 2$, $w = 0$, $t_{0} = -0.25$, and $t_{3} = 0$.}
\label{fig:t3=0}
\end{figure}

\begin{proof}[Proof of Lemma \ref{lem:rank0type}]
The proof follows exactly the same steps as Le Floch and Pelayo's \cite[Proposition 2.5]{LeFlochPelayoCAM} proof concerning the coupled angular momenta system, which corresponds to the special case $w = 1$, $t_{0} = t_{2} = 0$, and $t_{3} = t_{4} = 1 - t_{1}$.
\end{proof}

Lemma \ref{lem:rank0type} describes behaviour characteristic for the Hamiltonian Hopf bifurcation. Note that, for $k \in \{0,2\}$, $t_{4,m_{k}}^{\pm}$ is independent of $t_{0}$. The next theorem gives rigorous proof that the Gaudin model defined in \eqref{eq:Gaudin-system} actually undergoes (linear) Hamiltonian Hopf bifurcations at $t_{4,m_{k}}^{\pm}$. 

\begin{theorem} \label{thm:linearHHB}
Let $k \in \{0,2\}$. If the conditions necessary for focus-focus points in Lemma \ref{lem:rank0type} are met, then, for both $t_{4} = t_{4,m_{k}}^{+}$ and $t_{4} = t_{4,m_{k}}^{-}$, the point $m_{k}$ undergoes a linear Hamiltonian Hopf bifurcation.
\end{theorem}

\begin{proof}
The proof is similar in all four cases, so we present only one of them, namely the bifurcation at $t_{4} = t_{4,m_{0}}^{+}$. Furthermore, the proof is inspired by Cushman and Bates' \cite[Section 8.1]{Cushman1997} proof, where they show that the Lagrange top undergoes a linear Hamiltonian Hopf bifurcation.

We split the proof into three steps. The first two steps follow from Burgoyne and Cushman's algorithm \cite{Burgoyne1974a, Burgoyne1974b}; in step 1, we find the Jordan-Chevalley decomposition of the linearised vector field, and in step 2 we use the the decomposition to find a symplectic coordinate change putting the Hamiltonian in the matrix version of the normal form \eqref{eq:HHB-normal-form-linear-matrix}. In step 3 we find an unfolding of $H_{\mathbf{t}}$ in the shape of \eqref{eq:HHB-unfolding-linear-matrix}, and show that the eigenvalues intersect the imaginary axis transversally.

\textit{Step 1:} We parameterise the two $2$-spheres by Cartesian coordinates, such that near $m_{0} = (0,0,1,0,0,-1)$, we have
\begin{align*}
z_{1}(x_{1},y_{1}) = \sqrt{1 - x_{1}^{2} - y_{1}^{2}}, \quad z_{2}(x_{2},y_{2}) = - \sqrt{1 - x_{2}^{2} - y_{2}^{2}}.
\end{align*}
Thus, near $m_{0}$ the symplectic form is
\begin{align*}
\omega = \frac{R_{1} dx_{1} \wedge dy_{1}}{\sqrt{1 - x_{1}^{2} - y_{1}^{2}}} - \frac{R_{2} dx_{2} \wedge dy_{2}}{\sqrt{1 - x_{2}^{2} - y_{2}^{2}}},
\end{align*}
and the integrals from \eqref{eq:Gaudin-system} are given by
\begin{align*}
\begin{cases}
J = R_{1} \sqrt{1 - x_{1}^{2} - y_{1}^{2}} - R_{2} \sqrt{1 - x_{2}^{2} - y_{2}^{2}}, \\
H_{w,\mathbf{t}} = t_{0} \left( \sqrt{1 - x_{1}^{2} - y_{1}^{2}} - \sqrt{1 - x_{2}^{2} - y_{2}^{2}} \right)^{2} \\\qquad\qquad + w \left( t_{1} \sqrt{1 - x_{1}^{2} - y_{1}^{2}} - t_{2} \sqrt{1 - x_{2}^{2} - y_{2}^{2}} \right) \\\qquad\qquad + t_{3} \left( x_{1}x_{2} + y_{1}y_{2} \right) - t_{4} \sqrt{1 - x_{1}^{2} - y_{1}^{2}}\sqrt{1 - x_{2}^{2} - y_{2}^{2}}.
\end{cases}
\end{align*}
The terms of second order in the Taylor series of $H_{w,\mathbf{t}}$ are then
\begin{align*}
H_{w,\mathbf{t}}^{2} = - \frac{wt_{1} - t_{4}}{2}(x_{1}^{2} + y_{1}^{2}) + \frac{wt_{2} + t_{4}}{2}(x_{2}^{2} + y_{2}^{2}) + t_{3} (x_{1}x_{2} + y_{1}y_{2}).
\end{align*}
Thus, the linearisation of the Hamiltonian vector field of $H_{w,\mathbf{t}}$ at $m_{0}$, in a basis of the tangent space $T_{m_{0}}\mathcal{M}$ associated with $(x_{1},y_{1},x_{2},y_{2})$, is given by
\begin{align*}
A_{t_{4}} :=
A_{H_{w,\mathbf{t}}}(R_{1},R_{2}) =
\left(
\begin{array}{cccc}
 0 & \frac{wt_{1}-t_{4}}{R_1} & 0 & -\frac{t_{3}}{R_1} \\
 -\frac{wt_{1}-t_{4}}{R_1} & 0 & \frac{t_{3}}{R_1} & 0 \\
 0 & \frac{t_{3}}{R_2} & 0 & \frac{wt_{2}+t_{4}}{R_2} \\
 -\frac{t_{3}}{R_2} & 0 & -\frac{wt_{2}+t_{4}}{R_2} & 0 \\
\end{array}
\right).
\end{align*}
The characteristic polynomial of $A_{t_{4,m_{0}}^{+}}$ is $p(\lambda) = (\lambda^{2} + \alpha^{2})^{2}$, where 
\begin{align*}
\alpha &
:= \alpha(R_{1},R_{2},w,\mathbf{t}) 
:= \frac{w(t_{1}+t_{2})\sqrt{R_{1}R_{2}} + t_{3}(R_{1}-R_{2})}{\sqrt{R_{1}R_{2}}(R_{1}+R_{2})}.
\end{align*}
With the characteristic polynomial in this shape, Burgoyne and Cushman \cite{Burgoyne1974a, Burgoyne1974b} give us an algorithm to find the Jordan-Chevalley decomposition of $A_{t_{4}}$, i.e.\ the unique decomposition $A_{t_{4,m_{0}}^{+}} = \mathcal{S} + \mathcal{N}$ such that $\mathcal{S}$ is semi-simple, $\mathcal{N}$ is nilpotent, and $\mathcal{S}\mathcal{N} = \mathcal{N}\mathcal{S}$. In particular, they showed that the semi-simple part is given by
\begin{equation*}
\mathcal{S} = A_{t_{4}} \left( \mathds{1} + \sum_{j=1}^{m-1} \binom{2j}{j} \left( \frac{p(A_{t_{4}})}{4\alpha^{2}} \right)^{j} \right),
\end{equation*}
where $\mathds{1}$ is the identity matrix, $p(A_{t_{4}}) = (A_{t_{4}}^{2} + \alpha^{2}\mathds{1})^{2}$, and $m$ is the integer such that $p(A_{t_{4}})^{m} = 0$ while $p(A_{t_{4}})^{m-1} \neq 0$. A short calculation shows that $m = 2$. Furthermore, $\mathcal{N} = A_{t_{4}} - \mathcal{S}$. Thus, we find
\begin{align*}
\mathcal{S} = 
\left(
\begin{array}{cccc}
 0 & \alpha & 0 & 0 \\
 -\alpha & 0 & 0 & 0 \\
 0 & 0 & 0 & \alpha \\
 0 & 0 & -\alpha & 0 \\
\end{array}
\right), \qquad
\mathcal{N} = 
\left(
\begin{array}{cccc}
 0 & -\frac{t_{3}}{\sqrt{R_{1}R_{2}}} & 0 & -\frac{t_{3}}{R_{1}} \\
 \frac{t_{3}}{\sqrt{R_{1}R_{2}}} & 0 & \frac{t_{3}}{R_{1}} & 0 \\
 0 & \frac{t_{3}}{R_{2}} & 0 & \frac{t_{3}}{\sqrt{R_{1}R_{2}}} \\
 -\frac{t_{3}}{R_{2}} & 0 & -\frac{t_{3}}{\sqrt{R_{1}R_{2}}} & 0 \\
\end{array}
\right).
\end{align*}

\textit{Step 2:} The next step of Burgoyne and Cushman's algorithm consists of finding a symplectic basis. First one needs to construct certain subspaces of $\R^{4}$. For $1 \leq i \leq m = 2$, let $K_{i} := \{ x \in \R^{4} : \mathcal{N}^{i}x = 0 \}$, and $K_{0} := \{0\}$. Note that $\mathcal{N}^{2} = 0$, and so $K_{2} = \R^{4}$. Furthermore, a simple calculation yields that $K_{1}$ is spanned by $(\sqrt{R_{2}},0,-\sqrt{R_{1}},0)$ and $(0,\sqrt{R_{2}},0,-\sqrt{R_{1}})$. 

The next move consists of recursively creating sets $W_{j}$. To this end, we define sets $E_{j}$ such that $W_{j} = E_{j} + W_{j-1}$, where we define $W_{0} := \{0\}$. Let $W_{j}^{\omega} := \{ x \in \R^{4} : \omega(x,y) = 0 \, \forall \, y \in W_{j} \}$ be the symplectic complement of $W_{j}$. We need to find the number $k_{j} \in [0,m]$ such that
\begin{equation*}
W_{j}^{\omega} \cap K_{k_{j}+1} = W_{j}^{\omega}, 
\quad \text{and} \quad 
W_{j}^{\omega} \cap K_{k_{j}} \neq W_{j}^{\omega}.
\end{equation*}
Note that, for $j = 0$, this is satisfied if we choose $k_{0} = 1$. Next, we need to choose some $e \not\in K_{k_{j}}$ and $\omega(e,\mathcal{N}^{k_{j}}e) = \epsilon_{j}$, where $\epsilon_{j}^{2} = 1$ (if $k_{j}$ is even, change $\mathcal{N}^{k_{j}}$ to $\mathcal{N}^{k_{j}}\mathcal{S}$). Again, for $j = 0$, let $\beta := \beta(R_{1},R_{2},t_{3}) := \frac{1}{2} \left(R_{1}R_{2}t_{3}^{2}\right)^{-1/4}$. Then we may for example choose $e = \beta (\sqrt{R_{2}},0,\sqrt{R_{1}},0) \not\in K_{1}$, which also satisfies $\omega(e,\mathcal{N}e) = 1$, i.e.\ $\epsilon_{0} = 1$. Now, $E_{j}$ is the space spanned by $\mathcal{N}^{l}e$, $\mathcal{N}^{l}\mathcal{S}e$, for $0 \leq l \leq k_{j}$. One repeats this construction until $W_{j} = E_{j} + W_{0} = \R^{4}$. In our case, only one iteration is necessary, as the matrix defined as the columns of $e$, $\mathcal{S}e$, $\mathcal{N}e$, and $\mathcal{N}\mathcal{S}e$, which we denote by $\textup{col} \left( e, \mathcal{S}e, \mathcal{N}e, \mathcal{N}\mathcal{S}e \right)$, has maximal rank.

Finally we construct a symplectic basis for $E_{j}$. Burgoyne and Cushman \cite{Burgoyne1974a} tell us that one constructs the basis as follows:
\begin{align*}
&f := e + \frac{1}{2\alpha^{2}} \omega(e,\mathcal{S}e)\mathcal{N}\mathcal{S}e = \beta \left( \sqrt{R_{2}}, 0, \sqrt{R_{1}}, 0 \right), 
&&\frac{1}{\alpha} \mathcal{S}f = \beta \left( 0, -\sqrt{R_{2}}, 0, -\sqrt{R_{1}} \right), \\
&\mathcal{N}f = \beta \left( 0, \frac{2t_{3}}{\sqrt{R_{1}}}, 0, -\frac{2t_{3}}{\sqrt{R_{2}}} \right), 
&&\frac{1}{\alpha}\mathcal{S}\mathcal{N} = \beta \left( \frac{2t_{3}}{\sqrt{R_{1}}}, 0, -\frac{2t_{3}}{\sqrt{R_{2}}}, 0 \right).
\end{align*}
It follows that $\omega(f,\mathcal{N}f) = 1 = \omega(\frac{1}{\alpha}\mathcal{S}f,\frac{1}{\alpha}\mathcal{S}\mathcal{N}f)$, and $\omega$ applied to the other basis vectors vanishes. Let $P_{t_{4,m_{0}}^{+}} := \textup{col}(f,\frac{1}{\alpha}\mathcal{S}f,\mathcal{N}f,\frac{1}{\alpha}\mathcal{S}\mathcal{N}f)$ be the matrix with columns $f$, $\frac{1}{\alpha}\mathcal{S}f$, $\mathcal{N}f$, and $\frac{1}{\alpha}\mathcal{S}\mathcal{N}f$. It maps the $(x_{1},y_{1},x_{2},y_{2})$-coordinates to the new, canonical, coordinates. Conjugating $A_{t_{4,m_{0}}^{+}}$ by $P_{t_{4,m_{0}}^{+}}$, i.e.\ performing the coordinate transformation, yields the desired matrix:
\begin{align*}
\widehat{A} := P_{t_{4,m_{0}}^{+}}^{-1} A_{t_{4,m_{0}}^{+}} P_{t_{4,m_{0}}^{+}} = 
\left(
\begin{array}{cccc}
 0 & -\alpha & 0 & 0 \\
 \alpha & 0 & 0 & 0 \\
 1 & 0 & 0 & -\alpha \\
 0 & 1 & \alpha & 0 \\
\end{array}
\right).
\end{align*}

\textit{Step 3:} Finally, we prove that the eigenvalues intersect the imaginary axis transversally, by finding the unfolding given by \eqref{eq:HHB-unfolding-linear-matrix}. For this, we aim to find a smooth family $t_{4} \to V_{t_{4}}$ of symplectic matrices with respect to $(\R^{4},\omega_{0})$ such that $V_{t_{4,m_{0}}^{+}} = \widehat{A}$. Consider the following coordinate change:
\begin{align} \label{eq:Q}
Q = 
\begin{pmatrix}
\frac{1}{\sqrt{2R_{1}}} & 0 & 0 & \frac{1}{\sqrt{2R_{1}}} \\
0 & -\frac{1}{\sqrt{2R_{1}}} & \frac{1}{\sqrt{2R_{1}}} & 0 \\
-\frac{1}{\sqrt{2R_{2}}} & 0 & 0 & \frac{1}{\sqrt{2R_{2}}} \\
0 & \frac{1}{\sqrt{2R_{2}}} & \frac{1}{\sqrt{2R_{2}}} & 0 \\
\end{pmatrix}.
\end{align}
Then $Q^{\intercal} \Omega Q = \Omega_{0}$, where ${}^{\intercal}$ denotes the transpose, and $\Omega$ and $\Omega_{0}$ are the matrices of the symplectic form $\omega$ (in the old basis) and the canonical symplectic form $\omega_{0}$ (in the new basis), respectively. To obtain the smooth family $t_{4} \to V_{t_{4}}$, consider the transformed smooth family $t_{4} \to U_{t_{4}} = Q^{-1} A_{t_{4}} Q$. If we define $Q' := Q^{-1}P_{t_{4,m_{0}}^{+}}$, then $(Q')^{-1}U_{t_{4,m_{0}}^{+}}Q' = \widehat{A}$. Thus, the family $t_{4} \to V_{t_{4}} = (Q')^{-1}U_{t_{4}}Q'$ is the one we sought.

Following Cushman and Bates \cite[pp.\ 258-260]{Cushman1997}, the family $t_{4} \to V_{t_{4}}$ can be transformed into the smooth normal form
\begin{align*}
Y_{t_{4}} = 
\begin{pmatrix}
0 & -(\alpha + \nu_{1}(t_{4})) & \nu_{2}(t_{4}) & 0 \\
\alpha + \nu_{1}(t_{4}) & 0 & 0 & \nu_{2}(t_{4}) \\
1 & 0 & 0 & -(\alpha + \nu_{1}(t_{4})) \\
0 & 1 & \alpha + \nu_{1}(t_{4}) & 0
\end{pmatrix},
\end{align*}
where $\nu_{j} : I \to \R$ is a smooth function on some interval $I$ containing $t_{4,m_{0}}^{+}$, such that $\nu_{j}(t_{4,m_{0}}^{+}) = 0$, for $j \in \{1,2\}$. We want to find expressions for $\nu_{j}$. Note that they do not (necessarily) depend only on $t_{4}$, but on all the parameters $R_{1}$, $R_{2}$, $w$, and $\mathbf{t}$. However, we consider $t_{4}$ as the bifurcation parameter, and so special emphasis is put on this variable. The transformation $V_{t_{4}}$ to $Y_{t_{4}}$ is given by a conjugation by symplectic matrices $P_{t_{4}}$, i.e.\ $Y_{t_{4}} = P_{t_{4}}V_{t_{4}}P_{t_{4}}^{-1}$. This makes $t_{4} \to Y_{t_{4}}$ and $t_{4} \to A_{t_{4}}$ smoothly conjugate, and so we may compare their characteristic polynomials. If $\lambda$ denote an eigenvalue for $Y_{t_{4}}$ and $A_{t_{4}}$, then the characteristic polynomial is, respectively,
\begin{align} \label{eq:nu1nu2}
&p_{Y}(\lambda) = \lambda^{4} + a_{Y}(\nu_{1},\nu_{2}) \lambda^{2} + b_{Y}(\nu_{1},\nu_{2}) 
\quad \text{and} \quad
p_{A}(\lambda) = \lambda^{4} + a_{A} \lambda^{2} + b_{A},
\end{align}
for some suitably chosen functions $a_{k}, b_{k}$, $k \in \{Y,A\}$. Equating the coefficients for equal powers of $\lambda$ yields expressions for $\nu_{1}(t_{4})$ and $\nu_{2}(t_{4})$. There exists four different solutions to \eqref{eq:nu1nu2}, but only one giving $\nu_{1}(t_{4,m_{0}}^{+}) = 0$ and $\nu_{2}(t_{4,m_{0}}^{+}) = 0$, namely
\begin{align*}
\begin{dcases}
\nu_{1}(t_{4}) = - \alpha + \frac{(wt_{1}-t_{4})R_{2} + (wt_{2}+t_{4})R_{1}}{2R_{1}R_{2}}, \\
\nu_{2}(t_{4}) = - \frac{(wt_{1} - t_{4})^{2}R_{2}^{2} + (wt_{2} + t_{4})R_{1}^{2} - 2R_{1}R_{2}((wt_{1} - t_{4})(wt_{2} + t_{4}) + 2t_{3}^{2})}{4R_{1}^{2}R_{2}^{2}}.
\end{dcases}
\end{align*}
Furthermore, this solution to \eqref{eq:nu1nu2} yields
\begin{align*}
\frac{d\nu_{2}}{dt_{4}}\bigg|_{t_{4}=t_{4,m_{0}}^{+}} = - \frac{t_{3}(R_{1} + R_{2})}{(R_{1}R_{2})^{3/2}} \neq 0 \quad \text{if } t_{3} \neq 0,
\end{align*}
and so the unfolding given by $Y_{t_{4}}$ goes through a linear Hamiltonian Hopf bifurcation at $t_{4} = t_{4,m_{0}}^{+}$ if $t_{3} \neq 0$.
\end{proof}

\section{Non-linear normal form} \label{sec:nonlinear-HHB}

In this section we compute a normal form for $H_{w,\mathbf{t}}$ up to $6$th order. For this, we need to first find a coordinate transformation making the symplectic form standard. However, by the following lemma it suffices to transform $\omega$ in such a way that its Taylor series, apart from the constant term, vanishes up to $6$th order. This process is called \emph{flattening} of the symplectic form. Furthermore, we need the constant term in the flattened symplectic form to be standard.

\begin{lemma}[Efstathiou, Cushman and Sadovski{\'i} {\cite[Lemma 1]{Efstathiou2004}}] \label{lem:j-jet-Hamiltonian}
Consider a Hamiltonian $H = H^{2} + H^{3} + \cdots$ and a symplectic form $\omega = \omega^{0} + \omega^{j} + \cdots$, i.e.\ $\omega^{k} = 0$ for $1 \leq k \leq j-1$. Then the $j$-jet of the Hamiltonian vector field $X$ of $H$ with respect to $\omega$ is equal to the $j$-jet of the Hamiltonian vector field $Y$ of $H$ with respect to $\omega^{0}$.
\end{lemma}

We introduce a small parameter $\varepsilon$. Let $(x,y) := (x_{1},y_{1},x_{2},y_{2})$ denote the Cartesian coordinates on $\mathcal{M}$. The coordinates $(q,p) := (q_{1},q_{2},p_{1},p_{2})$ defined by $(q,p)^{\intercal} = Q^{-1}(x,y)^{\intercal}$, where $Q$ is the matrix defined in \eqref{eq:Q}, is now replaced by $(\varepsilon q, \varepsilon p)$. Furthermore, we introduce the blown up symplectic form $\omega'(q,p) = \frac{1}{\varepsilon^{2}} \omega(\varepsilon q, \varepsilon p)$, which allows us to apply Proposition \ref{prop:flattening_of_sympl_form} below for the flattening of the symplectic form. Let us now consider its $6$-jet at $m_{0}$. After a simple calculation, we find that all non-vanishing terms of the $6$-jet are given by $\omega^{0}(q,p) + \varepsilon^{2} \omega^{2}(q,p) + \varepsilon^{4} \omega^{4}(q,p) + \varepsilon^{6} \omega^{6}(q,p)$. To write out each term, it is convenient to define
\begin{equation*}
\chi_{n}^{\pm} := \left( (p_{2}-q_{1})^{2} + (p_{1}+q_{2})^{2} \right)^{n/2} R_{1}^{n/2} \pm \left( (p_{2}+q_{1})^{2} + (p_{1}-q_{2})^{2} \right)^{n/2} R_{2}^{n/2}.
\end{equation*}
Then the non-vanishing terms of the $6$-jet of $\omega$ are
\begin{equation*}
\omega^{0}(q,p) = dq_{1} \wedge dp_{1} + dq_{2} \wedge dp_{2},
\end{equation*}
\begin{equation*}
\omega^{2}(q,p) = \frac{\chi_{2}^{-}(dq_{1} \wedge dq_{2} + dp_{1} \wedge dp_{2})}{8R_{1}R_{2}} + \frac{\chi_{2}^{+}(dq_{1} \wedge dp_{1} + dq_{2} \wedge dp_{2})}{8R_{1}R_{2}},
\end{equation*}
\begin{equation*}
\omega^{4}(q,p) = \frac{3\chi_{4}^{-}(dq_{1} \wedge dq_{2} + dp_{1} \wedge dp_{2})}{64R_{1}^{2}R_{2}^{2}} + \frac{3\chi_{4}^{+}(dq_{1} \wedge dp_{1} + dq_{2} \wedge dp_{2})}{64R_{1}^{2}R_{2}^{2}},
\end{equation*}
\begin{equation*}
\omega^{6}(q,p) = \frac{5\chi_{6}^{-}(dq_{1} \wedge dq_{2} + dp_{1} \wedge dp_{2})}{256R_{1}^{2}R_{3}^{3}} + \frac{5\chi_{6}^{+}(dq_{1} \wedge dp_{1} + dq_{2} \wedge dp_{2})}{256R_{1}^{3}R_{2}^{3}}.
\end{equation*}
By Lemma \ref{lem:j-jet-Hamiltonian}, it is sufficient to find coordinate transformations which map $\omega^{2}(q,p)$ and $\omega^{4}(q,p)$ to $0$. This is the content of the following proposition. Here $\mathcal{L}_{X}$ denotes the Lie derivative along a vector field $X$. Note that in the computation we are doing for the Gaudin models, we simply have to replace $\varepsilon$ by $\varepsilon^{2}$.

\begin{proposition} \label{prop:flattening_of_sympl_form}
Let $\beta = \beta^{0} + \varepsilon\beta^{1} + \varepsilon^{2}\beta^{2} + \mathcal{O}(\varepsilon^{3})$ be a formal power series of a closed $2$-form on $\R^{n}$ with $\beta^{0}$ a constant symplectic form. By the Poincaré lemma, there is a formal power series of a $1$-form $\alpha = \alpha^{0} + \varepsilon\alpha^{1} + \varepsilon^{2}\alpha^{2} + \mathcal{O}(\varepsilon^{3})$ such that $\beta = d\alpha$. Define two vector fields $X$ and $Y$ by $\iota_{X}\beta^{0} = -\alpha^{1}$ and $\iota_{Y}\beta^{0} = -\alpha^{2} - \frac{1}{2}\iota_{X}\beta^{1}$. Then changing coordinates by the time $\varepsilon$ map of the flow of $X$ and subsequently by the time $\varepsilon^{2}$ map of the flow of $Y$ flattens $\beta$ up to third order, i.e.\ $(\exp \varepsilon^{2} \mathcal{L}_{Y})^{*} \left( (\exp \varepsilon \mathcal{L}_{X})^{*} \beta \right) = \beta^{0} + \mathcal{O}(\varepsilon^{3})$.
\end{proposition}

\begin{proof}
The flattening up to second order was done by Cushman and Bates \cite[Section 8.2]{Cushman1997}. We recall their proof, to get an intuition how to proceed with the higher powers of $\varepsilon$. The flattening up to third order works exactly the same way.

Cushman and Bates \cite[Section 8.2]{Cushman1997} showed that for a differential form (in fact for any \emph{smooth geometric quantity on $\R^{n}$}), and a vector field $U$,
\begin{equation*}
(\exp \varepsilon \mathcal{L}_{U})^{*} T = \sum_{n \geq 0} \frac{\varepsilon^{n}}{n!} \mathcal{L}_{U}^{n} T.
\end{equation*}
Hence,
\begin{align*}
\widehat{\alpha} 
= &(\exp \varepsilon \mathcal{L}_{X})^{*} \alpha 
= \alpha^{0} + \varepsilon \left( \alpha^{1} + \mathcal{L}_{X}\alpha^{0} \right) + \varepsilon^{2} \left( \alpha^{2} + \mathcal{L}_{X}\alpha^{1} + \frac{1}{2} \mathcal{L}_{X}^{2}\alpha^{0} \right) + \mathcal{O}(\varepsilon^{3}),
\end{align*}
and
\begin{align*}
\widehat{\widehat{\alpha}}
= &(\exp \varepsilon^{2} \mathcal{L}_{Y})^{*} \widehat{\alpha} \\&
= \alpha^{0} + \varepsilon \left( \alpha^{1} + \mathcal{L}_{X}\alpha^{0} \right) +\varepsilon^{2} \left( \alpha^{2} + \mathcal{L}_{X}\alpha^{1} + \frac{1}{2} \mathcal{L}_{X}^{2}\alpha^{0} + \mathcal{L}_{Y}\alpha^{0} \right) + \mathcal{O}(\varepsilon^{3}).
\end{align*}
By the Cartan formula, and using the defining properties of $X$ and $Y$, we get
\begin{align*}
\widehat{\widehat{\alpha}}
= \alpha^{0} + \varepsilon d\iota_{X}\alpha^{0} + \varepsilon^{2} \left( d\iota_{X}\alpha^{1} + d\iota_{X}d\iota_{X}\alpha^{0} + d\iota_{Y}\alpha^{0} \right) + \mathcal{O}(\varepsilon^{3}).
\end{align*}
Thus,
\begin{align*}
(\exp \varepsilon^{2} \mathcal{L}_{Y})^{*} \left( (\exp \varepsilon \mathcal{L}_{X})^{*} \beta \right)
= d (\exp \varepsilon^{2} \mathcal{L}_{Y})^{*} \left( (\exp \varepsilon \mathcal{L}_{X})^{*} \alpha \right)
= d \widehat{\widehat{\alpha}}
= d\alpha_{0} + \mathcal{O}(\varepsilon^{3}),
\end{align*}
as desired.
\end{proof}

\begin{remark}
Note that Proposition \ref{prop:flattening_of_sympl_form} may be generalised such that one may flatten the symplectic form to any order.
\end{remark}

We are now ready to prove the following theorem.

\begin{theorem} \label{thm:nonlinearHHB}
Let the conditions necessary for focus-focus points in Lemma \ref{lem:rank0type} be met. The normal form of $H_{w,\mathbf{t}}$ at $m_{0}$ up to third order in the Hilbert generators $S$, $M$, and $N$ (from Section \ref{sec:HHB-theory}) is
\begin{equation} \label{eq:nonlinear-normal-form-Gaudin}
\widehat{H}_{w,\mathbf{t}} = a_{1}S + N + a_{2}M + a_{3}M^{2} + a_{4}MS + a_{5}S^{2} + a_{6}M^{3} + a_{7}M^{2}S + a_{8}MS^{2} + a_{9}S^{3},
\end{equation}
where the $a_{i}$'s are polynomials in $t_{0},t_{1},t_{2},t_{3},t_{4}$ and $w$. In particular, 
\begin{align*}
a_{2}|_{t_{4} = t_{4,m_{0}}^{+}} = 0 \quad \text{and} \quad \frac{\partial a_{2}}{\partial t_{4}}|_{t_{4} = t_{4,m_{0}}^{+}} \neq 0,
\end{align*}
so $\widehat{H}_{w,\mathbf{t}}$ is a normal form describing a Hamiltonian Hopf bifurcation. Similarly, if $M$ and $N$ change roles, i.e.\ $M = \frac{1}{2}(p_{1}^{2} + p_{2}^{2})$ and $N = \frac{1}{2}(q_{1}^{2} + q_{2}^{2})$, then $a_{2}|_{t_{4} = t_{4,m_{0}}^{-}} = 0$ and $\frac{\partial a_{2}}{\partial t_{4}}|_{t_{4} = t_{4,m_{0}}^{-}} \neq 0$. The same is also true if we change $m_{0}$ with $m_{2}$.
\end{theorem}

The coefficients $a_{i}$ in \eqref{eq:nonlinear-normal-form-Gaudin} are rather large, and therefore presented only in Appendix \ref{sec:appendix-thm-coeffs}. Note that the coefficients in the theorem are in fact scaled versions of those shown in the appendix. The scaling is explained at the very end of the proof of the theorem; it concerns the fact that the coefficient of $N$ in \eqref{eq:nonlinear-normal-form-Gaudin} should be $\pm 1$.

Note that in Section \ref{sec:momentum_map} we investigate $a_{3}$ and $a_{6}$ further, and those are therefore presented there, although only for $t_{4} = t_{4,m_{0}}^{+}$. 

\begin{proof}[Proof of Theorem \ref{thm:nonlinearHHB}]
Just as in the proof of Theorem \ref{thm:linearHHB}, we follow the approach of Cushamn and Bates \cite[Section 8.2]{Cushman1997}, used to find a nonlinear normal form for the Hamiltonian Hopf bifurcation of the Lagrange top. We also use the same parameterisation we used in the proof of Theorem \ref{thm:linearHHB}.

The proof is done in $2$ steps. In the first step, we flatten the symplectic form up to sixth order. In the second step, we apply the same coordinate transformation we used to flatten $\omega$ on $H_{w,\mathbf{t}}$. Finally, we find a transformation cancelling the appearances of unwanted Hilbert generators in the transformed $H_{w,\mathbf{t}}$. However, we begin with a preliminary step.

\textit{Step 0:} We aim to find a coordinate transformation $Q : \R^{4} \to \R^{4}$, which we use to define new coordinates $(q_{1},q_{2},p_{1},p_{2})$ by $(x_{1},y_{1},x_{2},y_{2})^{\intercal} = Q(q_{1},q_{2},p_{1},p_{2})^{\intercal}$. The transformation should simultaneously 
\begin{enumerate}[(i)]
\item map the symplectic form to a symplectic form which at the singularity is standard, i.e.\ $Q^{*}\omega|_{m_{0}} = \omega_{0}$, and
\item map the degree two part of the Taylor series of $J$, inducing the $\mathbb{S}^{1}$-action, to a multiple of the Hilbert generator $S = q_{1}p_{2} - q_{2}p_{1}$ (as described in van der Meer \cite[Chapter 3]{vanderMeer1985}).
\end{enumerate}
It turns out that the coordinate transformation defined in \eqref{eq:Q} does the job:
\begin{align*}
\begin{pmatrix}
x_{1} \\ y_{1} \\ x_{2} \\ y_{2}
\end{pmatrix}
=
Q
\begin{pmatrix}
q_{1} \\ q_{2} \\ p_{1} \\ p_{2}
\end{pmatrix}
= \frac{1}{\sqrt{2}}
\begin{pmatrix}
(q_{1} + p_{2})/\sqrt{R_{1}} \\
(q_{2} - p_{1})/\sqrt{R_{1}} \\
(- q_{1} + p_{2})/\sqrt{R_{2}} \\
(q_{2} + p_{1})/\sqrt{R_{2}}
\end{pmatrix}.
\end{align*}

We define
\begin{align*}
&\xi := Q^{*}z_{1}(x_{1},y_{1}) = \sqrt{1 - \frac{1}{2R_{1}} \left( (p_{2} + q_{1})^{2} + (p_{1} - q_{2})^{2} \right)}, \\ 
&\eta := Q^{*}z_{2}(x_{2},y_{2}) = - \sqrt{1 - \frac{1}{2R_{2}} \left( (p_{2} - q_{1})^{2} + (p_{1} + q_{2})^{2} \right)}.
\end{align*}
Then, in the $(q_{1},q_{2},p_{1},p_{2})$-coordinates, we have
\begin{align*}
&\check{\omega} := Q^{*}\omega = - \left(\frac{1}{2\xi} + \frac{1}{2\eta}\right)\left(dq_{1} \wedge dq_{2} + dp_{1} \wedge dp_{2}\right) \\&\qquad\qquad\qquad+ \left(\frac{1}{2\xi} - \frac{1}{2\eta}\right)\left(dq_{1} \wedge dp_{1} + dq_{2} \wedge dp_{2}\right), \\
&\check{J} := Q^{*}J = R_{1}\xi + R_{2}\eta, \\
&\check{H}_{w,\mathbf{t}} := Q^{*}H = t_{0} (\xi + \eta)^{2} + w(t_{1}\xi + t_{2}\eta) + t_{3} \frac{q_{1}^{2} + q_{2}^{2} + p_{1}^{2} + p_{2}^{2}}{2\sqrt{R_{1}R_{2}}} + t_{4}\xi\eta.
\end{align*}
Note that at $m_{0}$, the new coordinates give us $\xi = 1$ and $\eta = -1$, which yields $\check{\omega}|_{m_{0}} = \omega_{0}$.

\textit{Step 1:} Let us introduce a small parameter $\varepsilon$, and replace the coordinates $(q,p)$ by $(\varepsilon q,\varepsilon p)$. Then the blown up symplectic form $\check{\omega}'(q,p) = \frac{1}{\varepsilon^{2}} \check{\omega}(\varepsilon q,\varepsilon p)$ and the blown up Hamiltonian $\check{H}_{w,\mathbf{t}}'(q,p) = \frac{1}{\varepsilon^{2}} \check{H}_{w,\mathbf{t}}(\varepsilon q,\varepsilon p)$ have the respective $6$-jets (where the constant term is dropped, as it does not affect the dynamics of the system, and we immediately drop the prime, simplifying the notation): 
\begin{align*}
&\check{\omega} = \check{\omega}^{0} + \varepsilon^{2} \check{\omega}^{2} + \varepsilon^{4} \check{\omega}^{4} + \varepsilon^{6} \check{\omega}^{6} + \mathcal{O}(\varepsilon^{8}), \quad 
\check{H}_{w,\mathbf{t}} = \check{H}_{w,\mathbf{t}}^{2} + \varepsilon^{2} \check{H}_{w,\mathbf{t}}^{4} + \varepsilon^{4} \check{H}_{w,\mathbf{t}}^{6} + \mathcal{O}(\varepsilon^{6}).
\end{align*}
We want to flatten the form, which we do using Proposition \ref{prop:flattening_of_sympl_form}. By Lemma \ref{lem:j-jet-Hamiltonian}, we only need to find coordinate transformations which make $\check{\omega}^{2}$ and $\check{\omega}^{4}$ vanish, and so we need to find primitives of $\check{\omega}^{2}$ and $\check{\omega}^{4}$. This we can do using the constructive Poincaré lemma given by Cushman and Bates \cite[p. 263]{Cushman1997}; they are, respectively, $\alpha^{2} = - \frac{1}{4}\iota_{A}\check{\omega}^{2}$ and $\alpha^{4} = - \frac{1}{6}\iota_{A}\check{\omega}^{4}$, where $A = - q_{1} \frac{\partial}{\partial q_{1}} - q_{2} \frac{\partial}{\partial q_{2}} - p_{1} \frac{\partial}{\partial p_{1}} - p_{2} \frac{\partial}{\partial p_{2}}$. Now a computation determines the vector fields from Proposition \ref{prop:flattening_of_sympl_form}. Let $\zeta_{\pm} = (p_{2} \pm q_{1})^{2} + (p_{1} \mp q_{2})^{2}$. Then we find
\begin{align*}
32X = &
- \left( \frac{p_{2} + q_{1}}{R_{1}} \zeta_{+} - \frac{p_{2} - q_{1}}{R_{2}} \zeta_{-} \right) \frac{\partial}{\partial q_{1}}
+ \left( \frac{p_{1} - q_{2}}{R_{1}} \zeta_{+} - \frac{p_{1} + q_{2}}{R_{2}} \zeta_{-} \right) \frac{\partial}{\partial q_{2}} \\&
- \left( \frac{p_{1} - q_{2}}{R_{1}} \zeta_{+} + \frac{p_{1} + q_{2}}{R_{2}} \zeta_{-} \right) \frac{\partial}{\partial p_{1}}
- \left( \frac{p_{2} + q_{1}}{R_{1}} \zeta_{+} + \frac{p_{2} - q_{1}}{R_{2}} \zeta_{-} \right) \frac{\partial}{\partial p_{2}},
\end{align*}
\begin{align*}
256Y = &
- \left( \frac{p_{2} + q_{1}}{R_{1}} \zeta_{+}^{2} - \frac{p_{2} - q_{1}}{R_{2}} \zeta_{-}^{2} \right) \frac{\partial}{\partial q_{1}}
+ \left( \frac{p_{1} - q_{2}}{R_{1}} \zeta_{+}^{2} - \frac{p_{1} + q_{2}}{R_{2}} \zeta_{-}^{2} \right) \frac{\partial}{\partial q_{2}} \\&
- \left( \frac{p_{1} - q_{2}}{R_{1}} \zeta_{+}^{2} + \frac{p_{1} + q_{2}}{R_{2}} \zeta_{-}^{2} \right) \frac{\partial}{\partial p_{1}}
- \left( \frac{p_{2} + q_{1}}{R_{1}} \zeta_{+}^{2} + \frac{p_{2} - q_{1}}{R_{2}} \zeta_{-}^{2} \right) \frac{\partial}{\partial p_{2}}.
\end{align*}
Pulling back $\check{\omega}$ by $\exp(\varepsilon^{2} \mathcal{L}_{X})$ and by $\exp(\varepsilon^{4} \mathcal{L}_{Y})$ then flattens it up to $6$th order. 

\textit{Step 2:} Now we apply the same transformations to $\check{H}$:
\begin{align*}
\bar{H} _{w,\mathbf{t}}
:= (\exp \varepsilon^{4} \mathcal{L}_{Y})^{*}(\exp \varepsilon^{2} \mathcal{L}_{X})^{*} \check{H}_{w,\mathbf{t}}
= \bar{H}_{w,\mathbf{t}}^{2} + \varepsilon^{2} \bar{H}_{w,\mathbf{t}}^{4} + \varepsilon^{4} \bar{H}_{w,\mathbf{t}}^{6} + \mathcal{O}(\varepsilon^{6}),
\end{align*}
where
\begin{align*}
&\bar{H}_{w,\mathbf{t}}^{2} := \check{H}_{w,\mathbf{t}}^{2}, \quad
\bar{H}_{w,\mathbf{t}}^{4} := \check{H}_{w,\mathbf{t}}^{4} + \mathcal{L}_{X} \check{H}_{w,\mathbf{t}}^{2}, \\
&\bar{H}_{w,\mathbf{t}}^{6} := \check{H}_{w,\mathbf{t}}^{6} + \mathcal{L}_{X} \check{H}_{w,\mathbf{t}}^{4} + \frac{1}{2} \mathcal{L}_{X}^{2} \check{H}_{w,\mathbf{t}}^{2} + \mathcal{L}_{Y} \check{H}_{w,\mathbf{t}}^{2}.
\end{align*}
It turns out that $\bar{H}_{w,\mathbf{t}}^{i}$, $i \in \{2,4,6\}$ are polynomial functions in the Hilbert generators
\begin{align*}
S = q_{1}p_{2} - q_{2}p_{1}, \quad
M = \frac{1}{2} (q_{1}^{2} + q_{2}^{2}), \quad
N = \frac{1}{2} (p_{1}^{2} + p_{2}^{2}), \quad 
T = q_{1}p_{1} + q_{2}p_{2},
\end{align*}
where the coefficients are polynomial functions in $t_{j}$, $j \in \{0,1,2,3,4\}$. In fact,
\begin{align*}
\bar{H}_{w,\mathbf{t}}^{2} = &
- \frac{w(R_{2}t_{1} + R_{1}t_{2}) + (R_{1}-R_{2})t_{4}}{2R_{1}R_{2}}S \\&\quad
- \frac{w(R_{2}t_{1} - R_{1}t_{2}) - 2\sqrt{R_{1}R_{2}}t_{3} - (R_{1}+R_{2})t_{4}}{2R_{1}R_{2}}N \\&\quad
- \frac{w(R_{2}t_{1} - R_{1}t_{2}) + 2\sqrt{R_{1}R_{2}}t_{3} - (R_{1}+R_{2})t_{4}}{2R_{1}R_{2}}M,
\end{align*}
and the term multiplying $M$ vanishes for $t_{4} = t_{4,m_{0}}^{+}$. For $\bar{H}_{w,\mathbf{t}}^{4}$ and $\bar{H}_{w,\mathbf{t}}^{6}$, however, there are terms multiplying $N$ and $T$, which should not be present in the normal form. After removing the $MN$ term by the equality $4MN = S^{2} + T^{2}$, we get rid of the remaining unwanted terms by another coordinate transformation. We pull back by $\exp(\varepsilon^{2} \ad_{E})$ and by $\exp(\varepsilon^{4} \ad_{F})$, where $E = e_{1} MT + e_{2} NT + e_{3} ST$ and $F = f_{1} M^{2}T + f_{2} MST + f_{3} N^{2}T + f_{4} NST + f_{5} S^{2}T + f_{6} T^{3}$. The coefficients $e_{i}$ and $f_{j}$ are determined such that, after the transformation
\begin{align*}
\widehat{H}_{w,\mathbf{t}} := &
(\exp \varepsilon^{4} \ad_{F})^{*} \left( (\exp \varepsilon^{2} \ad_{E})^{*} \bar{H}_{w,\mathbf{t}} \right) \\ = &
\bar{H}_{w,\mathbf{t}}^{2} + \varepsilon^{2} \left( \bar{H}_{w,\mathbf{t}}^{4} + \ad_{E} \bar{H}_{w,\mathbf{t}}^{2} \right) + \varepsilon^{4} \left( \bar{H}_{w,\mathbf{t}}^{6} + \ad_{E} \bar{H}_{w,\mathbf{t}}^{4} + \frac{1}{2} \ad_{E}^{2} \bar{H}_{w,\mathbf{t}}^{2} + \ad_{F} \bar{H}_{w,\mathbf{t}}^{2} \right),
\end{align*}
the terms multiplying $\varepsilon^{2}$ and $\varepsilon^{4}$ are void of $N$ and $T$ terms. Recall that $\ad_{G}H = \{G,H\}$ for functions $G,H$. Thus, this becomes
\begin{align*}
\widehat{H}_{w,\mathbf{t}} := &
\bar{H}_{w,\mathbf{t}}^{2} + \varepsilon^{2} \left( \bar{H}_{w,\mathbf{t}}^{4} + \{E,\bar{H}_{w,\mathbf{t}}^{2}\} \right) \\&+ \varepsilon^{4} \left( \bar{H}_{w,\mathbf{t}}^{6} + \{E,\bar{H}_{w,\mathbf{t}}^{4}\} + \frac{1}{2} \{E,\{E,\bar{H}_{w,\mathbf{t}}^{2}\}\} + \{F, \bar{H}_{w,\mathbf{t}}^{2}\} \right).
\end{align*}
Using the Leibniz formula $\{G_{1}G_{2},H\} = G_{1}\{G_{2},H\} + G_{2}\{G_{1},H\}$, yields the coefficients in Appendix \ref{sec:appendix-EF-coeffs}. This, in turn, gives us the following Hamiltonian,
\begin{equation*}
\widehat{H}_{w,\mathbf{t}} = \tilde{a}_{1}S + bN + \tilde{a}_{2}M + \tilde{a}_{3}M^{2} + \tilde{a}_{4}MS + \tilde{a}_{5}S^{2} + \tilde{a}_{6}M^{3} + \tilde{a}_{7}M^{2}S + \tilde{a}_{8}MS^{2} + \tilde{a}_{9}S^{3},
\end{equation*}
where the coefficients $\tilde{a}_{i}$ and $b$ are given in Appendix \ref{sec:appendix-thm-coeffs} (and $\tilde{a}_{1}$, $b$, and $\tilde{a}_{2}$ are also given in the expression for $\bar{H}_{w,\mathbf{t}}^{2}$). The coefficient of $N$ in the normal form \eqref{eq:HHB-nonlin-normal-form-def} should be either $1$ or $-1$. By scaling the Hamiltonian by $b^{-1}$, i.e.\ by multiplying $b^{-1}\tilde{a}_{i}$, for $i \in \{1,\dots,9\}$, we get the coefficient of $N$ to be $1$. 
\if We obtain the desired coefficient by a symplectic coordinate change: for $j \in \{1,2\}$, $p_{j}^{2} \to p_{j}^{2}/\tilde{a}_{2}$, such that $N = \frac{1}{2}(p_{1}^{2} + p_{2}^{2}) \to N/a_{2}$. To keep the coordinates canonical, we similarly must take $q_{j}^{2} \to q_{j}^{2}/\tilde{a}_{2}$, such that $M = \frac{1}{2}(q_{1}^{2} + q_{2}^{2}) \to \tilde{a}_{2}M$. The remaining Hilbert generator, $S$, does not change, as we have $q$-coordinates multiplying $p$-coordinates. Hence, we also get $M^{2} \to \tilde{a}_{2}^{2}M^{2}$, $MS \to \tilde{a}_{2}MS$, and so on. \fi 
This yields the normal form from the statement of the theorem.
\end{proof}

\section{Analysis of the momentum map} \label{sec:momentum_map}

In this section we study the dynamics and geometry of the momentum map by analysing the non-linear normal form for $H_{w,\mathbf{t}}$ that we found in Theorem \ref{thm:nonlinearHHB}. The corollaries \ref{cor:subcritical-HHB-rational}, \ref{cor:subcritical-HHB-trig}, \ref{cor:degenerate-HHB-rat}, and \ref{cor:degenerate-HHB-trig} follow immediately from the coefficients in \eqref{eq:nonlinear-normal-form-Gaudin}. We only give the results for the bifurcation at $t_{4} = t_{4,m_{0}}^{+}$, but similar results hold for the other bifurcation values $t_{4} = t_{4,m_{0}}^{-}$ and $t_{4} = t_{4,m_{2}}^{\pm}$.

Note that at $t_{4} = t_{4,m_{0}}^{+}$, we have
\begin{align*}
a_{3} &= \frac{2 \left(R_1-R_2\right){}^2 \left(R_1+R_2\right) t_0+2 R_1 R_2 w \left(R_1 t_2-R_2 t_1\right)+(R_{1}-R_{2})^{2}\sqrt{R_{1}R_{2}}t_{3}}{8 \left(R_1 R_2\right){}^{3/2} \left(R_1+R_2\right) t_3}.
\end{align*}
If $a_{3} > 0$, then the bifurcation is supercritical, and if $a_{3} < 0$ it is subcritical. The rational Gaudin model, given by $t_{0} = 0$ and $w = 1$, goes through a subcritical Hamiltonian Hopf bifurcation for the following requirements on $R_{1}$, $R_{2}$, $t_{1}$, $t_{2}$, and $t_{3}$:

\begin{corollary} \label{cor:subcritical-HHB-rational}
The point $m_{0}$ in the rational Gaudin model, with Hamiltonian given by
\begin{equation*}
H_{1,\mathbf{t}} = t_{1}z_{1} + t_{2}z_{2} + t_{3}(x_{1}x_{2} + y_{1}y_{2}) + t_{4}z_{1}z_{2},    
\end{equation*}
goes through a subcritical Hamiltonian Hopf bifurcation at $t_{4} = t_{4,m_{0}}^{+}$ if either
\begin{enumerate}[(i)]
    \item $R_{1} = R_{2}$, and either $t_{1} < t_{2}$ and $t_{3} < 0$, or $t_{1} > t_{2}$ and $t_{3} > 0$, or
    \item $(R_{1}-R_{2})^{2}t_{3} < 2\sqrt{R_{1}R_{2}}(R_{1}t_{2} - R_{2}t_{1})$ and $t_{3} < 0$, or
    \item $(R_{1}-R_{2})^{2}t_{3} < -2\sqrt{R_{1}R_{2}}(R_{1}t_{2} - R_{2}t_{1})$ and $t_{3} > 0$.
\end{enumerate}
\end{corollary}

One can determine the requirements for the bifurcation to be supercritical by considering the cases not mentioned in the corollary. There is however also the case $a_{3} = 0$ which is neither supercritical nor subcritical; the bifurcation is degenerate in that case, and dealt with in Corollary \ref{cor:degenerate-HHB-rat}.

Figure \ref{fig:rat-subcrit} shows what Corollary \ref{cor:subcritical-HHB-rational} means for the image of the momentum map. In the figure one can see $m_{0}$ going through a subcritical Hamiltonian Hopf bifurcation at $t_{4} = t_{4,m_{0}}^{+}$ as it satisfies the criteria \textit{(i)} of Corollary \ref{cor:subcritical-HHB-rational}. Likewise we could have written conditions for $m_{2}$, for which $a_{3} > 0$, which would tell us that the bifurcation $m_{2}$ goes through in Figure \ref{fig:rat-subcrit} is indeed supercritical.

\begin{figure}[tb]
\def\scale{0.69}
\centering
\begin{tabular}{cc}
    \subfloat[At $t_{4} = 0$, $m_{0}$ and $m_{2}$ are two focus-focus points.]{\includegraphics[scale=\scale]{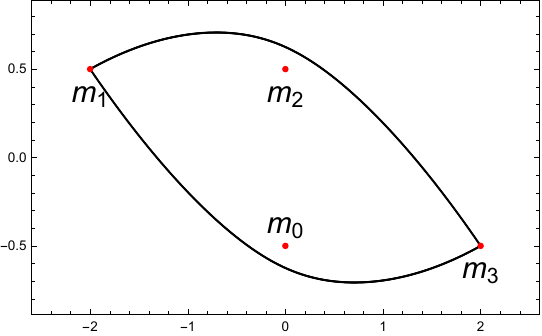}} & 
    \subfloat[At $t_{4} = t_{4,m_{2}}^{+} = -0.25$, $m_{2}$ undergoes a supercritical bifurcation.]{\includegraphics[scale=\scale]{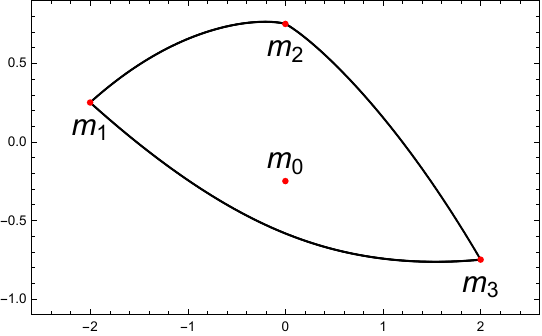}} \\
    \subfloat[At $t_{4} = t_{4,m_{0}}^{+} = -0.75$, $m_{0}$ undergoes a subcritical bifurcation.]{\includegraphics[scale=\scale]{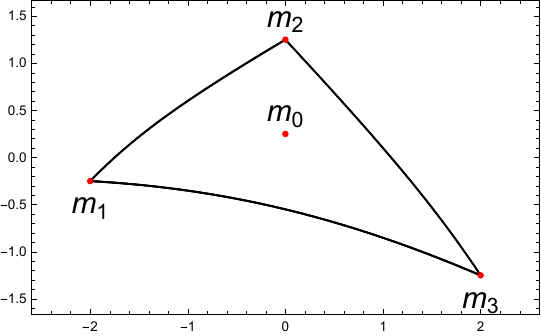}} & 
    \subfloat[At $t_{4} = -1.5$, $m_{0}$ is sitting on a flap.]{\includegraphics[scale=\scale]{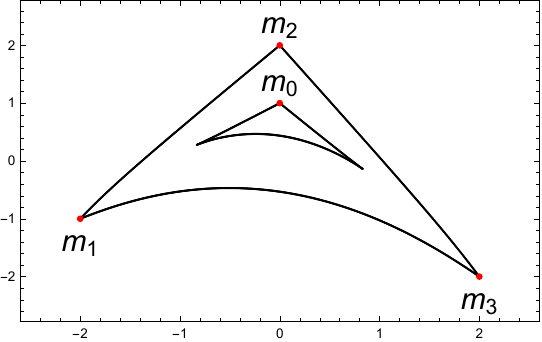}}
\end{tabular}
\caption{The figures shows the image of the momentum map of a rational Gaudin model for various choices of $t_{4}$. Here we have chosen $R_{1} = 1$, $R_{2} = 1$, $t_{1} = -0.5$, $t_{2} = 0$, and $t_{3} = -0.5$.}
\label{fig:rat-subcrit}
\end{figure}

We present a similar result for the trigonometric Gaudin model, given by $w = 0$, depending on $R_{1}$, $R_{2}$, $t_{0}$, and $t_{3}$:

\begin{corollary} \label{cor:subcritical-HHB-trig}
The point $m_{0}$ in the trigonometric Gaudin model, with Hamiltonian given by
\begin{equation*}
H_{0,\mathbf{t}} = t_{0}(z_{1} + z_{2})^{2} + t_{3}(x_{1}x_{2} + y_{1}y_{2}) + t_{4}z_{1}z_{2},    
\end{equation*} goes through a subcritical Hamiltonian Hopf bifurcation at $t_{4} = t_{4,m_{0}}^{+}$ if $R_{1} \neq R_{2}$, and either
\begin{enumerate}[(i)]
    \item $2(R_{1}+R_{2})t_{0} > \sqrt{R_{1}R_{2}}t_{3}$ and $t_{3} < 0$, or
    \item $2(R_{1}+R_{2})t_{0} > -\sqrt{R_{1}R_{2}}t_{3}$ and $t_{3} > 0$.
\end{enumerate}
\end{corollary}

Just as in the rational case, one can determine the requirements for the bifurcation to be supercritical by considering the cases not mentioned in the corollary, except for the special case $a_{3} = 0$. The degenerate case also happens for the trigonometric model, as described in Corollary \ref{cor:degenerate-HHB-trig}. 

In Figure \ref{fig:trig-col} we present a situation where both $m_{0}$ and $m_{2}$ go through subcritical Hamiltonian Hopf bifurcations. The situation is given by a trigonometric Gaudin model satisfying criteria \textit{(ii)} of Corollary \ref{cor:subcritical-HHB-trig}. In fact, the figure suggests that there are more bifurcations taking place than only Hamiltonian Hopf bifurcations, which we list below. To locate the exact values for the bifurcations, it is useful to consider the reduced Hamiltonian. Furthermore, we will represent the Hamiltonian in terms of $J = R_{1}z_{1} + R_{2}z_{2}$, $K = R_{1}z_{1} - R_{2}z_{2}$, and $\xi = x_{1}x_{2} + y_{1}y_{2} + z_{1}z_{2}$, as Sadovski{\'i} and Zhilinski{\'i} did \cite{Sadovskii1999}. Then, with $R_{1} = 1$, $R_{2} = 2$, $w = 0$, $t_{0} = -\frac{1}{2}$ and $t_{3} = \frac{1}{2}$, as in Figure \ref{fig:trig-col}, the reduced Hamiltonian at $J = j$ is
\begin{equation*}
\mathcal{H}_{j,t_{4}}(K,\xi) = -\frac{1}{2} \left( \frac{j+K}{2} + \frac{j-K}{4} \right)^{2} + \frac{1}{2} \left( \xi - \frac{1}{8}(j+K)(j-K) \right) + \frac{t_{4}}{8} (j+K)(j-K).
\end{equation*}
Sadovski{\'i} and Zhilinski{\'i} also introduced a variable $\sigma$, which is the $z$-projection of the vector product of $(x_{1},y_{1},z_{1})$ with $(x_{2},y_{2},z_{2})$, $\sigma = (0,0,1) \cdot \left( (x_{1},y_{1},z_{1}) \times (x_{2},y_{2},z_{2}) \right)$. One can show that these variables are related by
\begin{equation} \label{eq:invariants-relation}
\sigma^{2} + \left( \xi - \frac{(j+K)(j-K)}{4R_{1}R_{2}} \right)^{2} = \left( 1- \frac{(j+K)^{2}}{4R_{1}^{2}} \right) \left( 1- \frac{(j-K)^{2}}{4R_{2}^{2}} \right).
\end{equation}
It turns out that it is more convenient to do the computations with $J,K,\sigma$ variables. In these variables, the $1$-jet of the reduced Hamiltonian is
\begin{align*}
(D^{1}_{K_{0},\sigma_{0}}\mathcal{H}_{j,t_{4}})(K,\sigma) = 4\sigma f_{j}(K_{0},\sigma_{0}) (\sigma-\sigma_{0}) - \frac{1}{8} F_{j,t_{4}}^{\pm}(K_{0},\sigma_{0}) (K-K_{0}),
\end{align*}
where
\begin{align*}
f_{j}(K_{0},\sigma_{0}) &= \left( 64 + j^{4} - 24jK_{0} - 20K_{0}^{2} + K_{0}^{4} - 2j^{2}(10+K_{0}^{2}) - 64\sigma_{0}^{2} \right)^{-1/2}, \\
F_{j,t_{4}}^{\pm}(K_{0},\sigma_{0}) &= \frac{3j}{2} + 2 \left( t_{4}+\frac{1}{4} \right) K_{0} \pm \frac{1}{4}f_{j}(K_{0},\sigma_{0})(24j + 40K_{0} + 4j^{2}K_{0} - 4K_{0}^{3}).
\end{align*}
The choice of plus or minus comes from the squares in \eqref{eq:invariants-relation}. We note that all singularities have $\sigma = 0$. The other singularities are then given by $F_{j}(K_{0},0) = 0$. Furthermore, we are in particular interested in knowing when the singularities are degenerate. Hence, we want to study the Hessian. We find that the Hessian at $K = K_{0}$ and $\sigma = 0$ is given by
\begin{equation*}
\text{Hess}_{\mathcal{H}_{j,t_{4}}}(K_{0},0) = 
\begin{pmatrix}
G_{j,t_{4}}(K_{0}) & 0 \\
0 & 4f_{j}(K_{0},0)
\end{pmatrix},
\end{equation*}
where
\begin{align*}
G_{j,t_{4}}(K_{0}) 
= &\frac{(f_{j}(K_{0},0))^{3}}{4} \left( 6j + 10K_{0} + j^{2}K_{0} - K_{0}^{3} \right)^{2} \\&+ f_{j}(K_{0},0) \left( \frac{5}{4} + \frac{j^{2}}{8} - \frac{3K_{0}^{2}}{8} \right) - \frac{t_{4}}{4} - \frac{1}{16}.
\end{align*}

\begin{figure}[p]
\def\scale{0.5}
\centering
\begin{tabular}{ccc}
    \subfloat[At $t_{4} = -\frac{3}{2}$, a pleat appears in the image of the momentum map. \label{fig:trig-pleat1}]{\includegraphics[scale=\scale]{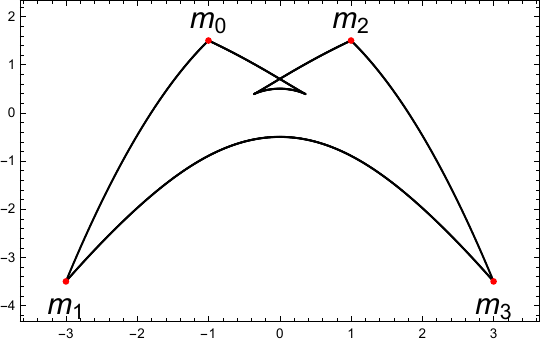}} & 
    \subfloat[At $t_{4} = t_{4,m_{0}}^{-} = t_{4,m_{2}}^{-}$, two supercritical Hamiltonian Hopf bifurcations take place. \label{fig:trig-super}]{\includegraphics[scale=\scale]{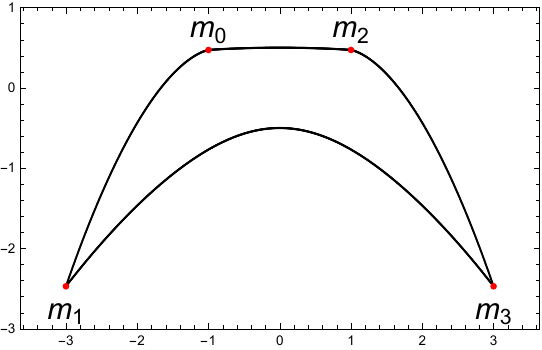}} & 
    \subfloat[At $t_{4} = \frac{3}{8}$, $m_{0}$ and $m_{2}$ have both become focus-focus points. At this point, a bifurcation a pleat occurs. \label{fig:trig-pleatbif1}]{\includegraphics[scale=\scale]{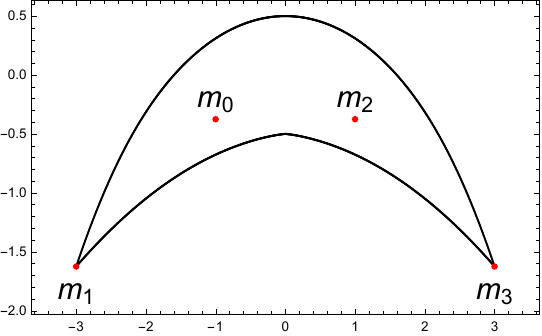}} \\
    \subfloat[At $t_{4} = t_{4,m_{0}}^{+} = t_{4,m_{2}}^{+}$, two subcritical Hamiltonian Hopf bifurcations take place. \label{fig:trig-sub}]{\includegraphics[scale=\scale]{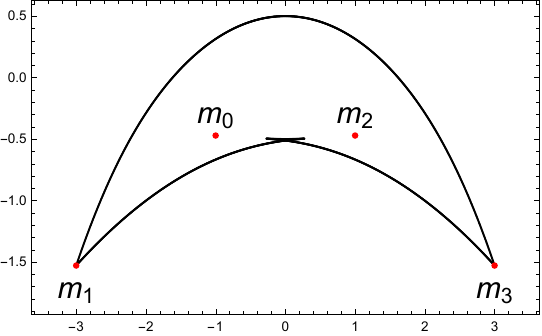}} & 
    \multicolumn{2}{c}{\subfloat[At $t_{4} = 0.495$, $m_{0}$ and $m_{2}$ have gone through subcritical Hamiltonian Hopf bifurcation, and two flaps have appeared, on which $m_{0}$ and $m_{2}$ are sitting. \label{fig:trig-flap}]{\includegraphics[scale=\scale]{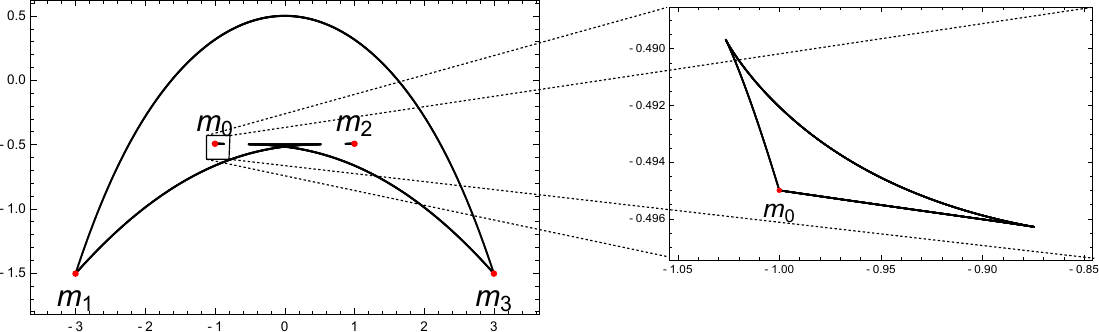}}} \\
    \subfloat[At $t_{4} = 0.58$, the two flaps have collided with the pleat. \label{fig:trig-flap-pleat-bif1}]{\includegraphics[scale=\scale]{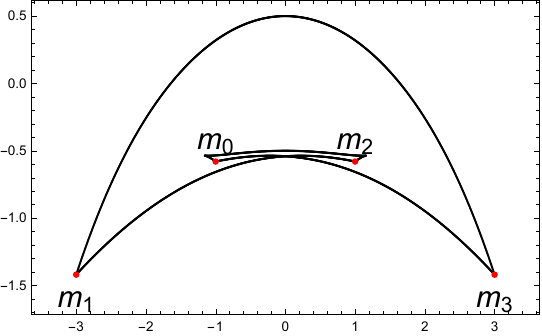}} &
    \subfloat[At $t_{4} = 1$, we can clearly see a pleat has been formed after the collision. \label{fig:trig-pleat2}]{\includegraphics[scale=\scale]{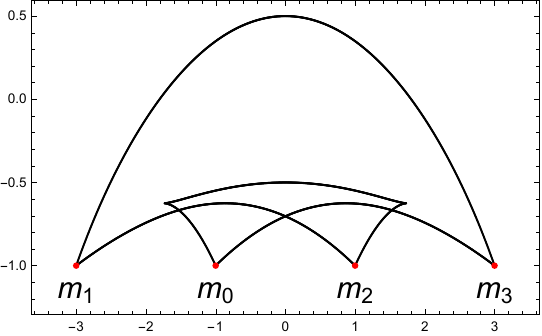}} &
    \subfloat[At $t_{4} = \frac{3}{2}$, the pleat splits in three. \label{fig:trig-newbif}]{\includegraphics[scale=\scale]{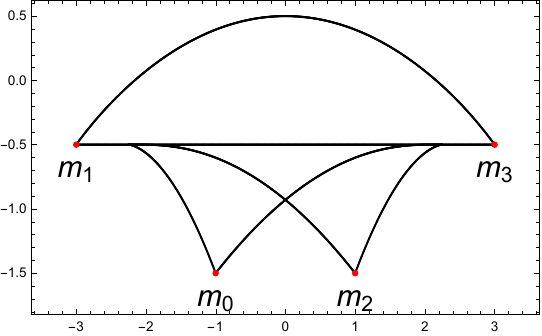}} \\
    \multicolumn{2}{c}{\subfloat[At $t_{4} = 1.51$, the pleat has split into three smaller pleats. \label{fig:trig-flap-pleat-bif2}]{\includegraphics[scale=\scale]{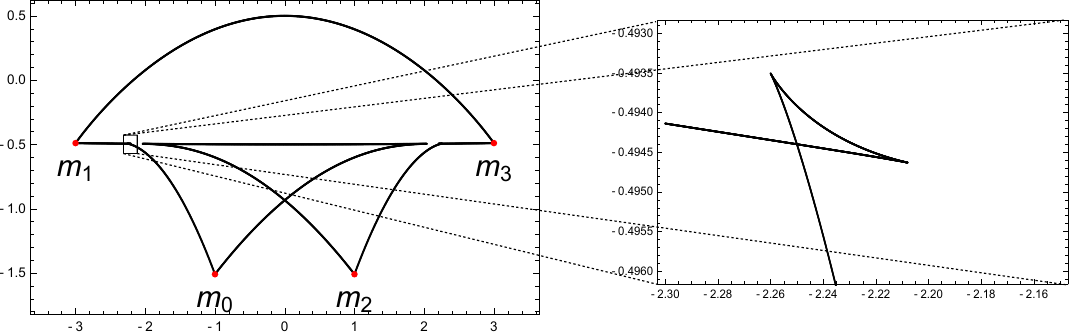}}} &
    \subfloat[At $t_{4} = 2$, the two smallest pleats have disappeared, leaving only one pleat. \label{fig:trig-pleat3}]{\includegraphics[scale=\scale]{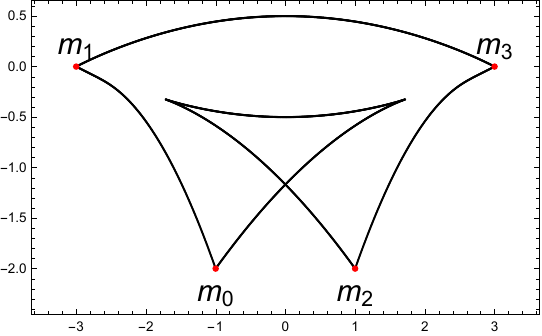}}
\end{tabular}
\caption{The figures shows the image of the momentum map of a trigonometric Gaudin model for various choices of $t_{4}$. Here we have chosen $R_{1} = 1$, $R_{2} = 2$, $t_{0} = -\frac{1}{2}$, and $t_{3} = \frac{1}{2}$, which makes $t_{4,m_{0}}^{\pm} = t_{4,m_{2}}^{\pm} = \pm \frac{\sqrt{2}}{3} \approx \pm 0.47$.}
\label{fig:trig-col}
\end{figure}

\begin{figure}[tb]
    \centering
    \includegraphics[scale=0.7]{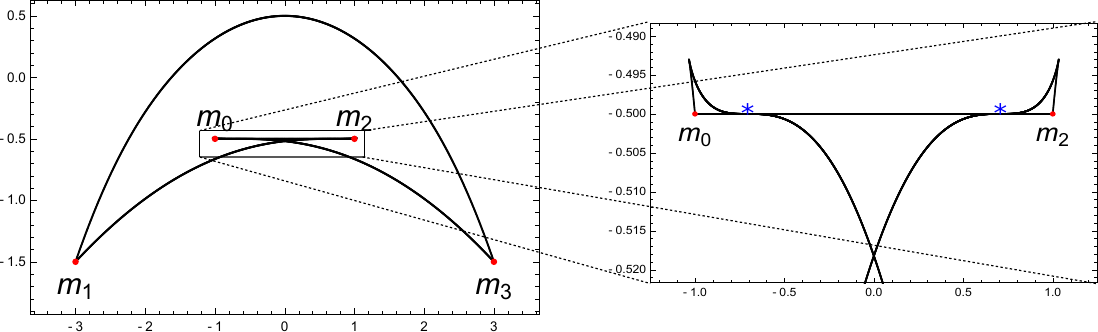}
    \caption{Two collisions, both involving two cusps, denoted by the blue stars. The setup is as in Figure \ref{fig:trig-col}, i.e.\ $R_{1} = 1$, $R_{2} = 2$, $w = 0$, $t_{0} = -\frac{1}{2}$, $t_{3} = \frac{1}{2}$, and $t_{4} = \frac{1}{2}$.}
    \label{fig:pleat-flap-col}
\end{figure}

Now, let us use the reduced Hamiltonian to understand more about the following bifurcations:
\begin{enumerate}
    \item We see in Figures \ref{fig:trig-pleat1} and \ref{fig:trig-super} that a pleat ceases to exist. By the symmetry of the system, we can see that the bifurcation takes place at $J = 0$. To find out when the bifurcation takes place, we simply need to solve the system
    \begin{align} \label{eq:sys-other-bif}
    \begin{cases}
    F_{0,t_{4}}(K_{0},0) = 0, \\
    \det(\text{Hess}_{\mathcal{H}_{0,t_{4}}}(K_{0},0)) = 0.
    \end{cases}
    \end{align} 
    It turns out that $t_{4} = -\frac{7}{8}$ and $K_{0} = 0$ satisfies this system. Thus, the bifurcation takes place at $t_{4} = -\frac{7}{8}$ for $J = 0$ and $\mathcal{H}_{0,-\frac{7}{8}}(0,0) = \frac{1}{2}$. The same method applies to the bifurcation discussed in Point \ref{itm:t4=3/8}. This bifurcation was for instance discussed by Efstathiou and Sugny \cite{EfstathiouSugny}.
    \item At $t_{4} = t_{4,m_{0}}^{-} = t_{4,m_{2}}^{-} = -\frac{\sqrt{2}}{3}$, as illustrated in Figure \ref{fig:trig-super}, the points $m_{0}$ and $m_{2}$ undergo supercritical Hamiltonian Hopf bifurcations. This happens for $H_{w,\mathbf{t}} = \frac{\sqrt{2}}{3}$, and $J = -1$ for $m_{0}$ and $J = 1$ for $m_{2}$.
    \item \label{itm:t4=3/8} A new pleat is born in Figure \ref{fig:trig-pleatbif1}, at $t_{4} = \frac{3}{8}$, and for which $J = 0$ and $\mathcal{H}_{0,\frac{3}{8}}(0,0) = -\frac{1}{2}$.
    \item At $t_{4} = t_{4,m_{0}}^{+} = t_{4,m_{2}}^{+} = \frac{\sqrt{2}}{3}$, as illustrated in Figure \ref{fig:trig-sub}, the points $m_{0}$ and $m_{2}$ undergo subcritical Hamiltonian Hopf bifurcations. This happens for $H_{w,\mathbf{t}} = -\frac{\sqrt{2}}{3}$, and $J = -1$ for $m_{0}$ and $J = 1$ for $m_{2}$.
    \item \label{itm:t4=1/2} At $t_{4} = \frac{1}{2}$, somewhere between Figures \ref{fig:trig-flap} and \ref{fig:trig-flap-pleat-bif1}, the two flaps collide with the pleat generated in Figure \ref{fig:trig-pleatbif1}, see also Figure \ref{fig:pleat-flap-col}. In this case, we solve the system of equations \eqref{eq:sys-other-bif} also with respect to $J$, and find that it is true for $t_{4} = \frac{1}{2}$, $J = \pm \frac{\sqrt{2}}{2}$, and $K = \mp \frac{3\sqrt{2}}{2}$. As this is a collision of cusps, there must be singularities connecting to the degenerate singularity on both sides. We consider $F_{\pm \frac{\sqrt{2}}{2}+\varepsilon,\frac{1}{2}}(\mp \frac{3\sqrt{2}}{2}+\delta,0) = 0$, and find that for all $\varepsilon$ small enough, there exist $\delta$ solving the equation. Thus, this is the bifurcation point. It corresponds to $\mathcal{H}_{\pm \frac{\sqrt{2}}{2},\frac{1}{2}}(\mp \frac{3\sqrt{2}}{2},0) = - \frac{1}{2}$. The same method applies to the bifurcation discussed in Point \ref{itm:t4=3/2}. A similar bifurcation was also studied by Gullentops and Hohloch \cite[Example 7.7]{Gullentops2022}. Note that the new pleat has a different constellation to the one in Figure \ref{fig:trig-pleat1}. That is, now there are curves connecting $m_{1}$ to $m_{2}$, and $m_{3}$ to $m_{0}$, as is best illustrated in Figure \ref{fig:trig-pleat2}. 
    \item \label{itm:t4=3/2} At $t_{4} = \frac{3}{2}$, illustrated by Figure \ref{fig:trig-newbif}, the pleat splits into three smaller pleats. This takes place at $J = \pm \frac{3\sqrt{2}}{2}$, $\mathcal{H}_{\pm \frac{3\sqrt{2}}{2},\frac{1}{2}}(\mp \frac{\sqrt{2}}{2},0) = -\frac{1}{2}$.
    \item[(7?)] We conjecture that the two small pleats on the left and right sides in Figure \ref{fig:trig-flap-pleat-bif2} at some point cease to exist. Note that we have returned to the original constellation, where there are curves connecting $m_{1}$ to $m_{0}$, and $m_{3}$ to $m_{2}$.
\end{enumerate}

\if
Let us now consider the bifurcation depicted in Figure \ref{fig:trig-col-pleat-bif-2} a bit further, by studying the level sets of the reduced Hamiltonian on the reduced phase space. As in Sadovski\'{i} and Zhilinski\'{i} \cite{Sadovskii1999}, we introduce $J_{z} = J = R_{1}z_{1} + R_{2}z_{2}$, $K_{z} = R_{1}z_{1} - R_{2}z_{2}$, and $\xi = x_{1}x_{2} + y_{1}y_{2} + z_{1}z_{2}$. Thus, after noting that we may write
\begin{align*}
z_{1} = \frac{J_{z}+K_{z}}{2R_{1}}, \quad \text{and} \quad z_{2} = \frac{J_{z}-K_{z}}{2R_{2}},
\end{align*}
then for a fixed value of $J_{z}$, the reduced Hamiltonian is
\begin{align*}
\mathcal{H}^{J_{z}}_{w,\mathbf{t}}(K_{z},\xi) = &t_{0} \left( \frac{J_{z}+K_{z}}{2R_{1}} + \frac{J_{z}-K_{z}}{2R_{2}} \right)^{2} + w \left( t_{1} \frac{J_{z}+K_{z}}{2R_{1}} + t_{2} \frac{J_{z}-K_{z}}{2R_{2}} \right) \\&+ t_{3} \left( \xi - \frac{J_{z}^{2}-K_{z}^{2}}{4R_{1}R_{2}} \right) + t_{4} \frac{J_{z}^{2}-K_{z}^{2}}{4R_{1}R_{2}}.
\end{align*}
As the other integral of motion, $J$, is the same as in Sadovski\'{i} and Zhilinski\'{i} \cite{Sadovskii1999}, the reduced phase space is also the same. In Figure \ref{fig:pleatBif}, level sets of the reduced Hamiltonian has been drawn for three values of $J_{z}$; one to the left of the point of bifurcation, one at the point of bifurcation, and one to the right of the the point of bifurcation. 

\begin{figure}[tb]
\centering
\includegraphics[scale=1.5]{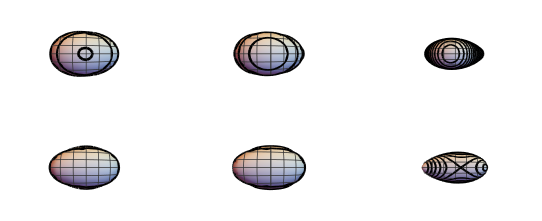}
\caption{We have the same configurations as in Figure \ref{fig:trig-col-pleat-bif-2}. The figure shows the reduced phase space for various values of $J$; let $j := -2.25$, then the left pair (top shows reduced phase space from above, bottom from below) is at $J = j-0.1$, middle pair at $J = j$, and right pair at $J = -0.5$. The thick black lines display level sets on the reduced phase space.}
\label{fig:pleatBif}
\end{figure}
\fi

Finally we consider the case when $a_{3} = 0$. We solve $a_{3} = 0$ for $t_{3}$, and find that this happens when
\begin{equation*}
t_{3} = t_{3,m_{0}}^{+} := -\frac{2((R_{1}+R_{2})(R_{1}-R_{2})^{2}t_{0} - wR_{1}R_{2}(R_{2}t_{1} - R_{1}t_{2})}{\left(R_1-R_2\right){}^2 \sqrt{R_1 R_2}}.
\end{equation*}
At $t_{4} = t_{4,m_{0}}^{+}$ and $t_{3} = t_{3,m_{0}}^{+}$ we find that the coefficient of $M^{3}$ in \eqref{eq:nonlinear-normal-form-Gaudin} is
\begin{equation*}
a_{6} = \frac{\left(R_1-R_2\right){}^2 \left(R_1 R_2 w \left(R_1 \left(t_1+2 t_2\right)-R_2 \left(2 t_1+t_2\right)\right)+\left(R_1+R_2\right) \left(R_1-R_2\right){}^2 t_0\right)}{384 \left(R_1 R_2\right){}^{2} \left(R_1 R_2 w \left(R_1 t_2-R_2 t_1\right)+\left(R_1+R_2\right) \left(R_1-R_2\right){}^2 t_0\right)}.
\end{equation*}
By definition, the Hamiltonian Hopf bifurcation is degenerate if $a_{3} = 0$ and $a_{6} \neq 0$, yielding the following results:

\begin{corollary} \label{cor:degenerate-HHB-rat}
The point $m_{0}$ in the rational Gaudin model goes through a degenerate Hamiltonian Hopf bifurcation at $t_{4} = t_{4,m_{0}}^{+}$ if $t_{3} = \frac{2 \sqrt{R_1 R_2} \left(R_2 t_1-R_1 t_2\right)}{\left(R_1-R_2\right){}^2}$, $R_{1} \neq R_{2}$, $R_{1}t_{2} \neq R_{2}t_{1}$, and $R_{1}(t_{1} + 2t_{2}) \neq R_{2}(2t_{1} + t_{2})$.
\end{corollary}

\begin{corollary} \label{cor:degenerate-HHB-trig}
The point $m_{0}$ in the trigonometric Gaudin model goes through a degenerate Hamiltonian Hopf bifurcation at $t_{4} = t_{4,m_{0}}^{+} = 4t_{0}$ if $t_{3} = -\frac{2(R_{1} + R_{2})t_{0}}{\sqrt{R_{1}R_{2}}}$ and $R_{1} \neq R_{2}$.
\end{corollary}

In Figure \ref{fig:trig-degenerate} we see the image of the momentum map of a trigonometric Gaudin model going through a degenerate Hamiltonian Hopf bifurcation. 

\begin{figure}[tb]
\def\scale{0.5}
\centering
\subfloat[At $t_{4} = 1$, $m_{0}$ and $m_{2}$ are two focus-focus points.]{\includegraphics[scale=\scale]{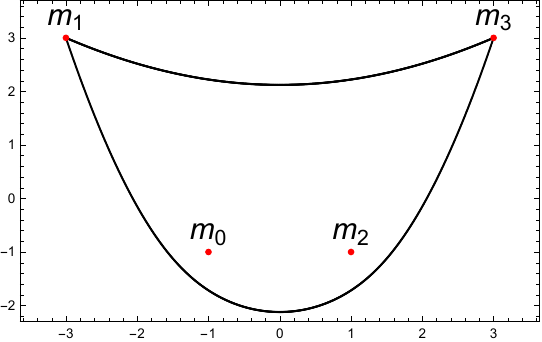}}
\hfil
\subfloat[At $t_{4} = t_{4,m_{0}}^{+} = t_{4,m_{2}}^{+} = 2$, $m_{0}$ and $m_{2}$ undergo degenerate bifurcation.]{\includegraphics[scale=\scale]{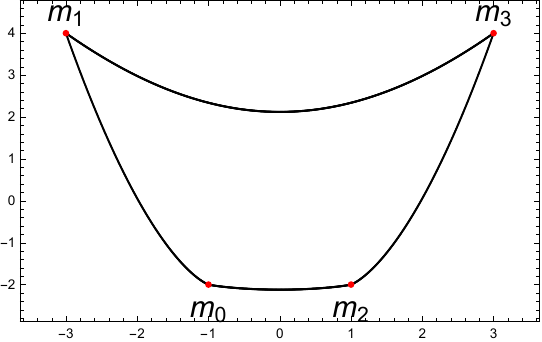}}
\hfil
\subfloat[At $t_{4} = 2.5$, $m_{0}$ and $m_{2}$ are two elliptic-elliptic points.]{\includegraphics[scale=\scale]{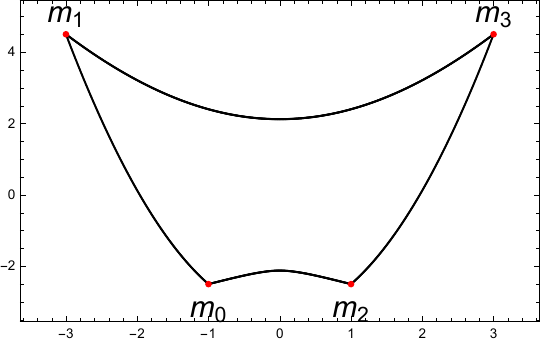}}
\caption{The figures shows the image of the momentum map of a trigonometric Gaudin model for various choices of $t_{4}$ going through a degenerate Hamiltonian Hopf bifurcation. Here we have chosen $R_{1} = 1$, $R_{2} = 2$, $t_{0} = \frac{1}{2}$, which makes $t_{3,m_{0}}^{+} = \frac{3}{\sqrt{2}}$ and $t_{4,m_{0}}^{+} = 2$.}
\label{fig:trig-degenerate}
\end{figure}

We could go even further, and look for possible $2$-degenerate Hamiltonian Hopf bifurcations. We find that, in the rational model, then $a_{6}$ vanishes for
\begin{equation*}
t_{2} = \frac{\left(R_1-2 R_2\right) t_1}{R_2-2 R_1}.
\end{equation*}
However, without computing the normal form up to fourth order, we cannot tell for sure that this satisfies the definition. If it would be $2$-degenerate, then we could even look for $3$-degenerate bifurcations in this case, as we still may vary $t_{1}$.

In the trigonometric model, on the other hand, we get that $a_{6}$ vanishes for $t_{0} = 0$. It could still be $2$-degenerate, but there will be no higher degenerate Hamiltonian Hopf bifurcations in this case.
\begin{appendices} 

\section{Coefficients from Theorem \ref{thm:nonlinearHHB}} \label{sec:appendix-thm-coeffs}

In this appendix we give the coefficients $\tilde{a}_{i}$, $i \in \{1,\dots,9\}$, and $b$, from the proof of Theorem \ref{thm:nonlinearHHB}.

\begin{align*}
\tilde{a}_{1} = - \frac{w(R_{2}t_{1} + R_{1}t_{2}) + (R_{1}-R_{2})t_{4}}{2R_{1}R_{2}}.
\end{align*}
\begin{align*}
b= - \frac{w(R_{2}t_{1} - R_{1}t_{2}) - 2\sqrt{R_{1}R_{2}}t_{3} - (R_{1}+R_{2})t_{4}}{2R_{1}R_{2}}.
\end{align*}
\begin{align*}
\tilde{a}_{2} = - \frac{w(R_{2}t_{1} - R_{1}t_{2}) + 2\sqrt{R_{1}R_{2}}t_{3} - (R_{1}+R_{2})t_{4}}{2R_{1}R_{2}}.
\end{align*}
\begin{align*}
\tilde{a}_{3} = \frac{\tilde{a}_{3}^{n}}{\tilde{a}_{3}^{d}}, \quad \text{where}
\end{align*}
\begin{align*}
\tilde{a}_{3}^{n} = &R_2^2 R_1^2 \big(2 t_0 \big(t_1^2+4 t_2 t_1+t_2^2\big) w^2+4 t_4^2 \big(t_1 w-t_2 w-t_0\big)+4 t_4 w \big(t_2 t_1 w+t_0 t_1-t_0 t_2\big) \\&+t_3^2 \big(-3 t_1 w+3 t_2 w-8 t_0+2 t_4\big)-4 t_4^3\big)+2 R_1^4 t_0 \big(t_2 w+t_4\big){}^2+2 R_2^4 t_0 \big(t_4-t_1 w\big){}^2 \\&+R_2 R_1^3 \big(t_3^2 \big(3 t_2 w+4 t_0+3 t_4\big)-2 \big(t_2 w+t_4\big) \big(t_2 t_4 w+2 t_0 \big(t_1+t_2\big) w+t_4^2\big)\big) \\&+R_2^3 R_1 \big(t_3^2 \big(-3 t_1 w+4 t_0+3 t_4\big)-2 \big(t_4-t_1 w\big) \big(-t_1 t_4 w-2 t_0 \big(t_1+t_2\big) w+t_4^2\big)\big),
\end{align*}
\begin{align*}
\tilde{a}_{3}^{d} = 3 R_1^2 R_2^2 \big(R_2 \big(t_4-t_1 w\big)+R_1 \big(t_2 w+t_4\big)+2 \sqrt{R_1 R_2} t_3\big){}^2.
\end{align*}
\begin{align*}
\tilde{a}_{4} = -\frac{\big(R_1-R_2\big) \big(2wt_{0}(R_{1}^{2}t_{2}+R_{1}R_{2}(t_{2}-t_{1})-R_{2}^{2}t_{1}) + R_{1}R_{2}t_{3}^{2} + 2(R_{1}+R_{2})^{2}t_{0}t_{4}\big)}{2 R_1^2 R_2^2 \big(R_2 \big(t_4-t_1 w\big)+R_1 \big(t_2 w+t_4\big)+2 \sqrt{R_1 R_2} t_3\big)}.
\end{align*}
\begin{align*}
\tilde{a}_{5} = \frac{2 R_1^2 t_0+R_2 R_1 \big(2 t_0+t_4\big)+2 R_2^2 t_0}{6 R_1^2 R_2^2}.
\end{align*}
\begin{align*}
\tilde{a}_{6} = \frac{\tilde{a}_{6}^{n}}{\tilde{a}_{6}^{d}}, \quad \text{where}
\end{align*}
\begingroup
\allowdisplaybreaks
\begin{align*}
\tilde{a}_{6}^{n} = &2 R_2^{11/2} t_3 \big(t_3^2 \big(40 t_0^2-68 \big(w t_1-t_4\big) t_0+30 \big(t_4-w t_1\big){}^2\big)-29 \big(t_4-w t_1\big) 
\\& \big(4 \big(t_4 - 2 w t_1-w t_2\big) t_0^2+2 \big(w t_1-t_4\big) \big(w t_1+w t_2-2 t_4\big) t_0+t_4 \big(t_4-w t_1\big){}^2\big)\big) R_1^{3/2} 
\\&+2 R_2^{9/2} t_3 \big(6 t_3^4-4 \big(40 t_0^2-2 \big(11 w t_1+6 w t_2-15 t_4\big) t_0-\big(w t_1-t_4\big) \big(3 w t_1-9 w t_2 
\\&+5 t_4\big)\big) t_3^2-29 \big(-2 \big(6 w^2 t_1^2+4 w \big(2 w t_2-t_4\big) t_1+w^2 t_2^2-t_4^2-6 w t_2 t_4\big) t_0^2
\\& +\big(w t_1-t_4\big) \big(w^2 t_1^2+w \big(3 w t_2-7 t_4\big) t_1+w^2 t_2^2-t_4^2-9 w t_2 t_4\big) t_0
\\&+2 w t_2 t_4 \big(t_4-w t_1\big){}^2\big)\big) R_1^{5/2} -2 R_2^{7/2} t_3 \big(12 t_3^4+2 \big(-3 \big(t_1^2-4 t_2 t_1+t_2^2\big) w^2
\\&-14 t_0 \big(t_1-t_2\big) w +6 \big(t_2-t_1\big) t_4 w-120 t_0^2+10 t_4^2-52 t_0 t_4\big) t_3^2+29 \big(8 \big(w^2 t_1^2
\\&+w \big(3 w t_2+t_4\big) t_1+w^2 t_2^2-t_4^2 -w t_2 t_4\big) t_0^2+4 t_4 \big(w^2 t_1^2+2 w \big(2 w t_2+t_4\big) t_1+w^2 t_2^2
\\&-2 t_4^2-2 w t_2 t_4\big) t_0+\big(w t_1-t_4\big) t_4  \big(w t_2+t_4\big) \big(w t_1-w t_2+2 t_4\big)\big)\big) R_1^{7/2}+2 R_2^{5/2} t_3
\\& \big(6 t_3^4-4 \big(40 t_0^2+2 \big(6 w t_1+11 w t_2+15 t_4\big) t_0 +\big(w t_2+t_4\big) \big(9 w t_1-3 w t_2+5 t_4\big)\big) t_3^2
\\&+29 \big(2 \big(w^2 t_1^2+2 w \big(4 w t_2+3 t_4\big) t_1+6 w^2 t_2^2-t_4^2 +4 w t_2 t_4\big) t_0^2+\big(w t_2+t_4\big) \big(w^2 t_1^2
\\&+3 w \big(w t_2+3 t_4\big) t_1+w^2 t_2^2-t_4^2+7 w t_2 t_4\big) t_0 +2 w t_1 t_4 \big(w t_2+t_4\big){}^2\big)\big) R_1^{9/2}
\\&+2 R_2^{3/2} t_3 \big(t_3^2 \big(40 t_0^2+68 \big(w t_2+t_4\big) t_0+30 \big(w t_2+t_4\big){}^2\big) -29 \big(w t_2+t_4\big) \big(4 \big(w t_1 
\\&+2 w t_2+t_4\big) t_0^2+2 \big(w t_2+t_4\big) \big(w t_1+w t_2+2 t_4\big) t_0+t_4 \big(w t_2+t_4\big){}^2\big)\big) R_1^{11/2}
\\&+58 \sqrt{R_2} t_0 t_3 \big(w t_2+t_4\big){}^2 \big(2 t_0+w t_2+t_4\big) R_1^{13/2}-3 t_0 \big(w t_2+t_4\big){}^3  \big(2 t_0+w t_2
\\&+t_4\big) R_1^7+R_2 \big(w t_2+t_4\big) \big(3 \big(w t_2+t_4\big) \big(2 \big(3 w t_1+4 w t_2+t_4\big) t_0^2 +\big(w t_2+t_4\big) \big(3 w t_1
\\&+2 w t_2+3 t_4\big) t_0+t_4 \big(w t_2+t_4\big){}^2\big)-2 t_3^2 \big(44 t_0^2+46 \big(w t_2+t_4\big) t_0 +9 \big(w t_2
\\&+t_4\big){}^2\big)\big) R_1^6+R_2^2 \big(6 \big(w t_2+t_4\big) t_3^4+2 \big(44 \big(w t_1+4 w t_2+3 t_4\big) t_0^2 +2 \big(w t_2+t_4\big) \big(35 w t_1
\\&+34 w t_2+43 t_4\big) t_0+\big(w t_2+t_4\big){}^2 \big(15 w t_1-18 w t_2+13 t_4\big)\big) t_3^2 +3 \big(w t_2+t_4\big) \big(
\\&-6 \big(w^2 t_1^2+2 w \big(2 w t_2+t_4\big) t_1+2 w^2 t_2^2-t_4^2\big) t_0^2-\big(w t_2+t_4\big) \big(3 w^2 t_1^2 +w \big(5 w t_2
\\&+11 t_4\big) t_1+w^2 t_2^2-5 t_4^2+5 w t_2 t_4\big) t_0-\big(3 w t_1-t_4\big) t_4 \big(w t_2+t_4\big){}^2\big)\big) R_1^5 
\\&-R_2^3 \big(6 \big(w t_1+2 w t_2+t_4\big) t_3^4+2 \big(88 \big(2 w t_1+3 w t_2+t_4\big) t_0^2+2 \big(12 w^2 t_1^2 +5 w \big(7 w t_2
\\&+11 t_4\big) t_1-w^2 t_2^2+20 t_4^2+51 w t_2 t_4\big) t_0+\big(w t_2+t_4\big) \big(3 w^2 t_1^2 +2 w \big(5 t_4-27 w t_2\big) t_1
\\&-3 w^2 t_2^2+4 t_4^2+24 w t_2 t_4\big)\big) t_3^2-3 \big(2 \big(w^3 t_1^3+3 w^2 \big(4 w t_2+3 t_4\big) t_1^2 +3 w \big(6 w^2 t_2^2
\\& +4 w t_4 t_2-t_4^2\big) t_1+4 w^3 t_2^3-3 t_4^3-12 w t_2 t_4^2-6 w^2 t_2^2 t_4\big) t_0^2+\big(w t_2+t_4\big) \big(w^3 t_1^3
\\&+3 w^2 \big(w t_2+4 t_4\big) t_1^2+w \big(w^2 t_2^2+20 w t_4 t_2-2 t_4^2\big) t_1+t_4 \big(3 w^2 t_2^2-15 w t_4 t_2-7 t_4^2\big)\big) t_0 
\\&+\big(w t_1-t_4\big) t_4 \big(w t_2+t_4\big){}^2 \big(3 w t_1-w t_2+2 t_4\big)\big)\big) R_1^4+R_2^4 \big(6 \big(2 w t_1+w t_2-t_4\big) t_3^4 
\\&-2 \big(-88 \big(3 w t_1+2 w t_2-t_4\big) t_0^2-2 \big(w^2 t_1^2+w \big(51 t_4-35 w t_2\big) t_1-12 w^2 t_2^2-20 t_4^2 
\\&+55 w t_2 t_4\big) t_0+\big(w t_1-t_4\big) \big(3 w^2 t_1^2+6 w \big(9 w t_2+4 t_4\big) t_1-3 w^2 t_2^2-4 t_4^2+10 w t_2 t_4\big)\big) 
\\&t_3^2 -3 t_4 \big(t_4-w t_1\big){}^2 \big(w t_2+t_4\big) \big(w t_1-3 w t_2+2 t_4\big)-6 t_0^2 \big(4 w^3 t_1^3+6 w^2 \big(3 w t_2+t_4\big) t_1^2 
\\&+12 w \big(w^2 t_2^2-w t_4 t_2-t_4^2\big) t_1+w^3 t_2^3+3 t_4^3-3 w t_2 t_4^2-9 w^2 t_2^2 t_4\big)+3 t_0 \big(w^3 \big(w t_2
\\&-3 t_4\big) t_1^3 +3 w^2 \big(w^2 t_2^2-7 w t_4 t_2-4 t_4^2\big) t_1^2+w \big(w^3 t_2^3-15 w^2 t_4 t_2^2+18 w t_4^2 t_2+22 t_4^3\big) t_1 
\\&+t_4 \big(-w^3 t_2^3+12 w^2 t_4 t_2^2+2 w t_4^2 t_2-7 t_4^3\big)\big)\big) R_1^3+R_2^5 \big(6 \big(t_4-w t_1\big) t_3^4+2 \big(-44 \big(
\\&4 w t_1 +w t_2-3 t_4\big) t_0^2+2 \big(w t_1-t_4\big) \big(34 w t_1+35 w t_2-43 t_4\big) t_0+\big(t_4-w t_1\big){}^2 \big(18 w t_1 
\\&-15 w t_2+13 t_4\big)\big) t_3^2+3 \big(t_4-w t_1\big) \big(-6 \big(2 t_1^2 w^2+t_2^2 w^2+4 t_1 t_2 w^2-2 t_2 t_4 w-t_4^2\big) t_0^2 
\\&+\big(w t_1-t_4\big) \big(w^2 t_1^2+5 w \big(w t_2-t_4\big) t_1+3 w^2 t_2^2-5 t_4^2-11 w t_2 t_4\big) t_0+t_4 \big(t_4-w t_1\big){}^2 
\\&\big(3 w t_2+t_4\big)\big)\big) R_1^2+R_2^6 \big(t_4-w t_1\big) \big(3 \big(t_4-w t_1\big) \big(\big(-8 w t_1-6 w t_2+2 t_4\big) t_0^2+\big(w t_1
\\&-t_4\big) \big(2 w t_1+3 w t_2-3 t_4\big) t_0+t_4 \big(t_4-w t_1\big){}^2\big)-2 t_3^2 \big(44 t_0^2-46 \big(w t_1-t_4\big) t_0+9 \big(t_4
\\&-w t_1\big){}^2\big)\big) R_1 +3 R_2^7 t_0 \big(w t_1-t_4\big){}^3 \big(2 t_0-w t_1+t_4\big)+58 \sqrt{R_1} R_2^{13/2} t_0 t_3 \big(t_4
\\&-w t_1\big){}^2 \big(2 t_0-w t_1+t_4\big),
\end{align*}
\endgroup
\begin{align*}
\tilde{a}_{6}^{d} = 60 R_1^3 R_2^3 \big(R_2 \big(t_4-t_1 w\big)+R_1 \big(t_2 w+t_4\big)+2 \sqrt{R_1} \sqrt{R_2} t_3\big){}^4.
\end{align*}
\begin{align*}
\tilde{a}_{7} = \frac{\tilde{a}_{7}^{n}}{\tilde{a}_{7}^{d}}, \quad \text{where}
\end{align*}
\begingroup
\allowdisplaybreaks
\begin{align*}
\tilde{a}_{7}^{n} = &2 R_2^{9/2} t_3 \big(t_3^2 \big(-40 t_0-3 w t_1+3 t_4\big)-4 \big(2 t_4^3-2 \big(t_0+2 w t_1\big) t_4^2+\big(-10 t_0^2+3 w \big(t_2
\\&-t_1\big) t_0 +2 w^2 t_1^2\big) t_4+w t_0 \big(5 w t_1^2+20 t_0 t_1-3 w t_2 t_1+10 t_0 t_2\big)\big)\big) R_1^{3/2}+2 R_2^{7/2} t_3 \big(
\\&t_3^2 \big(24 t_0-12 w t_1 -9 w t_2+19 t_4\big)-4 \big(2 t_4^3-4 w t_1 t_4^2+\big(-20 t_0^2-7 w t_1 t_0
\\&+3 w t_2 t_0+2 w^2 t_1^2\big) t_4 +w t_0 \big(2 w t_1^2-3 w t_2 t_1+5 w t_2^2-20 t_0 t_2\big)\big)\big) R_1^{5/2}-2 R_2^{5/2} t_3 \big(
\\&t_3^2 \big(24 t_0+9 w t_1+12 w t_2 +19 t_4\big)-4 \big(2 t_4^3+4 w t_2 t_4^2+\big(-20 t_0^2-3 w t_1 t_0+7 w t_2 t_0
\\&+2 w^2 t_2^2\big) t_4+w t_0 \big(5 w t_1^2 +20 t_0 t_1-3 w t_2 t_1+2 w t_2^2\big)\big)\big) R_1^{7/2}-2 R_2^{3/2} t_3 \big(-8 t_4^3
\\&+8 \big(t_0-2 w t_2\big) t_4^2+4 \big(10 t_0^2 +3 w \big(t_1-t_2\big) t_0-2 w^2 t_2^2\big) t_4+4 w t_0 \big(-5 w t_2^2+20 t_0 t_2
\\&+3 w t_1 t_2+10 t_0 t_1\big)+t_3^2 \big(-40 t_0 +3 w t_2+3 t_4\big)\big) R_1^{9/2}-16 \sqrt{R_2} t_0 t_3 \big(w t_2+t_4\big) \big(
\\&-5 t_0+w t_2+t_4\big) R_1^{11/2}-64 t_0^2 \big(w t_2+t_4\big){}^2 R_1^6 +R_2 \big(\big(-64 t_0^2-76 w t_2 t_0-76 t_4 t_0
\\&+9 w^2 t_2^2+9 t_4^2+18 w t_2 t_4\big) t_3^2+64 t_0 \big(w t_2+t_4\big) \big(t_4^2 +w t_2 t_4+2 w t_0 \big(t_1+t_2\big)\big)\big) R_1^5
\\&-2 R_2^2 \big(12 t_3^4-\big(64 t_0^2+54 w t_1 t_0-14 w t_2 t_0-36 t_4 t_0+3 w^2 t_2^2 +5 t_4^2+8 w t_2 t_4\big) t_3^2
\\&-32 t_0 \big(t_4-w t_1\big) \big(2 t_4^2+\big(3 t_0+2 w t_2\big) t_4+w t_0 \big(t_1+4 t_2\big)\big)\big) R_1^4 -w R_2^3 \big(t_1+t_2\big) \big(
\\&\big(4 t_0+9 w t_1-9 w t_2+18 t_4\big) t_3^2+64 t_0 \big(2 t_0+t_4\big) \big(w \big(t_2-t_1\big)+2 t_4\big)\big) R_1^3 +2 R_2^4 \big(12 t_3^4
\\&+\big(-64 t_0^2-14 w t_1 t_0+54 w t_2 t_0+36 t_4 t_0-3 w^2 t_1^2-5 t_4^2+8 w t_1 t_4\big) t_3^2 -32 t_0 \big(w t_2
\\&+t_4\big) \big(2 t_4^2+\big(3 t_0-2 w t_1\big) t_4-w t_0 \big(4 t_1+t_2\big)\big)\big) R_1^2+R_2^5 \big(\big(64 t_0^2-76 w t_1 t_0 +76 t_4 t_0
\\&-9 w^2 t_1^2-9 t_4^2+18 w t_1 t_4\big) t_3^2+64 t_0 \big(t_4-w t_1\big) \big(-t_4^2+w t_1 t_4+2 w t_0 \big(t_1+t_2\big)\big)\big) R_1 
\\&+64 R_2^6 t_0^2 \big(t_4-w t_1\big){}^2+16 \sqrt{R_1} R_2^{11/2} t_0 t_3 \big(t_4-w t_1\big) \big(-5 t_0-w t_1+t_4\big),
\end{align*}
\endgroup
\begin{align*}
\tilde{a}_{7}^{d} = 24 R_1^3 R_2^3 \big(2 \sqrt{R_1} \sqrt{R_2} t_3+R_2 \big(t_4-w t_1\big)+R_1 \big(w t_2+t_4\big)\big){}^3.
\end{align*}
\begin{align*}
\tilde{a}_{8} = \frac{\tilde{a}_{8}^{n}}{\tilde{a}_{8}^{d}}, \quad \text{where}
\end{align*}
\begingroup
\allowdisplaybreaks
\begin{align*}
\tilde{a}_{8}^{n} = &-4 R_2^{9/2} t_3 \big(7 t_4^3+7 \big(5 t_0-2 w t_1\big) t_4^2+\big(-78 t_0^2+11 w t_1 t_0+53 w t_2 t_0+7 w^2 t_1^2\big) t_4 
\\&-w t_0 \big(46 w t_1^2+56 t_0 t_1+53 w t_2 t_1+134 t_0 t_2\big)+3 t_3^2 \big(-60 t_0-7 w t_1+7 t_4\big)\big) R_1^{3/2} 
\\&+4 R_2^{7/2} t_3 \big(7 t_4^3+7 \big(4 t_0-w \big(t_1+t_2\big)\big) t_4^2-\big(-7 t_1 t_2 w^2+t_0 \big(123 t_1+95 t_2\big) w
\\&+212 t_0^2\big) t_4 +w t_0 \big(7 w t_1^2+156 t_0 t_1-53 w t_2 t_1-60 w t_2^2-56 t_0 t_2\big)+3 t_3^2 \big(-60 t_0
\\&-14 w t_1+3 w t_2 +17 t_4\big)\big) R_1^{5/2}+4 R_2^{5/2} t_3 \big(7 t_4^3+7 \big(4 t_0+w \big(t_1+t_2\big)\big) t_4^2+\big(7 t_1 t_2 w^2
\\&+t_0 \big(95 t_1+123 t_2\big) w -212 t_0^2\big) t_4+w t_0 \big(-60 w t_1^2+56 t_0 t_1-53 w t_2 t_1+7 w t_2^2
\\&-156 t_0 t_2\big)+3 t_3^2 \big(-60 t_0 -3 w t_1+14 w t_2+17 t_4\big)\big) R_1^{7/2}-4 R_2^{3/2} t_3 \big(7 t_4^3+7 \big(5 t_0
\\&+2 w t_2\big) t_4^2-\big(78 t_0^2+53 w t_1 t_0 +11 w t_2 t_0-7 w^2 t_2^2\big) t_4+w t_0 \big(-46 w t_2^2+56 t_0 t_2
\\&-53 w t_1 t_2+134 t_0 t_1\big)+3 t_3^2 \big(-60 t_0 +7 w t_2+7 t_4\big)\big) R_1^{9/2}+4 \sqrt{R_2} t_0 t_3 \big(w t_2+t_4\big) 
\\&\big(134 t_0+7 w t_2+7 t_4\big) R_1^{11/2}+3 t_0 \big(w t_2+t_4\big){}^2  \big(162 t_0+w t_2+t_4\big) R_1^6-R_2 \big(\big(920 t_0^2
\\&-368 w t_2 t_0-368 t_4 t_0+27 w^2 t_2^2+27 t_4^2+54 w t_2 t_4\big) t_3^2 +3 \big(w t_2+t_4\big) \big(t_4^3+2 \big(2 t_0
\\&+w t_2\big) t_4^2+\big(-316 t_0^2+2 w \big(t_1+3 t_2\big) t_0+w^2 t_2^2\big) t_4 +2 w t_0 \big(w t_2^2+4 t_0 t_2+w t_1 t_2
\\&+162 t_0 t_1\big)\big)\big) R_1^5+R_2^2 \big(-84 t_3^4+2 \big(3 t_2 \big(t_2-2 t_1\big) w^2-80 t_0^2 +5 t_4^2+t_0 \big(66 w t_1
\\&+46 w t_2\big)+\big(-60 t_0-6 w t_1+8 w t_2\big) t_4\big) t_3^2+3 \big(-\big(t_0-2 w t_1\big) t_4^3 +\big(4 t_1 t_2 w^2
\\&+9 t_0 t_1 w+6 t_0 t_2 w-162 t_0^2\big) t_4^2+w \big(w t_0 t_1^2+2 \big(-154 t_0^2+6 w t_2 t_0+w^2 t_2^2\big) t_1 
\\&+8 t_0 t_2 \big(w t_2-79 t_0\big)\big) t_4+w^2 t_0 \big(\big(162 t_0+w t_2\big) t_1^2+t_2 \big(16 t_0+3 w t_2\big) t_1+t_2^2 \big(w t_2 
\\&-308 t_0\big)\big)\big)\big) R_1^4+R_2^3 \big(288 t_3^4+\big(3 \big(13 t_1^2-4 t_2 t_1+13 t_2^2\big) w^2+328 t_0 \big(t_1-t_2\big) w
\\&+2160 t_0^2 +154 t_4^2-2 \big(248 t_0+51 w \big(t_1-t_2\big)\big) t_4\big) t_3^2-3 \big(-2 t_4^4+\big(w \big(t_1-t_2\big)
\\&-8 t_0\big) t_4^3+\big(\big(t_1^2+t_2^2\big) w^2 +8 t_0 \big(t_1-t_2\big) w+632 t_0^2\big) t_4^2+w \big(w \big(4 t_0+w t_2\big) t_1^2
\\&-\big(632 t_0^2-16 w t_2 t_0+w^2 t_2^2\big) t_1 +4 t_0 t_2 \big(158 t_0+w t_2\big)\big) t_4+8 w^2 t_0^2 \big(t_1^2-77 t_2 t_1
\\&+t_2^2\big)\big)\big) R_1^3-R_2^4 \big(84 t_3^4-2 \big(3 t_1 \big(t_1-2 t_2\big) w^2 -2 t_0 \big(23 t_1+33 t_2\big) w-80 t_0^2+5 t_4^2
\\&+\big(-60 t_0-8 w t_1+6 w t_2\big) t_4\big) t_3^2+3 \big(\big(t_0+2 w t_2\big) t_4^3 +\big(-4 t_1 t_2 w^2+6 t_0 t_1 w
\\&+9 t_0 t_2 w+162 t_0^2\big) t_4^2-w \big(-2 w \big(w t_2-4 t_0\big) t_1^2+4 t_0 \big(158 t_0 +3 w t_2\big) t_1+t_0 t_2 \big(308 t_0
\\&+w t_2\big)\big) t_4+w^2 t_0 \big(w t_1^3+\big(308 t_0+3 w t_2\big) t_1^2+t_2 \big(w t_2-16 t_0\big) t_1 -162 t_0 t_2^2\big)\big)\big) R_1^2
\\&-R_2^5 \big(\big(920 t_0^2+368 w t_1 t_0-368 t_4 t_0+27 w^2 t_1^2+27 t_4^2-54 w t_1 t_4\big) t_3^2 +3 \big(t_4-w t_1\big) 
\\&\big(t_4^3+\big(4 t_0-2 w t_1\big) t_4^2+\big(-316 t_0^2-2 w \big(3 t_1+t_2\big) t_0+w^2 t_1^2\big) t_4+2 w t_0 \big(w t_1^2 -4 t_0 t_1
\\&+w t_2 t_1-162 t_0 t_2\big)\big)\big) R_1+3 R_2^6 t_0 \big(t_4-w t_1\big){}^2 \big(162 t_0-w t_1+t_4\big) 
\\&+4 \sqrt{R_1} R_2^{11/2} t_0 t_3 \big(t_4-w t_1\big) \big(134 t_0-7 w t_1+7 t_4\big),
\end{align*}
\endgroup
\begin{align*}
\tilde{a}_{8}^{d} = 240 R_1^3 R_2^3 \big(2 \sqrt{R_1} \sqrt{R_2} t_3+R_2 \big(t_4-w t_1\big)+R_1 \big(w t_2+t_4\big)\big){}^3.
\end{align*}
\begin{align*}
\tilde{a}_{9} = 0.
\end{align*}

\section{Coefficients in \texorpdfstring{$E$}{\textit{E}} and \texorpdfstring{$F$}{\textit{F}}} \label{sec:appendix-EF-coeffs}

\begin{align*}
e_{1} = \frac{e_{1}^{n}}{e_{1}^{d}}, \quad \text{where}
\end{align*}
\begingroup
\allowdisplaybreaks
\begin{align*}
e_{1}^{n} = &R_1^{3/2} R_2^{3/2} t_3 \big(-3 t_1 w+3 t_2 w-8 t_0+2 t_4\big)+\sqrt{R_1} R_2^{5/2} t_3 \big(-3 t_1 w+4 t_0+3 t_4\big) 
\\&+10 R_2^3 t_0 \big(t_1 w-t_4\big)-2 R_1 R_2^2 \big(5 t_0 \big(2 t_1 w+t_2 w-t_4\big)+5 t_1 t_4 w+3 t_3^2-5 t_4^2\big) 
\\&+2 R_1^2 R_2 \big(5 t_4 \big(t_2 w+t_4\big)+5 t_0 \big(t_1 w+2 t_2 w+t_4\big)-3 t_3^2\big)-10 R_1^3 t_0 \big(t_2 w+t_4\big) 
\\&+\sqrt{R_2} R_1^{5/2} t_3 \big(3 \big(t_2 w+t_4\big)+4 t_0\big),
\end{align*}
\endgroup
\begin{align*}
e_{1}^{d} = 12 R_1 R_2 \big(R_2 \big(t_4-t_1 w\big)+R_1 \big(t_2 w+t_4\big)+2 \sqrt{R_1} \sqrt{R_2} t_3\big){}^2.
\end{align*}
\begin{align*}
e_{2} = \frac{\sqrt{R_2} R_1^{3/2} t_3-2 R_1^2 t_0+2 R_2 R_1 \big(2 t_0+t_4\big)-2 R_2^2 t_0+\sqrt{R_1} R_2^{3/2} t_3}{4 R_1 R_2 \big(R_2 \big(t_4-t_1 w\big)+R_1 \big(t_2 w+t_4\big)+2 \sqrt{R_1} \sqrt{R_2} t_3\big)}.
\end{align*}
\begin{align*}
e_{3} = \frac{\big(R_1-R_2\big) \big(4 \big(R_1+R_2\big) t_0-\sqrt{R_1} \sqrt{R_2} t_3\big)}{4 R_1 R_2 \big(R_2 \big(t_4-t_1 w\big)+R_1 \big(t_2 w+t_4\big)+2 \sqrt{R_1} \sqrt{R_2} t_3\big)}.
\end{align*}

\begin{align*}
f_{1} = \frac{f_{1}^{n}}{f_{1}^{d}}, \quad \text{where}
\end{align*}
\begingroup
\allowdisplaybreaks
\begin{align*}
f_{1}^{n} = &R_2^{9/2} t_3 \big(-1088 \big(4 w t_1+w t_2-3 t_4\big) t_0^2+8 \big(\big(151 w t_1+98 w t_2-325 t_4\big) \big(w t_1
\\&-t_4\big)-60 t_3^2\big) t_0 +\big(w t_1-t_4\big) \big(348 t_3^2+\big(w t_1-t_4\big) \big(90 w t_1-75 w t_2+499 t_4\big)\big)\big) R_1^{3/2}
\\&+R_2^{7/2} t_3 \big(-454 t_4^3 +\big(-1936 t_0+574 w t_1+484 w t_2\big) t_4^2-\big(\big(105 t_1^2+214 t_2 t_1
\\&+15 t_2^2\big) w^2-32 t_0 \big(138 t_1 +85 t_2\big) w+2176 t_0^2\big) t_4-w \big(15 t_1 \big(t_1^2+18 t_2 t_1-t_2^2\big) w^2
\\&+8 t_0 \big(53 t_1^2+98 t_2 t_1+15 t_2^2\big) w -2176 t_0^2 \big(3 t_1+2 t_2\big)\big)+12 t_3^2 \big(40 t_0+2 w t_1-9 w t_2
\\&+29 t_4\big)\big) R_1^{5/2}+R_2^{5/2} t_3 \big(-454 t_4^3 -2 \big(968 t_0+242 w t_1+287 w t_2\big) t_4^2-\big(\big(15 t_1^2
\\&+214 t_2 t_1+105 t_2^2\big) w^2+32 t_0 \big(85 t_1+138 t_2\big) w +2176 t_0^2\big) t_4-w \big(-15 t_2 \big(-t_1^2
\\&+18 t_2 t_1+t_2^2\big) w^2+8 t_0 \big(15 t_1^2+98 t_2 t_1+53 t_2^2\big) w +2176 t_0^2 \big(2 t_1+3 t_2\big)\big)+12 t_3^2 \big(40 t_0
\\&+9 w t_1-2 w t_2+29 t_4\big)\big) R_1^{7/2}+R_2^{3/2} t_3 \big(1088 \big(w t_1 +4 w t_2+3 t_4\big) t_0^2+8 \big(\big(w t_2+t_4\big) 
\\&\big(98 w t_1+151 w t_2+325 t_4\big)-60 t_3^2\big) t_0+\big(w t_2 +t_4\big) \big(\big(w t_2+t_4\big) \big(75 w t_1-90 w t_2 
\\&+499 t_4\big)-348 t_3^2\big)\big) R_1^{9/2}-\sqrt{R_2} t_3 \big(w t_2+t_4\big) \big(1088 t_0^2+664 \big(w t_2+t_4\big) t_0+45 \big(w t_2
\\&+t_4\big){}^2\big) R_1^{11/2}+296 t_0 \big(w t_2+t_4\big){}^2 \big(2 t_0+w t_2+t_4\big) R_1^6 +2 R_2 \big(t_3^2 \big(800 t_0^2+928 \big(w t_2
\\&+t_4\big) t_0+243 \big(w t_2+t_4\big){}^2\big)-148 \big(w t_2+t_4\big) \big(4 \big(w t_1+2 w t_2 +t_4\big) t_0^2+2 \big(w t_2+t_4\big) 
\\&\big(w \big(t_1+t_2\big)+2 t_4\big) t_0+t_4 \big(w t_2+t_4\big){}^2\big)\big) R_1^5+4 R_2^2 \big(18 t_3^4-\big(1600 t_0^2 +8 \big(33 w t_1
\\&+83 w t_2+150 t_4\big) t_0+\big(w t_2+t_4\big) \big(111 w t_1-153 w t_2+200 t_4\big)\big) t_3^2+74 \big(2 \big(\big(t_1^2 
\\&+8 t_2 t_1+6 t_2^2\big) w^2+2 \big(3 t_1+2 t_2\big) t_4 w-t_4^2\big) t_0^2+\big(w t_2+t_4\big) \big(\big(t_1^2+3 t_2 t_1+t_2^2\big) w^2
\\&+\big(9 t_1 +7 t_2\big) t_4 w-t_4^2\big) t_0+2 w t_1 t_4 \big(w t_2+t_4\big){}^2\big)\big) R_1^4-2 R_2^3 \big(72 t_3^4+\big(3 \big(7 t_1^2
\\&+204 t_2 t_1+7 t_2^2\big) w^2  +42 \big(t_1-t_2\big) t_4 w-4800 t_0^2-314 t_4^2-128 t_0 \big(w t_1-w t_2
\\&+23 t_4\big)\big) t_3^2+148 \big(8 \big(\big(t_1^2+3 t_2 t_1+t_2^2\big) w^2+\big(t_1-t_2\big) t_4 w-t_4^2\big) t_0^2+4 t_4 \big(\big(t_1^2+4 t_2 t_1
\\&+t_2^2\big) w^2+2 \big(t_1-t_2\big) t_4 w-2 t_4^2\big) t_0 +\big(w t_1-t_4\big) t_4 \big(w t_2+t_4\big) \big(w t_1-w t_2+2 t_4\big)\big)\big) 
\\&R_1^3+4 R_2^4 \big(18 t_3^4+\big(-1600 t_0^2+8 \big(83 w t_1 +33 w t_2-150 t_4\big) t_0+\big(w t_1-t_4\big) \big(153 w t_1
\\&-111 w t_2+200 t_4\big)\big) t_3^2-74 \big(-2 \big(\big(6 t_1^2+8 t_2 t_1 +t_2^2\big) w^2-2 \big(2 t_1+3 t_2\big) t_4 w-t_4^2\big) t_0^2
\\&+\big(w t_1-t_4\big) \big(\big(t_1^2+3 t_2 t_1+t_2^2\big) w^2-\big(7 t_1+9 t_2\big) t_4 w -t_4^2\big) t_0+2 w t_2 t_4 \big(t_4-w t_1\big){}^2\big)\big) 
\\&R_1^2+2 R_2^5 \big(t_3^2 \big(800 t_0^2+928 \big(t_4-w t_1\big) t_0+243 \big(t_4-w t_1\big){}^2\big) -148 \big(t_4-w t_1\big) \big(4 \big(t_4
\\&-w \big(2 t_1+t_2\big)\big) t_0^2+2 \big(w t_1-t_4\big) \big(w \big(t_1+t_2\big)-2 t_4\big) t_0+t_4 \big(t_4 -w t_1\big){}^2\big)\big) R_1
\\&+296 R_2^6 t_0 \big(t_4-w t_1\big){}^2 \big(2 t_0-w t_1+t_4\big)+\sqrt{R_1} R_2^{11/2} t_3 \big(w t_1-t_4\big) \big(1088 t_0^2 
\\&+664 \big(t_4-w t_1\big) t_0+45 \big(t_4-w t_1\big){}^2\big),
\end{align*}
\endgroup
\begin{align*}
f_{1}^{d} = 960 R_1^2 R_2^2 \big(2 \sqrt{R_1} \sqrt{R_2} t_3+R_2 \big(t_4-w t_1\big)+R_1 \big(w t_2+t_4\big)\big){}^4.
\end{align*}
\begin{align*}
f_{2} = \frac{f_{2}^{n}}{f_{2}^{d}}, \quad \text{where}
\end{align*}
\begingroup
\allowdisplaybreaks
\begin{align*}
f_{2}^{n} = &2 R_2^{7/2} t_3 \big(-128 t_0^2+\big(-58 w t_1+98 w t_2+92 t_4\big) t_0+30 t_3^2-3 \big(w t_1-t_4\big) \big(w t_1
\\&-2 w t_2 +3 t_4\big)\big) R_1^{3/2}-w R_2^{5/2} \big(t_1+t_2\big) t_3 \big(44 t_0+3 w t_1-3 w t_2+70 t_4\big) R_1^{5/2}
\\&-2 R_2^{3/2} t_3 \big(-128 t_0^2 +\big(-98 w t_1+58 w t_2+92 t_4\big) t_0+30 t_3^2+3 \big(w t_2+t_4\big) \big(2 w t_1
\\&-w t_2+3 t_4\big)\big) R_1^{7/2}+\sqrt{R_2} t_3 \big( -128 t_0^2-36 \big(w t_2+t_4\big) t_0+15 \big(w t_2+t_4\big){}^2\big) R_1^{9/2}
\\&+32 t_0 \big(w t_2+t_4\big) \big(7 t_0-2 \big(w t_2+t_4\big)\big) R_1^5 -8 R_2 \big(28 \big(w t_1+2 w t_2+t_4\big) t_0^2-\big(29 t_3^2
\\&+2 \big(w t_2+t_4\big) \big(4 w t_1+5 w t_2-13 t_4\big)\big) t_0-\big(w t_2+t_4\big)  \big(5 t_4 \big(w t_2+t_4\big)-6 t_3^2\big)\big) R_1^4
\\&-4 R_2^2 \big(112 \big(t_4-w t_1\big) t_0^2+4 \big(t_2 \big(6 t_1+t_2\big) w^2-4 \big(2 t_1+t_2\big) t_4 w +9 t_4^2\big) t_0+10 \big(w t_1
\\&-t_4\big) t_4 \big(w t_2+t_4\big)+t_3^2 \big(30 t_0+9 w t_1+15 w t_2+28 t_4\big)\big) R_1^3+4 R_2^3 \big(112 \big(w t_2 +t_4\big) t_0^2
\\&+4 \big(t_1 \big(t_1+6 t_2\big) w^2+4 \big(t_1+2 t_2\big) t_4 w+9 t_4^2\big) t_0+10 \big(w t_1-t_4\big) t_4 \big(w t_2+t_4\big)
\\&+t_3^2 \big(30 t_0 -15 w t_1-9 w t_2+28 t_4\big)\big) R_1^2+8 R_2^4 \big(-28 \big(2 w t_1+w t_2-t_4\big) t_0^2-\big(29 t_3^2
\\&+2 \big(w t_1-t_4\big) \big(5 w t_1 +4 w t_2+13 t_4\big)\big) t_0-\big(w t_1-t_4\big) \big(6 t_3^2+5 \big(w t_1-t_4\big) t_4\big)\big) R_1
\\&+32 R_2^5 t_0 \big(w t_1-t_4\big) \big(7 t_0+2 w t_1 -2 t_4\big)+\sqrt{R_1} R_2^{9/2} t_3 \big(128 t_0^2+36 \big(t_4-w t_1\big) t_0
\\&-15 \big(t_4-w t_1\big){}^2\big),
\end{align*}
\endgroup
\begin{align*}
f_{2}^{d} = 96 R_1^2 R_2^2 \big(2 \sqrt{R_1} \sqrt{R_2} t_3+R_2 \big(t_4-w t_1\big)+R_1 \big(w t_2+t_4\big)\big){}^3.
\end{align*}
\begin{align*}
f_{3} = \frac{f_{3}^{n}}{f_{3}^{d}}, \quad \text{where}
\end{align*}
\begingroup
\allowdisplaybreaks
\begin{align*}
f_{3}^{n} = &R_2^{5/2} R_1^{3/2} t_3 \big(6 t_1 w+t_2 w+8 t_0+3 t_4\big)+R_2^{3/2} R_1^{5/2} t_3 \big(-\big(t_1+6 t_2\big) w+8 t_0+3 t_4\big) 
\\&-\sqrt{R_2} R_1^{7/2} t_3 \big(3 \big(t_2 w+t_4\big)+8 t_0\big)+8 R_1^4 t_0 \big(t_2 w+2 t_0+t_4\big)+2 R_2 R_1^3 \big(t_3^2 
\\&-4 \big(2 t_0+t_4\big) \big(t_2 w+4 t_0+t_4\big)\big)+4 R_2^2 R_1^2 \big(t_0 \big(-2 t_1 w+2 t_2 w+20 t_4\big)+24 t_0^2
\\&-t_3^2+4 t_4^2\big)+2 R_2^3 R_1 \big(t_3^2 -4 \big(2 t_0+t_4\big) \big(-t_1 w+4 t_0+t_4\big)\big)+8 R_2^4 t_0 \big(-t_1 w
\\&+2 t_0+t_4\big)+\sqrt{R_1} R_2^{7/2} t_3 \big(3 t_1 w -8 t_0-3 t_4\big),
\end{align*}
\endgroup
\begin{align*}
f_{3}^{d} = 64 R_1^2 R_2^2 \big(R_2 \big(t_4-t_1 w\big)+R_1 \big(t_2 w+t_4\big)+2 \sqrt{R_1} \sqrt{R_2} t_3\big){}^2.
\end{align*}
\begin{align*}
f_{4} = \frac{f_{4}^{n}}{f_{4}^{d}}, \quad \text{where}
\end{align*}
\begingroup
\allowdisplaybreaks
\begin{align*}
f_{4}^{n} = &R_2^{5/2} R_1^{3/2} t_3 \big(2 t_1 w+t_2 w-4 t_0-5 t_4\big)+R_2^{3/2} R_1^{5/2} t_3 \big(t_1 w+2 t_2 w+4 t_0+5 t_4\big) 
\\&+\sqrt{R_2} R_1^{7/2} t_3 \big(5 \big(t_2 w+t_4\big)-12 t_0\big)-16 R_1^4 t_0 \big(t_2 w+t_4\big)+2 R_2 R_1^3 \big(4 \big(2 t_0+t_4\big) 
\\&\big(t_2 w+t_4\big) +t_3^2\big)-2 R_2^3 R_1 \big(4 \big(2 t_0+t_4\big) \big(t_4-t_1 w\big)+t_3^2\big)+16 R_2^4 t_0 \big(t_4-t_1 w\big) 
\\&+\sqrt{R_1} R_2^{7/2} t_3 \big(5 t_1 w+12 t_0-5 t_4\big),
\end{align*}
\endgroup
\begin{align*}
f_{4}^{d} = 32 R_1^2 R_2^2 \big(R_2 \big(t_4-t_1 w\big)+R_1 \big(t_2 w+t_4\big)+2 \sqrt{R_1} \sqrt{R_2} t_3\big){}^2.
\end{align*}
\begin{align*}
f_{5} = \frac{f_{5}^{n}}{f_{5}^{d}}, \quad \text{where}
\end{align*}
\begingroup
\allowdisplaybreaks
\begin{align*}
f_{5}^{n} = &2 R_2^{7/2} t_3 \big(128 t_0^2+4 \big(25 w t_1+33 w t_2+24 t_4\big) t_0-36 t_3^2+7 \big(w t_1-t_4\big) \big(3 w \big(t_1+t_2\big)
\\&-4 t_4\big)\big)  R_1^{3/2}+R_2^{5/2} t_3 \big(3 \big(3 t_1^2-28 t_2 t_1+3 t_2^2\big) w^2+126 \big(t_2-t_1\big) t_4 w+1152 t_0^2
\\&+240 t_3^2+86 t_4^2 +8 t_0 \big(85 w \big(t_1-t_2\big)-202 t_4\big)\big) R_1^{5/2}+2 R_2^{3/2} t_3 \big(128 t_0^2-4 \big(33 w t_1
\\&+25 w t_2-24 t_4\big) t_0 -36 t_3^2+7 \big(w t_2+t_4\big) \big(3 w \big(t_1+t_2\big)+4 t_4\big)\big) R_1^{7/2}-\sqrt{R_2} t_3 
\\&\big(832 t_0^2-616 \big(w t_2+t_4\big) t_0 +51 \big(w t_2+t_4\big){}^2\big) R_1^{9/2}+32 t_0 \big(w t_2+t_4\big) \big(7 \big(w t_2+t_4\big)
\\&-10 t_0\big) R_1^5-2 R_2 \big(\big(45 \big(w t_2+t_4\big) -312 t_0\big) t_3^2+16 \big(2 \big(-5 w t_1+4 w t_2+9 t_4\big) t_0^2
\\&+\big(7 w t_1+2 w t_2-t_4\big) \big(w t_2+t_4\big) t_0+t_4 \big(w t_2 +t_4\big){}^2\big)\big) R_1^4+2 R_2^2 \big(3 \big(-104 t_0
\\&+3 w t_1+42 w t_2+31 t_4\big) t_3^2+16 \big(4 \big(2 w t_1+9 w t_2+7 t_4\big) t_0^2 +\big(t_2 \big(7 t_1-5 t_2\big) w^2
\\&+\big(11 t_1-9 t_2\big) t_4 w-8 t_4^2\big) t_0+t_4 \big(w t_1+t_4\big) \big(w t_2+t_4\big)\big)\big) R_1^3 +2 R_2^3 \big(-3 \big(104 t_0
\\&+42 w t_1+3 w t_2-31 t_4\big) t_3^2-16 \big(4 \big(9 w t_1+2 w t_2-7 t_4\big) t_0^2+\big(t_1 \big(5 t_1 -7 t_2\big) w^2
\\&+\big(11 t_2-9 t_1\big) t_4 w+8 t_4^2\big) t_0-\big(w t_1-t_4\big) \big(w t_2-t_4\big) t_4\big)\big) R_1^2+2 R_2^4 \big(t_3^2 \big(312 t_0 
\\&+45 w t_1-45 t_4\big)-16 \big(2 \big(-4 w t_1+5 w t_2+9 t_4\big) t_0^2+\big(w t_1-t_4\big) \big(2 w t_1+7 w t_2
\\&+t_4\big) t_0 +t_4 \big(t_4-w t_1\big){}^2\big)\big) R_1+32 R_2^5 t_0 \big(w t_1-t_4\big) \big(10 t_0+7 w t_1-7 t_4\big)
\\&-\sqrt{R_1} R_2^{9/2} t_3 \big(832 t_0^2 +616 \big(w t_1-t_4\big) t_0+51 \big(t_4-w t_1\big){}^2\big),
\end{align*}
\endgroup
\begin{align*}
f_{5}^{d} = 384 R_1^2 R_2^2 \big(2 \sqrt{R_1} \sqrt{R_2} t_3+R_2 \big(t_4-w t_1\big)+R_1 \big(w t_2+t_4\big)\big){}^3.
\end{align*}
\begin{align*}
f_{6} = \frac{f_{6}^{n}}{f_{6}^{d}}, \quad \text{where}
\end{align*}
\begingroup
\allowdisplaybreaks
\begin{align*}
f_{6}^{n} = &-2 R_2^{7/2} t_3 \big(-448 t_0^2+112 \big(w t_1-3 t_4\big) t_0+4 t_3^2+\big(w t_1-t_4\big) \big(3 w \big(t_1+t_2\big)
\\&+56 t_4\big)\big) R_1^{3/2} +R_2^{5/2} t_3 \big(3 \big(t_1^2+4 t_2 t_1+t_2^2\big) w^2+6 \big(t_1-t_2\big) t_4 w-1344 t_0^2+16 t_3^2
\\&-230 t_4^2+112 t_0 \big(w t_1 -w t_2-10 t_4\big)\big) R_1^{5/2}+2 R_2^{3/2} t_3 \big(448 t_0^2+112 \big(w t_2+3 t_4\big) t_0
\\&-4 t_3^2-\big(3 w \big(t_1+t_2\big)-56 t_4\big)  \big(w t_2+t_4\big)\big) R_1^{7/2}+\sqrt{R_2} t_3 \big(-224 t_0^2-112 \big(w t_2
\\&+t_4\big) t_0+3 \big(w t_2+t_4\big){}^2\big) R_1^{9/2}+8 t_0 \big(w t_2 +t_4\big) \big(2 t_0+w t_2+t_4\big) R_1^5+2 R_2 \big(t_3^2 \big(64 t_0
\\&+33 \big(w t_2+t_4\big)\big)-4 \big(2 \big(w t_1+4 w t_2+3 t_4\big) t_0^2 +\big(w t_2+t_4\big) \big(w t_1+2 w t_2+5 t_4\big) t_0
\\&+t_4 \big(w t_2+t_4\big){}^2\big)\big) R_1^4+2 R_2^2 \big(4 \big(4 \big(2 w t_1+3 w t_2+t_4\big) t_0^2 +\big(t_2 \big(t_1+t_2\big) w^2+\big(5 t_1
\\&+9 t_2\big) t_4 w+4 t_4^2\big) t_0+t_4 \big(w t_1+t_4\big) \big(w t_2+t_4\big)\big)-t_3^2 \big(64 t_0+w t_1 -30 w t_2
\\&+33 t_4\big)\big) R_1^3-2 R_2^3 \big(t_3^2 \big(64 t_0+30 w t_1-w t_2+33 t_4\big)-4 \big(4 \big(-3 w t_1-2 w t_2
\\&+t_4\big) t_0^2 +\big(t_1 \big(t_1+t_2\big) w^2-\big(9 t_1+5 t_2\big) t_4 w+4 t_4^2\big) t_0+\big(w t_1-t_4\big) \big(w t_2-t_4\big) t_4\big)\big) R_1^2
\\&+2 R_2^4 \big(t_3^2 \big( 64 t_0+33 \big(t_4-w t_1\big)\big)-4 \big(-2 \big(4 w t_1+w t_2-3 t_4\big) t_0^2+\big(w t_1-t_4\big) \big(2 w t_1
\\&+w t_2-5 t_4\big) t_0 +t_4 \big(t_4-w t_1\big){}^2\big)\big) R_1+8 R_2^5 t_0 \big(t_4-w t_1\big) \big(2 t_0-w t_1+t_4\big)
\\&+\sqrt{R_1} R_2^{9/2} t_3 \big(-224 t_0^2 +112 \big(w t_1-t_4\big) t_0+3 \big(t_4-w t_1\big){}^2\big),
\end{align*}
\endgroup
\begin{align*}
f_{6}^{d} = 384 R_1^2 R_2^2 \big(2 \sqrt{R_1} \sqrt{R_2} t_3+R_2 \big(t_4-w t_1\big)+R_1 \big(w t_2+t_4\big)\big){}^3.
\end{align*}

\end{appendices}
\printbibliography

\end{document}